\documentclass[12pt,a4paper]{article}
\title{Genus $n$ forms over Hyperbolic groups}
\usepackage[dvips]{graphicx}
\usepackage{psfrag, graphicx}
\usepackage{amsmath}
\usepackage{amsthm}
\usepackage{amsfonts}
\theoremstyle{plain}
\newtheorem{lemma}{Lemma}[section]
\newtheorem{thm}[lemma]{Theorem}

\newtheorem{prop}[lemma]{Proposition}
\theoremstyle{definition}
\newtheorem{example}[lemma]{Example}
\newtheorem{definition}[lemma]{Definition}

\begin{document}
\linespread{1.1}
{\flushleft{\underline{\bf{{\LARGE{Genus $n$ Forms over Hyperbolic
            Groups\hspace{2cm}\hspace{-9.18068pt}}}}}}}
\vspace{12pt}
{\flushright{{\bf{STEVEN FULTHORP}}\\
{{\small{\emph{School of Mathematics and Statistics\\
Merz Court\\
University of Newcastle upon Tyne\\
Newcastle upon Tyne}}\\
\vspace{12pt}
\today \\
\vspace{12pt}
AMS Mathematics Subject Classification: 20F12, 20F65, 20F67\\}}}}

\linespread{1.3}
\begin{abstract}
In 1962 M.J.~Wicks \cite{wicks} gave a list of forms for commutators in both
free groups and free products. Since then similar lists have been
constructed for elements of higher genus. In \cite{vdo1} A.~Vdovina
described a method for the construction of forms for elements of any
genus in free products. We shall give a similar result for the
construction of such forms in any hyperbolic group $H$ and from this
we shall obtain a full list of forms for commutators in $H$.
\end{abstract}
\section{Introduction}
In 1962 M.J.~Wicks \cite{wicks} showed that any 
commutator in a free group or a free product of groups could always be
reduced to a particular form.
\begin{example}
For any free group $F(X)$,
a word $u\neq 1$ in $F(X)$ 
is a commutator in $F(X)$ if and only if
$u$ is conjugate to a cyclically reduced word of the form
$ABCA^{-1}B^{-1}C^{-1}$, with $A,B,C\in F(X)$, where at most one of
$A$, $B$ and $C$ may be
equal to the identity.
\end{example}
\begin{example} 
For any free product $G=*_{i\in I}G_i$, if
$v\in G$ is a commutator, either $v\in wG_iw^{-1}$ for some $w\in G$,
$i\in I$, and $v$ is a commutator in $wG_iw^{-1}$, or some fully
cyclically reduced conjugate of $v$ has one of the following forms.
\begin{enumerate}
\item $Xa_1X^{-1}a_2$ with $X\neq 1$, $a_1\neq 1$, $a_1,a_2\in G_i$
  for some $i\in I$, and $a_1$ conjugate to $a_2^{-1}$ in $G_i$; or
\item $Xa_1Ya_2X^{-1}a_3Y^{-1}a_4$ with $X\neq 1$, $Y\neq 1$, $a_1, a_2,
  a_3, a_4\in G_i$ for some $i\in I$, and $a_4a_3a_2a_1=1$; or
\item $Xa_1Yb_1Za_2X^{-1}b_2Y^{-1}a_3Z^{-1}b_3$ with $a_1, a_2, a_3
  \in G_i$ for some $i\in I$ and $a_3a_2a_1=1$, $b_1, b_2, b_3 \in
  G_j$ for $j\in I$ and $b_3b_2b_1=1$, and either not all of $a_1,
  a_2, a_3, b_1, b_2, b_3$ are in any one free factor of $G$ or each
  of $X, Y, Z$ are nontrivial.
\end{enumerate}
\end{example} 
{\flushleft{We}} call these `forms' for commutators in such groups. 
Similar lists of forms have also been constructed in these settings
for certain products of commutators, see \cite{com1} and \cite{vdo1}, 
and products of squares, see \cite{vdo2}.

In \cite{vdo1} A.~Vdovina described a procedure for constructing forms
for elements of any genus $n$ in a free product. This involved an `extension' over the
free product of
the graph associated with some orientable word. (These terms are
explained below).   
In this paper we establish a similar method for constructing
forms for elements of genus $n$ in hyperbolic groups. We then use this to give a
list of all the possible forms for commutators in
hyperbolic groups as Wicks did for free groups and free products.

We begin by introducing a number of definitions including how we
extend an orientable word over a hyperbolic group. This will put us in
the position to state the main result of this paper (see Theorem \ref{Thexp}), a method for
constructing genus $n$ forms in hyperbolic groups. We follow this by
an example of how to implement this theorem. In Section \ref{prelim} we give 
preliminary results which will be needed throughout the proof of
Theorem \ref{Thexp} before moving on to the proof itself in
Section \ref{startthm}. Finally we end this paper by proving
Proposition \ref{propos} in Section \ref{commutators} which states the
possible forms for commutators in a
hyperbolic group $H$.  
\section{Definitions and Main results}
\subsection{Definitions}
\begin{definition} Let $G$ be a group and $g_1,\ldots ,g_t$ be $t$ elements in $G$. We
define the \emph{genus of $(g_1,\ldots ,g_t)$}, denoted $genus_G(g_1,\ldots
,g_t)$, to be equal to $k$, if $k$  is the
smallest integer such that there exist elements $h_l$, $x_i$, $y_i$, for
$i=1,\ldots ,k$ and $l=1,\ldots ,t$, with
\begin{equation*}
h_1g_1h_1^{-1}\ldots h_tg_th_t^{-1} =[x_1,y_1]\ldots [x_k,y_k].
\end{equation*}
\end{definition}        

Let $H=\langle X|R\rangle$ be a hyperbolic group such that
geodesic triangles in the
Cayley graph $\Gamma _{X}(H)$ are $\delta$-thin. Consider any word $w$ in
$F(X)$. We denote the length of $w$ by $|w|$.  If  $|v|\geq
|w|$ for all words $v$ in $F(X)$ such that $w= _H v$ then we say that
the word
$w$ is 
\emph{minimal} in $H$. Clearly minimal words are  represented by
geodesic paths in $\Gamma _{X}(H)$.
If $w$ is not minimal in $H$ then we use the notation $|w| _H$
to denote the length of a word minimal in $H$ which is equal to $w$ in $H$.
We can think of the $\delta$-thin condition as being  equivalent
to the following condition. Let $w$ and $z$  be any words in $F(X)$
which are minimal in $H$,  with
$w=w_1w_2$ and $z=z_1z_2$. If 
\begin{eqnarray*}
|w_2|=|z_1| &\leq & \frac{1}{2}(|w|+|z|-|wz|_H)\\
\textrm{then}\qquad |w_2z_1|_H & \leq & \delta. 
\end{eqnarray*} 

Let $\mathcal{A}$ be an infinite countable alphabet. We
shall  define a word in $\mathcal{A}^{\pm 1}$ to be \emph{orientable quadratic} if
each letter appears exactly twice,
once with exponent $1$ and once with exponent $-1$.
\begin{definition}
A orientable quadratic word $w$ is said to be \emph{redundant}
if there are letters $x$ and $y$ in $\mathcal{A}^{\pm 1}$ which only occur in
$w$ as subwords of the form $(xy)^{\pm 1}$. A word is called
\emph{irredundant}  otherwise.
\end{definition}    
{\flushleft{We}} use the term {\emph{cyclic word}} to mean the equivalence class $[w]$ of a
word $w$ under the relation which relates two words if one is a cyclic
permutation of the other. When we talk about a word $w$ being cyclic we mean
that $w$ is a representative of the equivalence class $[w]$.
If a word $U$ is both cyclic and orientable quadratic then $U$ is
called an \emph{orientable word} in  $\mathcal{A}^{\pm 1}$.
\begin{definition}[Wicks Form] Let $W$ be an  orientable word.  $W$ is called a {\emph{Wicks
    form}} if  the following conditions hold.
\begin{enumerate}
\item $W$ is freely cyclically reduced and
\item $W$ is irredundant.
\end{enumerate}
\end{definition}
Consider an orientable word $U$ of genus $g$ in  $\mathcal{A}^{\pm
    1}$, that is an orientable word such that  
$genus(U)_{F(\mathcal{A})}=g$. Take the disc $D^2$ and divide its
boundary into $|U|$ segments. Write $U$ counterclockwise around the
boundary of a disc, labelling each segment with a letter of $U$. Let
the segments which are labelled by letters with exponent $1$ be
oriented counterclockwise and the segments which are labelled by
letters with exponent $-1$ be oriented clockwise. Now  identify the
segments labelled by the same letters,
respecting orientation. We obtain a  closed compact surface of genus $g$. After
identification the oriented boundary
of the disc gives us an oriented graph embedded on this surface. We shall label
this graph $\Gamma_U$ and call it the genus $g$ graph associated with
$U$.
\begin{example}
Suppose that we have the orientable word $U=ABCA^{-1}B^{-1}C^{-1}$ we
construct $\Gamma_U$ as shown in Figure \ref{gammaU}.
 \begin{figure}[ht]
\begin{center}
\psfrag{A}{{\scriptsize $A $}}
\psfrag{B}{{\scriptsize $B$}}
\psfrag{C}{{\scriptsize $C$}}
\psfrag{U}{{\scriptsize $\Gamma_U $}}
\includegraphics[scale=0.4]{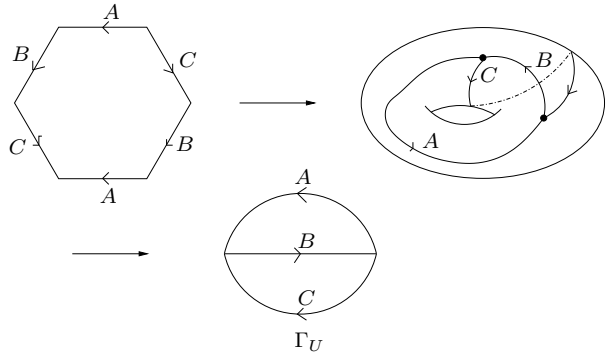}
\caption{Constructing $\Gamma_U$}\label{gammaU}
\end{center}
\end{figure}
\end{example}
{\flushleft{Note}} that the graph $\Gamma _W$ associated to a Wicks form $W$
contains no vertices of degree $1$ or $2$(if it did then rule $1$ or $2$ in the
definition of a Wicks form would be violated).  

Let $\Gamma$ be any oriented connected graph such that an Eulerian
circuit exists in $\Gamma$. Here we are taking Eulerian circuit to be a
circuit which traverses every edge exactly twice once in each
direction. Let $v$ be a vertex of $\Gamma$ of degree $d$ and let the
edges $e_1,\ldots ,e_d$ be incident to $v$ and oriented away from $v$. Note that these are not
necessarily distinct, we may have loops. We define $v$ to be
{\emph{regular}} if the edges can be renumbered such that the cyclic
subwords $e_1^{-1}e_2,\ldots ,e_{d-1}^{-1}e_d,e_d^{-1}e_1$ appear in an
Eulerian circuit. See Figure \ref{regular}.
 \begin{figure}[ht]
\begin{center}
\psfrag{e1}{{\scriptsize $e_1 $}}
\psfrag{e2}{{\scriptsize $e_2$}}
\psfrag{e3}{{\scriptsize $e_3$}}
\psfrag{ed1}{{\scriptsize $e_{d-1}$}}
\psfrag{ed}{{\scriptsize $e_{d}$}}
\includegraphics[scale=0.6]{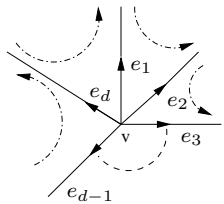}
\caption{Regular vertex $v$}\label{regular}
\end{center}
\end{figure}
If every vertex of $\Gamma$ is regular then we say
that it has a \emph{regular Eulerian circuit}. The following result can easily be deduced from results
found in \cite{vdo1}.
\begin{lemma}
If $U$ is an orientable word in $\mathcal{A}$ then its
associated graph has a regular Eulerian circuit labelled by $U$.
\end{lemma}
Let $\Gamma$ be any graph with a regular Eulerian circuit, $T$
say. Write $T$ around the boundary of a disc and identify the edges,
respecting orientation, to obtain a surface $S$. Then the Euler
characteristic $\chi(S)$ is given by the formula $v-e+1$, where $v$
and $e$ are the  number of vertices and edges respectively in
$\Gamma$ and $\chi(S)$ is equal to $2-2genus(S)$. We 
define the $genus(\Gamma)$ to be equal to $genus(S)$. It follows that this
is given by 
\begin{equation}\label{gengra}
genus(\Gamma)=\frac{1-v+e}{2}.
\end{equation}
\subsection{Extension of an orientable word over a Hyperbolic group $H$}
Let $U$ be an orientable word of genus $k$ in the infinitely countable
alphabet
$\mathcal{A}$ and $\Gamma _U$ its associated genus $k$
graph. We shall be thinking of edges of $\Gamma_U$ being labelled with
letters of $U$ and $U$ itself as being a regular Eulerian
circuit. Recall that $H=\langle X|R\rangle$ is a hyperbolic group in
which geodesic triangles in $\Gamma _X(H)$ are $\delta$-thin. We shall
now give a procedure which can be applied to $\Gamma_U$. This shall be
called an \emph{extension of the orientable word $U$ over $H$}. This
involves the following  three steps.
\begin{enumerate}
\item Let $e$ be a directed edge  in $\Gamma _U$ with end points $u=\iota(e)$
  and $v=\tau(e)$(note that $u$ may be the same vertex as $v$). We replace $e$ by
  two new edges $e_1$ and $e_2$ with labels in $\mathcal{A}^{\pm 1}$
  not in $U$ such that $\iota(e_1)=\iota(e_2)=u$ and $\tau(e_1)=\tau(e_2)=v$.
  We shall do this to every edge in
  $\Gamma_U$ and call the new graph $\Gamma _{U'}$ where $U'$ is the
  circuit which reads $e_1$ wherever $U$  reads $e$ in $\Gamma_U$ and $e_2^{-1}$
  wherever  $U$  reads $e^{-1}$. Let $v$ be  a
  vertex of $\Gamma_U$ of  degree
  $d$ and assume that the edges $e^1,\ldots ,e^d$ are oriented away
  from $v$. Since $v$ is regular, we can renumber the edges such
  that $U$ contains the cyclic subwords $(e^1)^{-1}e^2,\ldots
  ,(e^{d-1})^{-1}e^d,(e^d)^{-1}e^1$. Thus in $\Gamma_{U'}$ the circuit $U'$ contains the cyclic
  subwords  $(e_2^1)^{-1}e_1^2,\ldots
  ,(e_2^{d-1})^{-1}e_1^d,(e_2^d)^{-1}e_1^1$. See Figure \ref{extstep1}.
 \begin{figure}[ht]
\begin{center}
\psfrag{e1}{{\scriptsize $e^1$}}
\psfrag{e2}{{\scriptsize $e^2$}}
\psfrag{e11}{{\scriptsize $e_1^1$}}
\psfrag{e12}{{\scriptsize $e_2^1$}}
\psfrag{e21}{{\scriptsize $e_1^2$}}
\psfrag{e22}{{\scriptsize $e_2^2$}}
\psfrag{ed2}{{\scriptsize $e_2^d$}}
\psfrag{ed1}{{\scriptsize $e_1^d$}}
\psfrag{ed}{{\scriptsize $e^{d}$}}
\includegraphics[scale=0.6]{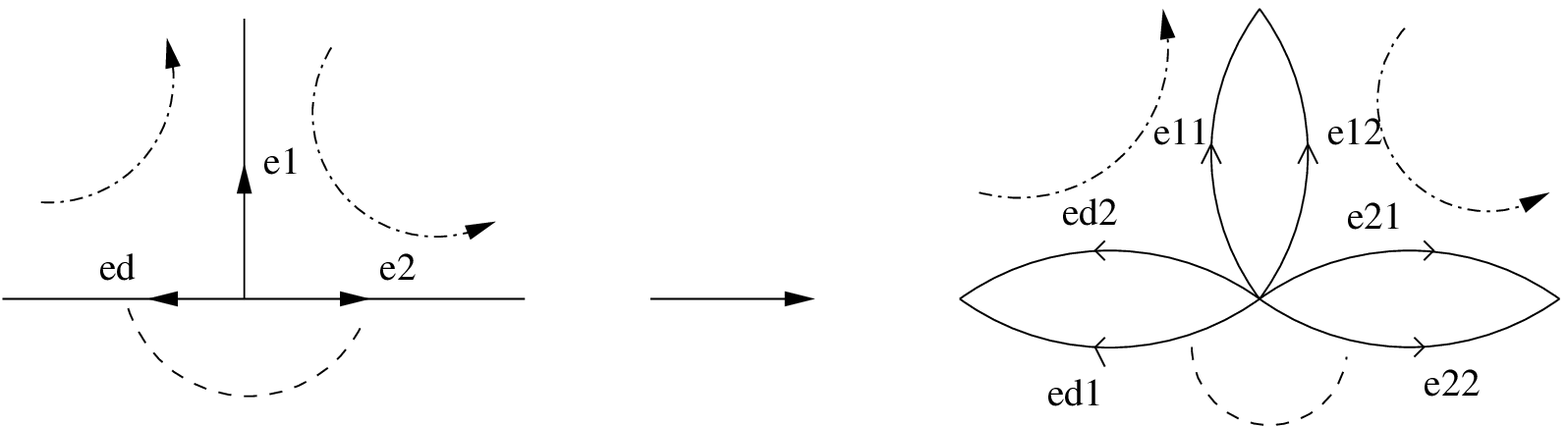}
\\
\vspace{0.38cm}                                                                  
\refstepcounter{figure}\label{extstep1}                                             
Figure \thefigure                  
\end{center}
\end{figure}
\item Now we extend each vertex of $\Gamma_{U'}$ by some cyclic word
  in $F(X)$. Let $v$ be a vertex of $\Gamma_{U'}$ of degree
  $d\geq 2$ and assume that the edges $e_i^1,\ldots ,e_i^d$, for $i=1,2$ are oriented
  away from $v$. Let $w$ be 
  a cyclic word in $F(X)$ such that $w=w_1\ldots w_d$. We have already
  shown that we can  
  renumber the edges of $\Gamma_{U'}$ such that $U'$ contains the cyclic subwords
   $(e_2^1)^{-1}e_1^2,\ldots
   ,(e_2^{d-1})^{-1}e_1^d,(e_2^d)^{-1}e_1^1$. We extend the vertex
   $v$ as follows. Remove the vertex $v$ from $\Gamma _{U'}$ and add
   $d$ distinct vertices $v_1,\ldots ,v_d$ such that
\begin{eqnarray*}
\iota(e_2^d)  = \iota(e_1^1) & = & v_1\\
\textrm{and}\qquad \iota(e_2^k)  =  \iota(e_1^{k+1}) & = & v_k,\qquad
\textrm{for }k=1,\ldots ,d-1.
\end{eqnarray*}
Now add an edge labelled by $w_j$ from $v_j$ to $v_{j+1}$ for all
$j=1,\ldots ,d$. 
 See Figure  \ref{extstep2}.
 \begin{figure}[ht]
\begin{center}
\psfrag{w1}{{\scriptsize $w_1$}}
\psfrag{w2}{{\scriptsize $w_2$}}
\psfrag{e11}{{\scriptsize $e_1^1$}}
\psfrag{e12}{{\scriptsize $e_2^1$}}
\psfrag{e21}{{\scriptsize $e_1^2$}}
\psfrag{e22}{{\scriptsize $e_2^2$}}
\psfrag{ed2}{{\scriptsize $e_2^d$}}
\psfrag{ed1}{{\scriptsize $e_1^d$}}
\psfrag{wd}{{\scriptsize $w_{d}$}}
\includegraphics[scale=0.6]{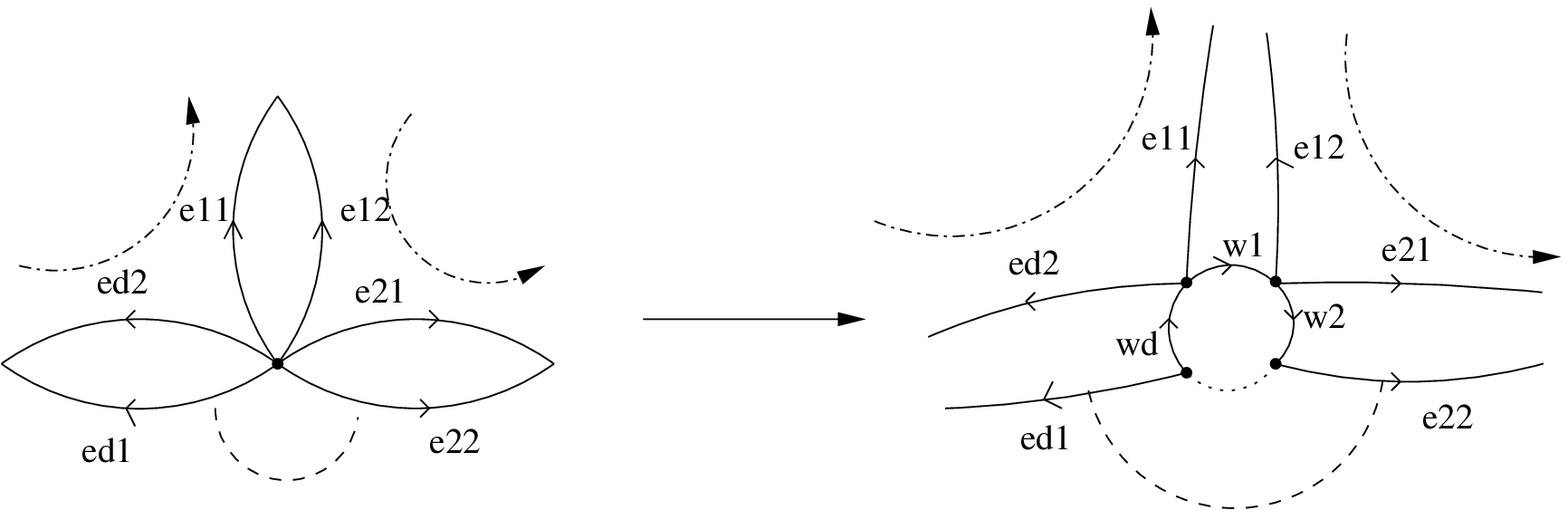}
\\
\vspace{0.38cm}                                                                  
\refstepcounter{figure}\label{extstep2}                                             
Figure \thefigure                  
\end{center}
\end{figure} 
If $v$ is a vertex of degree one then we add a
loop  with the label $w$ to $v$. 
See Figure \ref{extstep2b}. 
 \begin{figure}[ht]
\begin{center}
\psfrag{w}{{\scriptsize $w$}}
\psfrag{e11}{{\scriptsize $e_1^1$}}
\psfrag{e12}{{\scriptsize $e_2^1$}}
\includegraphics[scale=0.5]{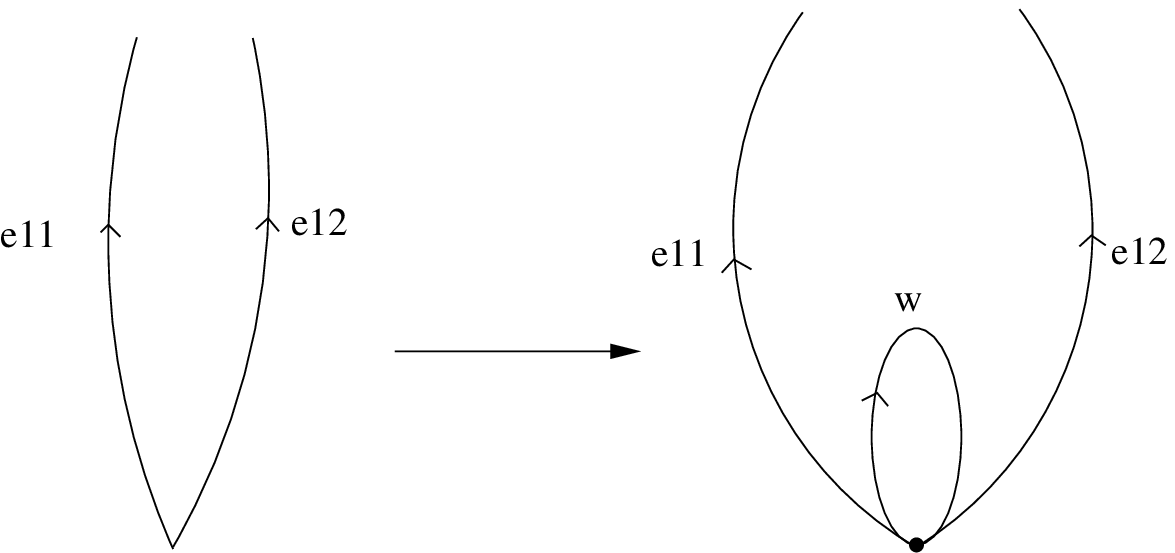}
\\
\vspace{0.38cm}                                                                  
\refstepcounter{figure}\label{extstep2b}                                             
Figure \thefigure                  
\end{center}
\end{figure} 
After each vertex has been extended by some cyclic word in $F(X)$, we
obtain a  new graph which we shall call $\Gamma_{U''}$.  We
can still read $U'$ in this graph. In fact it is now a Hamiltonian
cycle. We shall call $U'$ the \emph{Hamiltonian cycle associated with $U$}.
\item Consider a directed edge $e$ of the original graph $\Gamma _U$. In step
  1  this is
  replaced by a pair of edges $(e_1,e_2)$. Then in step 2 we extend the end
  points of these edges such that we have a  subgraph of $\Gamma
  _{U''}$ of the form shown in Figure \ref{extstep3}, 
 \begin{figure}[ht]
\begin{center}
\psfrag{x}{{\scriptsize $x$}}
\psfrag{y}{{\scriptsize $y$}}
\psfrag{e1}{{\scriptsize $e_1$}}
\psfrag{e2}{{\scriptsize $e_2$}}
\includegraphics[scale=0.5]{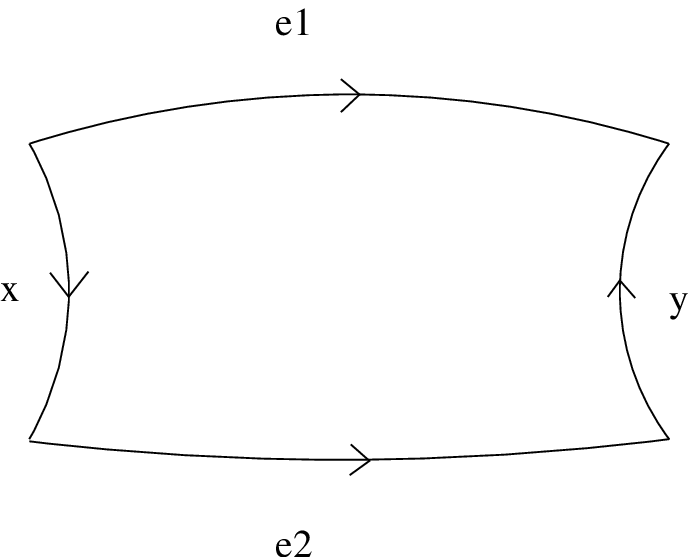}
\\
\vspace{0.38cm}                                                                  
\refstepcounter{figure}\label{extstep3}                                             
Figure \thefigure                  
\end{center}
\end{figure}   
where $x$ and $y$ are subwords of the cyclic words in $F(X)$ used in the
extension of the end points of $e$. We label the pair
$(e_1,e_2)$ with a pair of words $(h_1,h_2)$ in $F(X)$ which are
minimal in $H$ such that $h_1=_H xh_2y$. 
We do this to every pair of edges.  We then have an {\emph{extension
    of $U$ over the hyperbolic group $H$}}.
\end{enumerate}
\subsection{Genus and Length of an Extension of $U$ over $H$}
Suppose that we have an extension of an orientable word $U$ of genus
$k$. We define the {\emph{length of the extension}} to be the sum of
the lengths in $F(X)$ of the cyclic words $w$ given in step 2 of the extension. That
is, if $v_1,v_2,\ldots ,v_m $ are the vertices of $\Gamma _{U}$,  and these
are extended by cyclic words $w_1,w_2,\ldots ,w_m$ respectively, then the
length of the extension is  
\begin{equation*}
\sum_{i=1}^m|w_i|.
\end{equation*}
Let $v_1,\ldots ,v_t$ be a subset of the vertices of $\Gamma_U$ 
that  have been extended by cyclic words $w_1,\ldots ,w_t$. Then
we say that a {\emph{genus $g$ joint extension}} has been constructed
on these vertices if the genus of the $t$-tuple $(w_1,\ldots ,w_t)$ is
equal to $g-t+1$ in $H$. 

Now partition the vertices of $\Gamma_U$ into $p$ sets
$V_1,\ldots ,V_p$ such that a genus $g_i$ joint extension is
constructed on the vertices in the set $V_i$, for all $i=1,\ldots
,p$. We say that $U$ has had a {\emph{genus $g$ extension over $H$}} if
\begin{equation*}
\sum_{i=1}^pg_i=g
\end{equation*}
where $g_i\geq 1$ if $|V_i|=1$ and its only element is a vertex of degree one or
two (i.e.~if $g_i=0$ and hence $t_i=1$ then $V_i$'s only element is a
vertex of degree greater than or equal to $3$).
\subsection{Constructing genus $n$ forms in $H$}
We are now in a position to state the main result of this paper, a
method for constructing forms for elements of genus $n$ in $H$.
In the following theorem $M$ is the number of elements of $H$
represented by  words of length at most
$4\delta$ in $F(X)$ and $l=\delta(\log _2(12n-6)+1)$ where $n$ is given in
the hypothesis.
\begin{thm}\label{Thexp}
Let $h$ be a word in $X\cup X^{-1}$ such that the $genus_H(h)=n$. Then
$h$ is conjugate in $H$ to a word $F$ in $F(X)$ which is minimal in $H$ such that $F$ has one
of the following forms.
\begin{enumerate}
\item $|F|\leq (12n-6)(12l+M+4)$ and $F=_H \theta(W)$ where $W$ is a
  genus $n$ Wicks form and $\theta$ is a map from $F(\mathcal{A})$
  to $F(X)$ such that $|\theta(E)|\leq 12l+M+4$ for each letter $E$ of $W$.
\item $F$ is the label on the Hamiltonian cycle $U'$ obtained by a genus
  $g$  extension of length at most $2(12n-6)(12l+M+4)$ on some
  orientable word $U$ of genus $k$, where $n=g+k$.
\end{enumerate}
\end{thm}
The proof of this theorem runs from page \pageref{startthm} to page
\pageref{endthm}. It involves considering a genus $n$ Wicks form $W$ which can
be mapped to a word in $F(X)$ conjugate to $h$ in $H$ such that 
$W$ has the property of being the shortest length Wicks
form in $F(X)$ which is conjugate to $h$ in $H$. From page
\pageref{l1} to page \pageref{edw} we state and prove a number of
preliminary lemmas involving the Wicks form $W$ as well as showing
that $W$ either takes form 1 in our theorem or can be
reduced to an orientable word $U$ (obtained by setting certain
letters of $W$ to $1$). Then from page \pageref{stext} to page
\pageref{endthm} we will show that if we don't have form 1 then  $h$ is conjugate to a
form given by an
extension of $U$ over $H$ of bounded length and known genus, that is
it takes form 2 in our theorem. Below we give an example of how to
construct one of these genus $n$ forms from an extension.
\begin{example}{\bf{(A possible form of an element of genus $3$ in $H$)}}\\
Let $U=ABCC^{-1}B^{-1}A^{-1}$, an orientable word of genus $0$. Let
$u_1,u_2,v_1,v_2$ be the vertices of the associated graph
$\Gamma_U$. We shall do a joint genus $2$ extension on the vertices
$u_1$ and $u_2$, by words $w_1$ and $w_2$ respectively, and a joint genus $1$
extension on the vertices
$v_1$and $v_2$, by words $z_1$ and $z_2$ respectively. See Figure \ref{extex}.
\begin{figure}[ht]
\begin{center}
\psfrag{A1}{{\scriptsize $A_1$}}
\psfrag{A2}{{\scriptsize $A_2$}}
\psfrag{B1}{{\scriptsize $B_1$}}
\psfrag{B2}{{\scriptsize $B_2$}}
\psfrag{C1}{{\scriptsize $C_1$}}
\psfrag{C2}{{\scriptsize $C_2$}}
\psfrag{A}{{\scriptsize $A$}}
\psfrag{B}{{\scriptsize $B$}}
\psfrag{C}{{\scriptsize $C$}}
\psfrag{u1}{{\scriptsize $u_1$}}
\psfrag{u2}{{\scriptsize $u_2$}}
\psfrag{v1}{{\scriptsize $v_1$}}
\psfrag{v2}{{\scriptsize $v_2$}}
\psfrag{w1}{{\scriptsize $w_1$}}
\psfrag{w21}{{\scriptsize $w_{21}$}}
\psfrag{w22}{{\scriptsize $w_{22}$}}
\psfrag{z11}{{\scriptsize $z_{11}$}}
\psfrag{z12}{{\scriptsize $z_{12}$}}
\psfrag{z2}{{\scriptsize $z_2$}}
\includegraphics[scale=0.6]{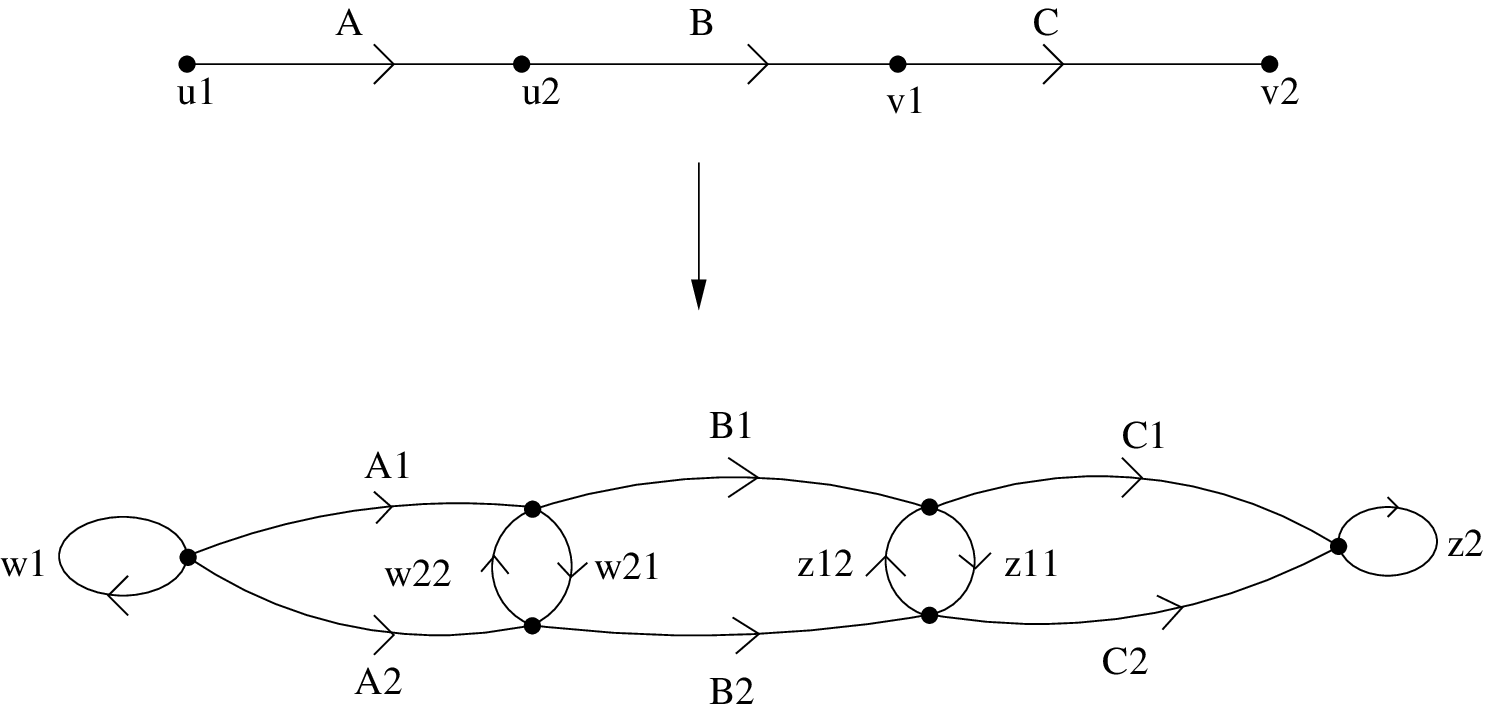}
\\
\vspace{0.38cm}                                                                  
\refstepcounter{figure}\label{extex}                                             
Figure \thefigure                  
\end{center}
\end{figure}
Here $w_2=w_{21}w_{22}$ and $z_1=z_{11}z_{12}$.
Thus, by Theorem \ref{Thexp}, a possible form for an element of genus
$3$ is $F=A_1B_1C_1C_2^{-1}B_2^{-1}A_2^{-1}$ where $F$ is a word in
$F(X)$ which is minimal in $H$ and  
\begin{eqnarray*}
A_1 & =_H & w_1^{-1}A_2w_{22},\\
B_1 & =_H & w_{21}B_2z_{12}^{-1},\\
C_1 & =_H & z_{11}C_2z_2,
\end{eqnarray*} 
$w_1w_2$ is a commutator in $H$, $z_1z_2=1$ and
$|w_1+w_2+z_1+z_2|\leq 60(12l+M+4)$ .
\end{example}
\section{Preliminary Results}\label{prelim}
Let $H=\langle X|R\rangle$ be a finitely generated
hyperbolic group. Then the following lemma by R.~I.~Grigorchuk and
I.~G.~Lysionok  in \cite{grig} shows that the conjugacy problem is
solvable $H$. 
\begin{lemma}[\cite{grig}]\label{lb}
If minimal words $h_1$ and $h_2$ are conjugate in $H$, then a word $w$ can
be found such that $h_1=_Hwh_2w^{-1}$ and 
\begin{equation*}
|w|\leq \frac{1}{2}(|h_1|+|h_2|)+M+1,
\end{equation*}
where $M$ is the number of elements of $H$ represented by words of
length $\leq 4\delta $.
\end{lemma}
The following lemma can be found in K.J.~Friel \cite{friel}, Lemma
$2.2.4$. A proof has been provided, to help understand Lemma  \ref{lac} which
follows it. 

\begin{lemma}\label{la}
Suppose that we have a closed path  $q=\gamma _0\gamma _1\ldots\gamma
_n$ in $\Gamma_X(H)$, where $\gamma _i$ is a geodesic path, for $i=0,\ldots , n$.
Let vertices $\zeta _1$ and $\zeta _2$ lie on  $\gamma
_0-\{\iota(\gamma_0),\tau(\gamma_0)\}$  such that
$d(\iota(\gamma _0),\zeta _1)<d(\iota(\gamma _0),\zeta _2)$. Then
vertices $\eta _1,\eta_2\notin \gamma _0$  can
be found on $q$  such that
\begin{enumerate}
\item $d(\zeta _1 ,\eta _1),d(\zeta _2 ,\eta _2)\leq \delta (\log _2(n)+1)$;

\item  $d_q(\iota(\gamma _0),\eta _1)>d_q(\iota (\gamma _0), \eta _2)$ and
\item if $\eta _1$, $\eta _2\in\gamma _i$ for some $i \neq 0$ then
$d(\eta _1,\eta _2)=d(\zeta _1,\zeta _2)$.   
\end{enumerate} 
\end{lemma}
\begin{proof}
Now for some integer $k$ it follows that $\log _2n\leq k \leq \log _2
(n) +1$. Therefore by adding paths of length zero between $\tau(\gamma_n)$ and
$\iota(\gamma_0)$, we may assume that $n=2^k$ and replace the bound 
\begin{equation*}
 d(\zeta _1 ,\eta _1),d(\zeta _2 ,\eta _2)\leq \delta (\log _2(n)+1)
\end{equation*}
in part $1$ by the bound
\begin{equation*}
 d(\zeta _1 ,\eta _1),d(\zeta _2 ,\eta _2)\leq \delta k.
\end{equation*}
We now carry out a subdivision on the closed path $q$. Let $q_0$
and $q_1$ be  geodesic paths from $\tau(\gamma
_0)$ and  $\tau(\gamma_{2^{k-1}})$ to  $\tau(\gamma_{2^{k-1}})$ and $\iota(\gamma _0)$ 
respectively. Let $b=b_1\ldots b_m$ be a binary sequence of length
$m\leq k-1$, where $b_i=0$ or $1$ for all $i=1,\ldots ,m$. We define $q_{b0}$ and
$q_{b1}$ to be  geodesic paths from $\iota(q_b)$ and
$\tau(\gamma_{r(b)})$ to $\tau(\gamma_{r(b)})$ and $\tau(q_b)$ respectively, where 
\begin{equation*}
r(b)=b_12^{k-1}+b_22^{k-2}+\ldots +b_m2^{k-m}+2^{k-(m+1)}. 
\end{equation*}
Note that, if $m=k-1$ then we choose the geodesic paths $q_{b0}$ and
$q_{b1}$ to be $\gamma_{r(b)}$ and $\gamma_{r(b)+1}$ respectively.
This gives a subdivision of $q$ into geodesic triangles.  
For example if $k=3$ we have the subdivision shown in Figure \ref{subdiv}.
\begin{figure}[ht]
\begin{center}
\psfrag{q0}{{\scriptsize{$q_0$}}}
\psfrag{q1}{{\scriptsize{$q_1$}}}
\psfrag{q00}{{\scriptsize{$q_{00}$}}}
\psfrag{q01}{{\scriptsize{$q_{01}$}}}
\psfrag{q10}{{\scriptsize{$q_{10}$}}}
\psfrag{q11}{{\scriptsize{$q_{11}$}}}
\psfrag{g0}{{\scriptsize{$\gamma _0$}}}  
\psfrag{g1}{{\scriptsize{$\gamma _1=q_{000}$}}}    
\psfrag{g2}{{\scriptsize{$\gamma _2=q_{001}$}}}  
\psfrag{g3}{{\scriptsize{$\gamma _3=q_{010}$}}}    
\psfrag{g4}{{\scriptsize{$\gamma _4=q_{011}$}}}  
\psfrag{g5}{{\scriptsize{$\gamma _5=q_{100}$}}}    
\psfrag{g6}{{\scriptsize{$\gamma _6=q_{101}$}}}  
\psfrag{g7}{{\scriptsize{$\gamma _7=q_{110}$}}}    
\psfrag{g8}{{\scriptsize{$\gamma _8=q_{111}$}}}  
\includegraphics[scale=0.5]{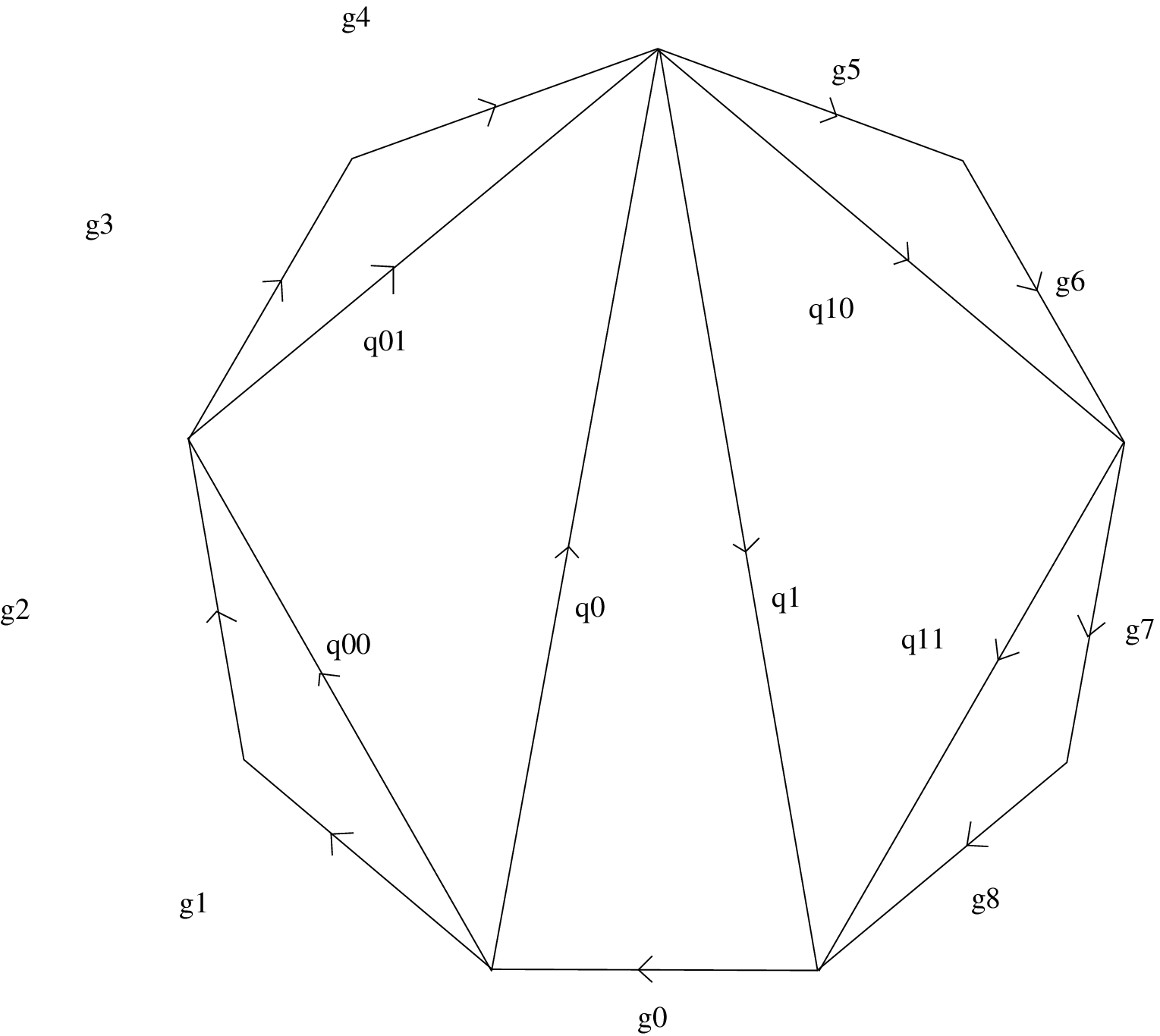}
\\
\vspace{0.38cm}                                                                  
\refstepcounter{figure}\label{subdiv}                                             
Figure \thefigure                  
\end{center}
\end{figure}

Consider a geodesic triangle of the subdivision with sides $q_b$, $q_{b0}$ and $q_{b1}$. Let 
$v$ be a vertex lying on $q_b$. By the definition of $\delta$-thin
triangles there exists a vertex $v'$ on $q_{b0}\cup q_{b1}$ such that
$d(v,v')\leq  \delta$ and
either
\begin{enumerate}
\item[(i)] $v'\in q_{b0}$ and $d(\iota(q_{b}),v)=d(\iota(q_{b0}),v')$,
  or
\item[(ii)] $v'\in q_{b1}$ and $d(\tau(q_{b}),v)=d(\tau(q_{b1}),v')$.
\end{enumerate}
From this we  see that for any vertex on a geodesic path $q_b$
there always exists a vertex on a geodesic path $q_{b'}$ where the
length of the binary sequence $b'$ is one greater than $b$.

Now $\zeta _1\in \gamma _0$ is within $\delta $ of some vertex $v _1$  on either
$q_0$ or $q_1$. From the above there exists a sequence of vertices
$v_1,\ldots ,v_k$, where each $v_i$ lies on a $q_{b(v_i)}$ such that the length of the
binary sequence $b(v_{i+1})$ is one greater than the length of the
binary sequence $b(v_{i})$ and $d(v_i,v_{i+1})\leq \delta$ for all
$i=1,\ldots ,k-1$. This implies that $d(\zeta _1,v_k)\leq \delta k$ and
from the construction of the subdivision $q_{b(v_{k})}=\gamma _j$ for
  some $j=1,\ldots ,n$. Therefore, let $v_k=\eta_1$. We shall denote this path
  passing through the sequence of vertices by $s_1$,
  i.e.~$|s_1|\leq\delta k$.  For example if
  $k=3$ we have a path as shown in figure
  \ref{subdiv2}.
\begin{figure}[ht]
\begin{center}
\psfrag{s}{{\scriptsize{$s_1$}}}
\psfrag{z1}{{\scriptsize{$\zeta _1$}}}
\psfrag{e1}{{\scriptsize{$\eta_1$}}}
\psfrag{q0}{{\scriptsize{$q_0$}}}
\psfrag{q1}{{\scriptsize{$q_1$}}}
\psfrag{q00}{{\scriptsize{$q_{00}$}}}
\psfrag{q01}{{\scriptsize{$q_{01}$}}}
\psfrag{q10}{{\scriptsize{$q_{10}$}}}
\psfrag{q11}{{\scriptsize{$q_{11}$}}}
\psfrag{g0}{{\scriptsize{$\gamma _0$}}}  
\psfrag{g1}{{\scriptsize{$\gamma _1=q_{000}$}}}    
\psfrag{g2}{{\scriptsize{$\gamma _2=q_{001}$}}}  
\psfrag{g3}{{\scriptsize{$\gamma _3=q_{010}$}}}    
\psfrag{g4}{{\scriptsize{$\gamma _4=q_{011}$}}}  
\psfrag{g5}{{\scriptsize{$\gamma _5=q_{100}$}}}    
\psfrag{g6}{{\scriptsize{$\gamma _6=q_{101}$}}}  
\psfrag{g7}{{\scriptsize{$\gamma _7=q_{110}$}}}    
\psfrag{g8}{{\scriptsize{$\gamma _8=q_{111}$}}}  
\includegraphics[scale=0.5]{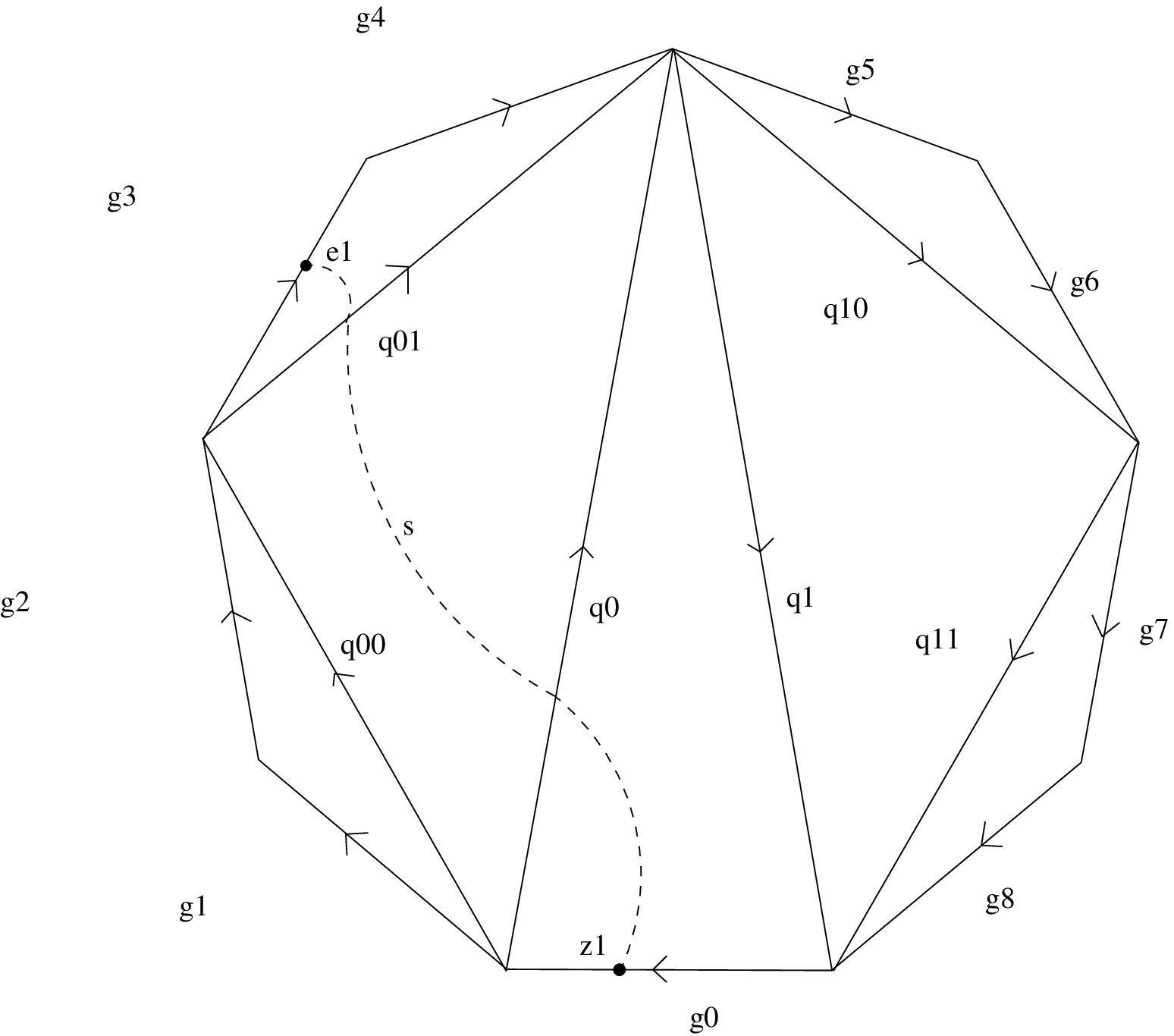}
\\
\vspace{0.38cm}                                                                  
\refstepcounter{figure}\label{subdiv2}                                             
Figure \thefigure                  
\end{center}
\end{figure}
Clearly the same construction of a path $s_2$ of length at most
$\delta k$ applies to $\zeta_2$.
Hence part $1$ of the lemma holds.

Consider two vertices $u_1$ and $u_2$ lying on the geodesic $q_b$ for some binary
sequence $b$ such that $d(\iota(q_b),u_2)=d(\iota(q_b),u_1)+B$, for
some positive constant $B$. Then by the definition of $\delta$-thin
triangle there exist vertices $u'_1, u_2'$ lying on $q_{b0}\cup
q_{b1}$ such that $d(u_1,u_1'),d(u_2,u_2')\leq \delta$ and either
\begin{enumerate}
\item [(a)]$u_1',u_2'\in q_{b0}$ and $d(\iota(q_{b0}),u_2')=
  d(\iota(q_{b0}),u_1')+B$, or
\item [(b)] $u_1',u_2'\in q_{b1}$ and $d(\tau(q_{b1}),u_2')=
  d(\tau(q_{b0}),u_1')-B$, or
\item [(c)] $u_2'\in q_{b0}$ and $u_1'\in q_{b1}$.
\end{enumerate}
Now suppose that  paths $s_1$ and $s_2$ follow the same sequence of $q_b$'s, until
for some $q_{b'}$, $s_2$ meets $q_{b'0}$ and $s_1$ meets
$q_{b'1}$. Let $s_1$ and $s_2$ meet the geodesic path $q_{b'}$ at the vertices $x_1$ and
$x_2$ respectively. By (a) and (b) it follows that
$d(\zeta_1,\zeta_2)=d(x_1,x_2)$ and
$d(\iota(q_{b'}),x_2)<d(\iota(q_{b'}),x_1)$. Since  $s_2$ meets $q_{b'0}$ and $s_1$ meets
$q_{b'1}$, the construction of the $q_b's$ implies that
$d_q(\iota(\gamma _0),\eta _1)>d_q(\iota (\gamma _0), \eta _2)$. 

Finally, suppose that $s_1$ and $s_2$ meet exactly the same sequence of $q_b's$. Let
 $q_{b'}$ be the last in this sequence. Again  by our construction of the
$q_b's$, $q_{b'}=\gamma _i$ for some $i=1,\ldots ,n$ and by (a) and (b) it follows that
$d(\zeta_1,\zeta_2)=d(\eta_1,\eta_2)$ and $d_q(\iota(\gamma _0),\eta
_1)>d_q(\iota (\gamma _0), \eta _2)$. Hence both part $2$ and $3$ hold.
\end{proof}
\begin{lemma}\label{lac}
Given $\zeta_1$ and $\eta_1$ from above, if $\zeta_3$ is vertex lying
on $q$ such that 
\begin{equation*}
d_q(\iota(\gamma_0),\zeta_1)<d_q(\iota(\gamma _0),\zeta
_3)<d_q(\iota(\gamma_0),\eta _1)
\end{equation*}
and $\zeta _3\neq \iota(\gamma_i)$ or $\tau(\gamma_i)$ for some
$i=1,\ldots ,n$ then a vertex $\eta _3\neq \zeta_3$ can be found on $q$ such that 
\begin{enumerate}
\item $d(\zeta _3,\eta_3)\leq 2\delta (\log _2(n) +1)$ and
\item $d_q(\iota(\gamma_0),\eta_1)>d_q(\iota(\gamma_0),\eta_3)$.
\end{enumerate}
\end{lemma}
\begin{proof}
Again consider the  subdivision constructed above with the path $s_1$
from $\zeta_1$ to $\eta_1$. Let the terminal vertex of the path  $s_1$ lie on
$\gamma_l$ for some $l\leq n$.
If $\zeta_3$ lies on
$\gamma_0$ then we have the hypothesis of the previous lemma, and the
lemma holds. So assume that $\zeta_3\in \gamma_i$ for some
$i=1,\ldots ,l$. 

Consider a geodesic triangle $q_bq_{b0}q_{b1}$ in the subdivision, for
some binary sequence $b$. Let $v$ be a vertex on $q_{b\varepsilon}$,
where $\varepsilon =0$ or $1$. Then there exists a vertex $v'$ on
$q_b\cup q_{b\xi}$, where $\xi =(\varepsilon +1)\mod 2$, such that $d(v,v')\leq
\delta$. 

It follows from the  subdivision  that $\gamma_i=q_{b'}$ for some binary sequence
${b'}$ of length $k$ (remember that $n=2^k$). By the above paragraph
we may 
choose a sequence  of vertices $\zeta_3=v_1,v_2,\ldots $
through the $q_b$'s such that each  $v_j$ lies on $q_{b(v_j)}$, where
each binary sequence $b(v_j)$ is distinct. Let $v_r$ be
the first vertex such that either
\begin{enumerate}
\item[(i)] the length of $b(v_r)$ is the same as the length of $b(v_{r+1})$, or
\item[(ii)] $v_r$ lies on $q_1\cup q_2$.
\end{enumerate}      
(i) The length of the binary sequence  decreases by one each
time until $v_r$ is reached. Therefore there exists a path of length
at most $\delta k$
from $\zeta _3=v_1$ to $v_r$. Now by the construction of
the subdivision and the argument in the previous lemma, the length of
the binary sequence associated to each vertex in the sequence
$v_{r+1},v_{r+2},\ldots $ increases by one each time. Thus for some $m$, $v_m$
lies on $q$ and there exists a path from $v_{r+1}$ to $v_m$ of length
at most $\delta k$. Let $\eta_3=v_m$. Then $d(\zeta_3,\eta_3)\leq
2\delta k$ and in this case the part $1$ holds.

{\flushleft (ii)} Without loss of generality let $v_r$ lie on $q_0$. Once again there is a path from $\zeta_3=v_1$ to
$v_r$ of length at most $\delta k$. If $v_{r+1}$ lies on $\gamma_0$
then $\eta_3=v_{r+1}$ and part $1$ holds. If $v_{r+1}$ lies on $q_1$,
then we use the same argument to as (i) to show that there is a $v_r$
on $q$ such that $d(v_{r+1},v_m)\leq \delta k$. Hence part $1$ holds
in all cases. 

Let $s_3$ be the path of length at most $2\delta k$ constructed
above. From the previous lemma $s_1$ passes through $k$ distinct $q_b$'s. Let
the sequence of vertices which lie on these $q_b$'s be  
$\zeta_1=u_0,u_1,\ldots ,u_k=\eta_1$ such that $u_j$ lies on
$q_{b(u_j)}$ for some binary sequence $b(u_j)$. Now, if $s_3$ never
passes through $q_{b(u_j)}$ for all $j=1,\ldots ,k$, then by the
construction of the subdivision, part $2$ clearly holds. Therefore,
assume that $s_3$ passes through $q_{b(u_j)}$ at the vertex $x$, for
some $j$.  Now, since
$d_q(\iota(\gamma_0),\zeta_1)<d_q(\iota(\gamma _0),\zeta
_3)<d_q(\iota(\gamma_0),\eta _1)$, it follows from the construction of
$s_3$ that 
 $d(\iota(q_{b(u_j)}),x)\leq d(\iota(q_{b(u_j)}),u_j)$. It clearly
 follows from the construction of the $q_b$'s and statement (a),
 (b) and (c) from the previous lemma that
 $d_q(\iota(\gamma_0),\eta_1)>d_q(\iota(\gamma_0),\eta_3)$. See Figure
  \ref{subdiv3}.
\begin{figure}[ht]
\begin{center}
\psfrag{s}{{\scriptsize{$s_1$}}}
\psfrag{z1}{{\scriptsize{$\zeta _1$}}}
\psfrag{e1}{{\scriptsize{$\eta_1$}}}
\psfrag{z3}{{\scriptsize{$\zeta _3$}}}
\psfrag{e3}{{\scriptsize{$\eta_3$}}}
\psfrag{q0}{{\scriptsize{$q_0$}}}
\psfrag{q1}{{\scriptsize{$q_1$}}}
\psfrag{q00}{{\scriptsize{$q_{00}$}}}
\psfrag{q01}{{\scriptsize{$q_{01}$}}}
\psfrag{q10}{{\scriptsize{$q_{10}$}}}
\psfrag{q11}{{\scriptsize{$q_{11}$}}}
\psfrag{g0}{{\scriptsize{$\gamma _0$}}}  
\psfrag{g1}{{\scriptsize{$\gamma _1=q_{000}$}}}    
\psfrag{g2}{{\scriptsize{$\gamma _2=q_{001}$}}}  
\psfrag{g3}{{\scriptsize{$\gamma _3=q_{010}$}}}    
\psfrag{g4}{{\scriptsize{$\gamma _4=q_{011}$}}}  
\psfrag{g5}{{\scriptsize{$\gamma _5=q_{100}$}}}    
\psfrag{g6}{{\scriptsize{$\gamma _6=q_{101}$}}}  
\psfrag{g7}{{\scriptsize{$\gamma _7=q_{110}$}}}    
\psfrag{g8}{{\scriptsize{$\gamma _8=q_{111}$}}}  
\includegraphics[scale=0.5]{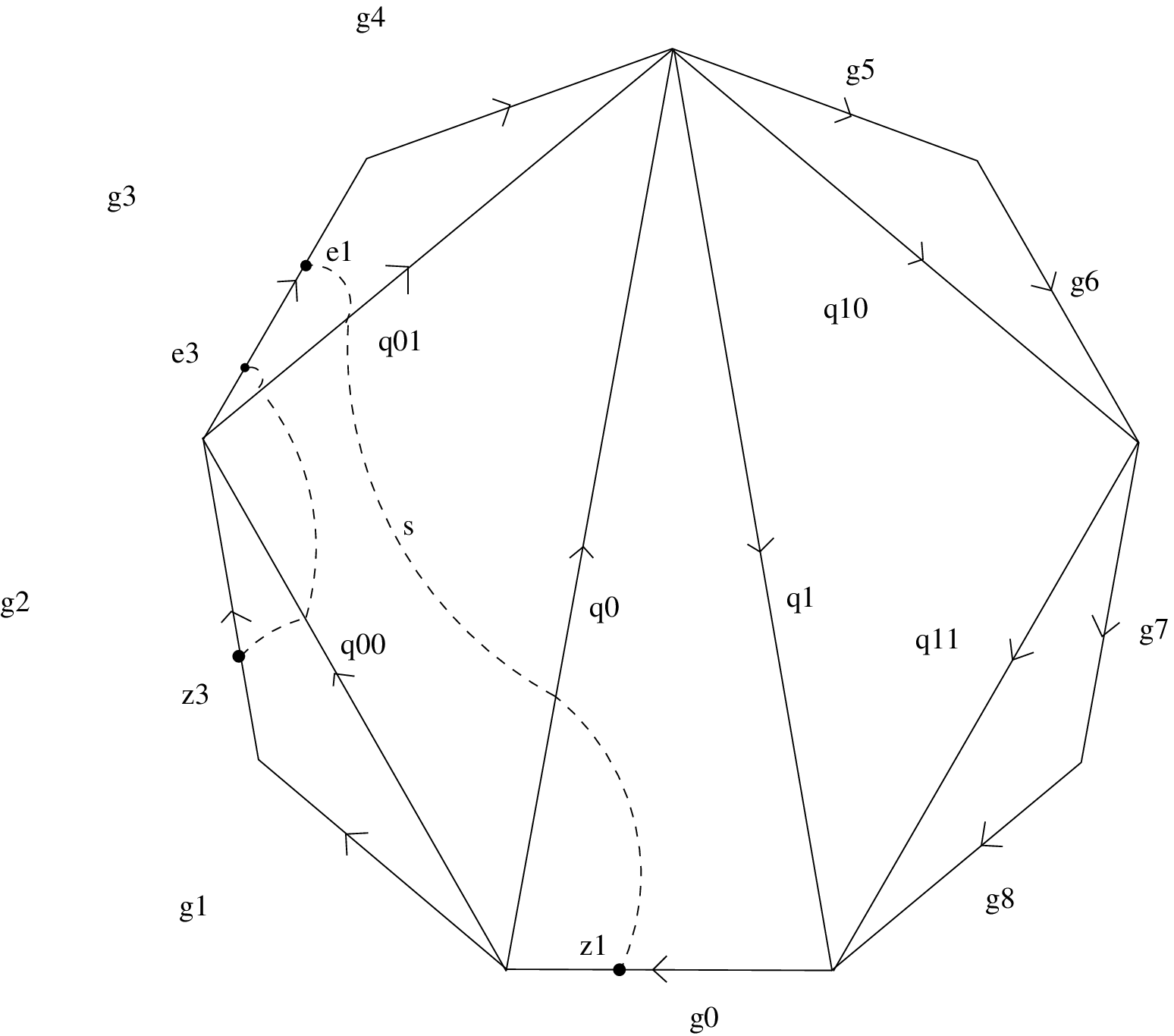}
\\
\vspace{0.38cm}                                                                  
\refstepcounter{figure}\label{subdiv3}                                             
Figure \thefigure                  
\end{center}
\end{figure}
Hence the lemma holds.
\end{proof}
\section{Proof of Theorem \ref{Thexp}}
\begin{proof}[Proof of Theorem \ref{Thexp}.]\label{startthm}  
We begin by considering all genus $n$ Wicks forms in
$\mathcal{A}^{\pm 1}$ (remember that this is an infinitely countable alphabet) and choosing one
which when mapped to a word in $F(X)$ which is conjugate
to h in $H$ has the property of being a word of shortest
length over all such mapping of Wicks forms to words into $F(X)$ which are
conjugate to $h$ in $H$.    

The genus of $h$ is equal to $n$ in $H=\langle X|R\rangle$. Therefore, by definition,
there exist  words $a_i,b_i\in X\cup X^{-1}$, for $i=1,\ldots n$, such that 
\begin{equation*}
h=_H [a_1,b_1][a_2,b_2]\ldots [a_n,b_n].
\end{equation*}
In $\mathcal{A}^{\pm 1}$ the orientable word $U=[A_1,B_1]\ldots
[A_n,B_n]$ is a genus $n$ Wicks form. Let $L=\{A_1,\ldots
,A_n,B_1,\ldots ,B_n\}$ and let $\phi:F(L)\to F(X)$
be a  homomorphism defined by
 $\phi(A_i)=a_i$ and $\phi(B_j)=b_j$, for all $i,j=1,\dots ,n$. We
shall call this a \emph{labelling function for $U$}. Note that  
\begin{eqnarray*}
\phi(U) & = & [\phi(A_1),\phi(B_1)]\ldots[\phi(A_n),\phi(B_n)]\\
        & = & [a_1,b_1]\ldots [a_n,b_n].
\end{eqnarray*}
Let $\mathcal{F}$ be the set of pairs $(U,\phi)$ where $U$ is a genus
$n$ Wicks form and $\phi$ is a labelling function for $U$ such that
$\phi(U)$ is conjugate to $h$ in $H$. Consider a pair $(W,\theta)$ in
which  $|\theta(W)|$ is minimal amongst all pairs in
$\mathcal{F}$(since we have shown that at least one pair exists in
$\mathcal{F}$ this is always possible). Clearly $\theta(E)$ is minimal
in $H$ for each letter $E$ of $W$ or our choice of minimal pair in
$\mathcal F$ would be incorrect.  We should
note that there exists no genus $m$ Wicks form $V$, $m<n$, with a
labelling function $\psi$ such that $\psi(V)$ is conjugate to $h$ in $H$,
as this would contradict the genus of $h$ in $H$.  

Now, for our minimal pair $(W,\theta)$ in $\mathcal{F}$ we are able to
state a number of preliminary lemmas. The first
lemma we state uses ideas from \cite{lys}, Lemma
$12$. It displays bounded length properties of subwords of $\theta(A)$,
where $A$ is any letter of $W$. This lemma will be used regularly in the proof of Lemma \ref{lsc} which
shows that in the Cayley graph $\Gamma_X(H)$ the path represented by
the word $\theta(W)$ is `close' to a geodesic path representing a word
$F$ which is equal to $\theta(W)$ in $H$.  Let $\Gamma
_W$ be the genus $n$ graph associated to $W$.
\begin{lemma}\label{l1}
Let $E_1,\ldots ,E_r,A,E_{r+1},\ldots ,E_s,A^{-1},E_{s+1},\ldots ,E_t$
be the cyclic  sequence of letters in the regular Eulerian circuit
$W$ in $\Gamma _W$. Let $\theta(A)=a_1a_2$ where $a_1$ and $a_2$ are
subwords of the word
$\theta(A)$ which is minimal in $H$. Similarly, let $\theta(E_i)=e_{i1}e_{i2}$, 
for $1\leq i\leq t$. Then 
\begin{enumerate}
\item[(i)] $|a_1|\leq |e_{i2}\theta(E_{i+1}\ldots E_r)a_1|_H$, for
  $1\leq i\leq r$;
\item[(ii)] $|a_2|\leq |a_2\theta(E_{r+1}\ldots E_{j-1})e_{j1}|_H$,
  for $r+1\leq j\leq s ;$ and 
\item[(iii)] $|a_1|\leq |a_2\theta(E_{r+1}\ldots
  E_sA^{-1}E_{s+1}\ldots E_{k-1})e_{k1}|_H$, for $s+1\leq k\leq t$.
\end{enumerate}   
\end{lemma}
\begin{proof}
First we shall prove statement (ii). Let $t_1$ be a word in $F(X)$
which is minimal in
$H$ such that $t_1=_H a_2\theta(E_{r+1}\ldots E_{j-1})e_{j1}$.
Consider the edges
labelled $A$ and $E_j$  in the graph $\Gamma _W$. Bisect $E_j$ into
two new edges, 
the  first new edge shall be denoted $E_{j1}$ and the
second  $E_{j2}$, where $E_{j1},E_{j2}$ are elements of
$\mathcal{A}^{\pm 1}$ not occurring in $W$.
Now remove edge $A$  and add a new edge,
  $A'$,  joining 
$\iota (A)$ to $\tau(E_{j1})$.   See Figure \ref{newfgra2}.
\begin{figure}[htb]
\begin{center}
\psfrag{A}{{\tiny $A$}}
\psfrag{Ej}{{\tiny $E_j$}}
\psfrag{Er1}{{\tiny $E_{r\hspace{-1.5pt}+\hspace{-1.5pt}1}$}}
\psfrag{Er2}{{\tiny $E_s$}}
\psfrag{A'}{{\tiny $A'$}}
\psfrag{Ej1}{{\tiny $E_{j1}$}}
\psfrag{Ej2}{{\tiny $E_{j2}$}}
\psfrag{Ej'}{{\tiny $E_{j\hspace{-1.5pt}-\hspace{-1.5pt}1}$}}
\psfrag{W}{{\tiny $W$}}
\psfrag{W'}{{\tiny $W'$}}
\psfrag{Er}{{\tiny $E_r$}}
\psfrag{Es}{{\tiny $E_{s\hspace{-1.5pt}+\hspace{-1.5pt}1}$}}
\includegraphics[scale=0.52]{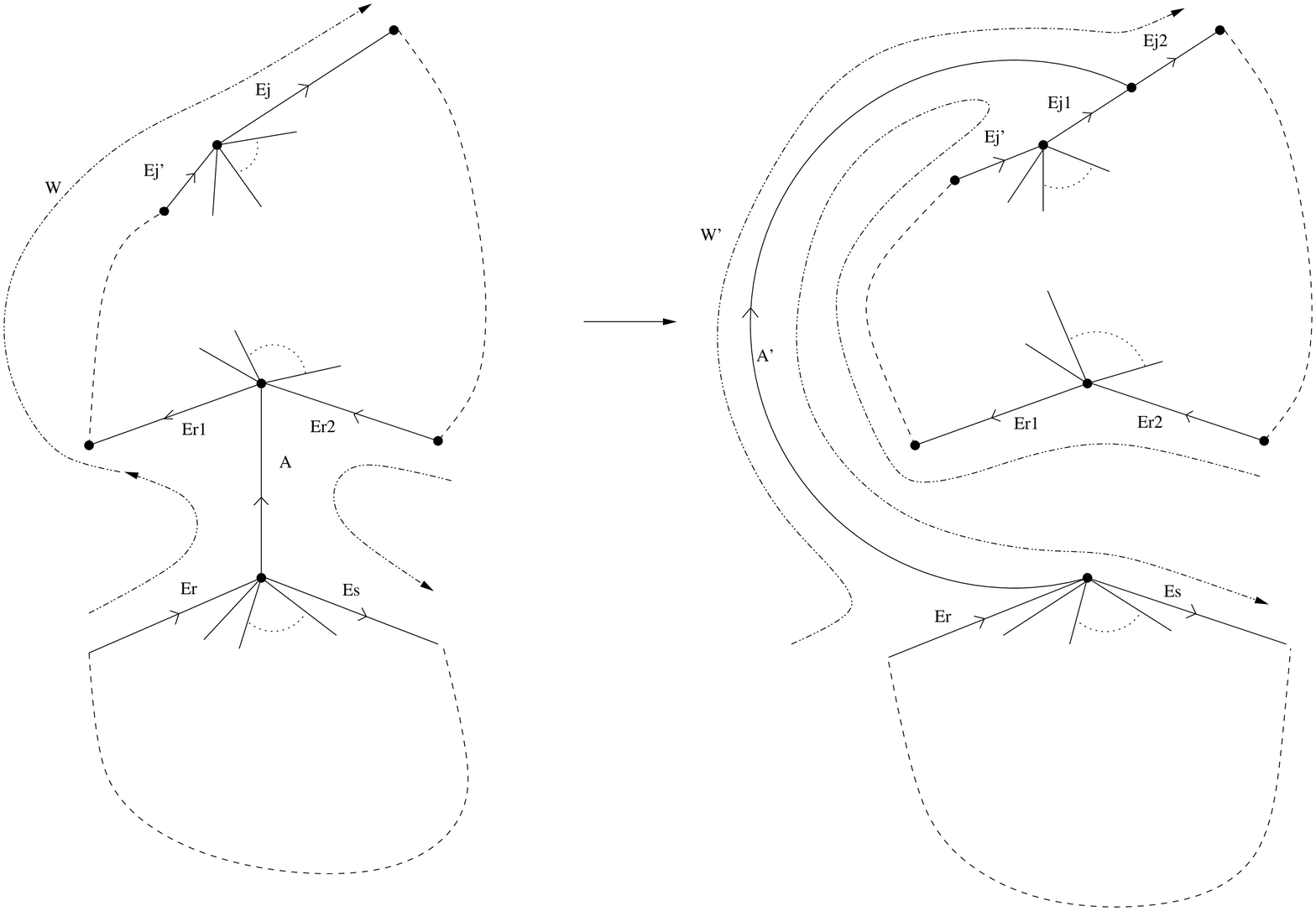}
\\
\vspace{0.38cm}                                                                  
\refstepcounter{figure}\label{newfgra2}                                             
Figure \thefigure                  
\end{center}
\end{figure}
Note that if $\tau(A)$ in $\Gamma _W$ has degree $3$ then we also
remove the edges $E_{r+1}$ and $E_s$ and add a new edge $E'\in
\mathcal{A}^{\pm 1}$ from $\iota(E_s)$ to $\tau(E_{r+1})$.
We can see from the new graph that we have a regular Eulerian
circuit $W'$. We know that $E_j$ occurs as $E_i^{-1}$ for some $i\neq
j$, $1\leq i\leq t$ in $W$ so $E_i$ is replaced by
$E_{j2}^{-1}E_{j1}^{-1}$. But since these new edges remain together in the new
Eulerian circuit $W'$ we shall ignore this as it does not effect the
proof. Therefore the  cyclic sequence of letters of $W'$ is either
\begin{equation*}
E_1,\ldots ,E_r,A',E_{j2},\ldots ,E_s,E_{r+1},\ldots
,E_{j1},A'^{-1},E_{s+1},\ldots ,E_t
\end{equation*}
 if the degree of $\tau(A)$ in $\Gamma _W$ is greater than three or
\begin{equation*}
E_1,\ldots ,E_r,A',E_{j2},\ldots ,E_{s-1},E',E_{r+2}\ldots
,E_{j1},A'^{-1},E_{s+1},\ldots ,E_t, 
\end{equation*}
otherwise.
Suppose that the numbers of vertices and edges  of $\Gamma _W$ are $v$ and $e$
respectively. If the degree of $\tau(A)$ in $\Gamma _W$ is greater than
three  then the numbers
of vertices and edges in the new graph are $v+1$ and $e+1$ respectively.
If the degree of $\tau(A)$ is three then the numbers of vertices
and edges in the new graph are $v$ and $e$ respectively.
In both cases it is easy to check, using equation (\ref{gengra}), that
the new graph also has genus $n$ and
has no vertices of degree $1$ or $2$. Thus $W'$ is a genus $n$ Wicks
form. Now we shall define a labelling function for $W'$. First consider
the case where $\tau(A)$ has degree greater than $3$. Let
$L=\{E_1,\ldots ,E_{j-1},E_{j1},E_{j2},E_{j+1},\ldots ,E_t,A'\}$. We
define a  homomorphism $\psi:F(L)\to F(X)$
in the following way.
\begin{equation*}
\psi(E) \left
\{\begin{array}{ll}
a_1t_1 & \textrm{if $E=A'$}\\
e_{j1} & \textrm{if $E=E_{j1}$}\\
e_{j2} & \textrm{if $E=E_{j2}$}\\
\theta(E) & \textrm{otherwise}
\end{array}
\right. .
\end{equation*} 
Therefore,
\begin{eqnarray*}
\psi(W') & =_H &\psi(E_1\ldots E_rA'E_{j2}\ldots E_sE_{r+1}\ldots
E_{j1}A'^{-1}E_{s+1}\ldots E_t)\\
         & =_H & \psi(E_1\ldots E_r)\psi(A')\psi(E_{j2})\psi(E_{j+1}\ldots
         E_sE_{r+1}\ldots E_{j-1})\\
&&\qquad\qquad\qquad\qquad\qquad\qquad\qquad\qquad
         \psi(E_{j1})\psi(A')^{-1}\psi(E_{s+1}\ldots E_t)\\
         & =_H & \theta(E_1\ldots E_r)a_1t_1e_{j2}\theta(E_{j+1}\ldots
         E_sE_{r+1}\ldots E_{j-1})
         e_{j1}t_1^{-1}a_1^{-1}\theta(E_{s+1}\ldots E_t)\\
         &=_H & \theta(E_1\ldots E_r)a_1a_2\theta(E_{r+1}\ldots E_{j-1})
         e_{j1}e_{j2}\theta(E_{j+1}\ldots E_sE_{r+1}\ldots E_{j-1})\\
&&\qquad\qquad\qquad\qquad\qquad\qquad 
         e_{j1}e_{j1}^{-1}\theta(E_{r+1}\ldots
         E_{j-1})^{-1}a_2^{-1}a_1^{-1}\theta(E_{s+1}\ldots
         E_t)\\
&=_H & \theta(E_1\ldots E_r)a_1a_2\theta(E_{r+1}\ldots E_{j-1})
         e_{j1}e_{j2}\theta(E_{j+1}\ldots 
         E_s)a_2^{-1}a_1^{-1}\theta(E_{s+1}\ldots E_t)\\
&=_H & \theta(E_1\ldots E_rAE_{r+1}\ldots E_sA^{-1}E_{s+1}\ldots E_t)\\    
&=_H & \theta(W).
\end{eqnarray*}
This implies that $\psi(W')$ is conjugate to $h$ in $H$. Thus
$(W',\psi)$ is an element of $\mathcal{F}$. Now since $|\theta(W)|$ was chosen to
be minimal over all pairs in $\mathcal{F}$, it follows that
\begin{eqnarray}\label{psitheta1}
|\psi(W')| & \geq & |\theta(W)|.
           \end{eqnarray}
Also, we can see from Figure \ref{newfgra2} that 
\begin{eqnarray}\label{psitheta2}
|\theta(W)| & = & |\psi(W')|-2|\psi(A')|-2|\psi (E_{j1})|-2|\psi (E_{j2})| +2|\theta(A)|+2|\theta(E_j)|\nonumber\\
            & = &
            |\psi(W')|-2|a_1t_1|-2|e_{j1}|-2|e_{j2}|+2|a_1a_2|+|e_{1j}e_{j2}|\nonumber \\
            & \geq & |\psi(W')|-2|t_1|+2|a_2|.
\end{eqnarray}
Therefore, by equations (\ref{psitheta1}) and (\ref{psitheta2}), it is
clear  that
\begin{eqnarray*}
|a_2\theta(E_{r+1}\ldots E_{j-1})e_{j1}|_H & = &|t_1|\\
         & \geq & |a_2|.               
\end{eqnarray*}
as required.

Now suppose that the degree of $\tau(A)$ in $\Gamma_W$ is three. Let
\begin{equation*}
L'=\{E_1,\ldots ,E_r,E_{r+2},\ldots ,E_{j-1},E_{j1},E_{j2},E_{j+1},\ldots
,E_{s-1},E_{s+2},\ldots  ,E_t,A',E'\}.
\end{equation*}
 We define a  homomorphism $\psi':F(L')\to F(X)$
in the following way.
\begin{equation*}
\psi'(E) \left
\{\begin{array}{ll}
a_1t_1 & \textrm{if $E=A'$}\\
e_{j1} & \textrm{if $E=E_{j1}$}\\
e_{j2} & \textrm{if $E=E_{j2}$}\\
\theta(E_sE_{r+1})     & \textrm{if $E=E'$}\\
\theta(E)     & \textrm{otherwise}
\end{array}
\right. .
\end{equation*} 
We can use the same argument to show that $(W',\psi')\in\mathcal{F}$
and again, since $|\theta(W)|$ was chosen to be minimal over all pairs
in $\mathcal{F}$, it follows that
\begin{eqnarray}\label{psitheta3}
|\psi'(W')| & \geq & |\theta(W)|.
           \end{eqnarray}
Also, we know  that 
\begin{eqnarray}\label{psitheta4}
|\theta(W)| & = &
|\psi'(W')|-2|\psi'(A')|-2|\psi'(E')|-2|\psi'(E_{j1})|-2|\psi'(E_{j2})|\nonumber\\
&&\qquad\qquad\qquad\qquad+2|\theta(A)|+
2|\theta(E_{r+1})|+2|\theta(E_{s})|+
2|\theta(E_{j})|\nonumber\\
            & = &
            |\psi'(W')|-2|a_1t_1|-2|\theta(E_sE_{r+1})|-2|e_{j1}|-2|e_{j2}|      
\nonumber\\
&&\qquad\qquad\qquad\qquad+2|a_
1a_2|+2|\theta(E_{r+1})|+2|\theta(E_{s})|+2|e_{j1}e_{j2}|
\nonumber\\
            & \geq & |\psi'(W')|-2|t_1|+2|a_2|.
\end{eqnarray}
Therefore, by equations (\ref{psitheta3}) and (\ref{psitheta4}), it is
clear  that
\begin{eqnarray*}
|a_2\theta(E_{r+1}\ldots E_{j-1})e_{j1}|_H & = &|t_1|\\
         & \geq & |a_2|.               
\end{eqnarray*}
as required. Hence in both cases $(ii)$ holds.

The same argument, using $(W^{-1},\theta)\in \mathcal{F}$, can be used
to show  that $(i)$ holds. Therefore we only need to consider case
$(iii)$. Let $t_2$ be a word in $F(X)$ which is minimal in $H$ such that 
\begin{equation}
t_2=_H
a_2\theta(E_{r+1}\ldots E_s)a_2^{-1}a_1^{-1}\theta(E_{s+1}\ldots E_{k-1})e_{k1}.
\end{equation}
Again we follow the same method of altering the graph $\Gamma
_W$. Bisect the edge labelled $E_k$. The first half shall be labelled by
$E_{k1}$ and the second half labelled by $E_{k2}$, where $E_{k1},
E_{k2}$ are elements of $ \mathcal{A}^{\pm 1}$ not occurring in $W$. Again, we shall
remove the edge $A$ but now we add a new edge $A''$ joining $\tau(E_{k1})$ to
$\tau(A)$, with $A''\in \mathcal{A}^{\pm 1}$. See Figure \ref{newfgra3}.
\begin{figure}[ht]
\begin{center}
\psfrag{A}{{\tiny $A$}}
\psfrag{Ek}{{\tiny $E_k$}}
\psfrag{Er1}{{\tiny $E_{r\hspace{-1.5pt}+\hspace{-1.5pt}1}$}}
\psfrag{Er2}{{\tiny $E_s$}}
\psfrag{A2}{{\tiny $A''$}}
\psfrag{Ek1}{{\tiny $E_{k1}$}}
\psfrag{Ek2}{{\tiny $E_{k2}$}}
\psfrag{Ek'}{{\tiny $E_{k\hspace{-1.5pt}-\hspace{-1.5pt}1}$}}
\psfrag{W}{{\tiny $W$}}
\psfrag{U}{{\tiny $W''$}}
\psfrag{Er}{{\tiny $E_r$}}
\psfrag{Es}{{\tiny $E_{s\hspace{-1.5pt}+\hspace{-1.5pt}1}$}}
\includegraphics[scale=0.45]{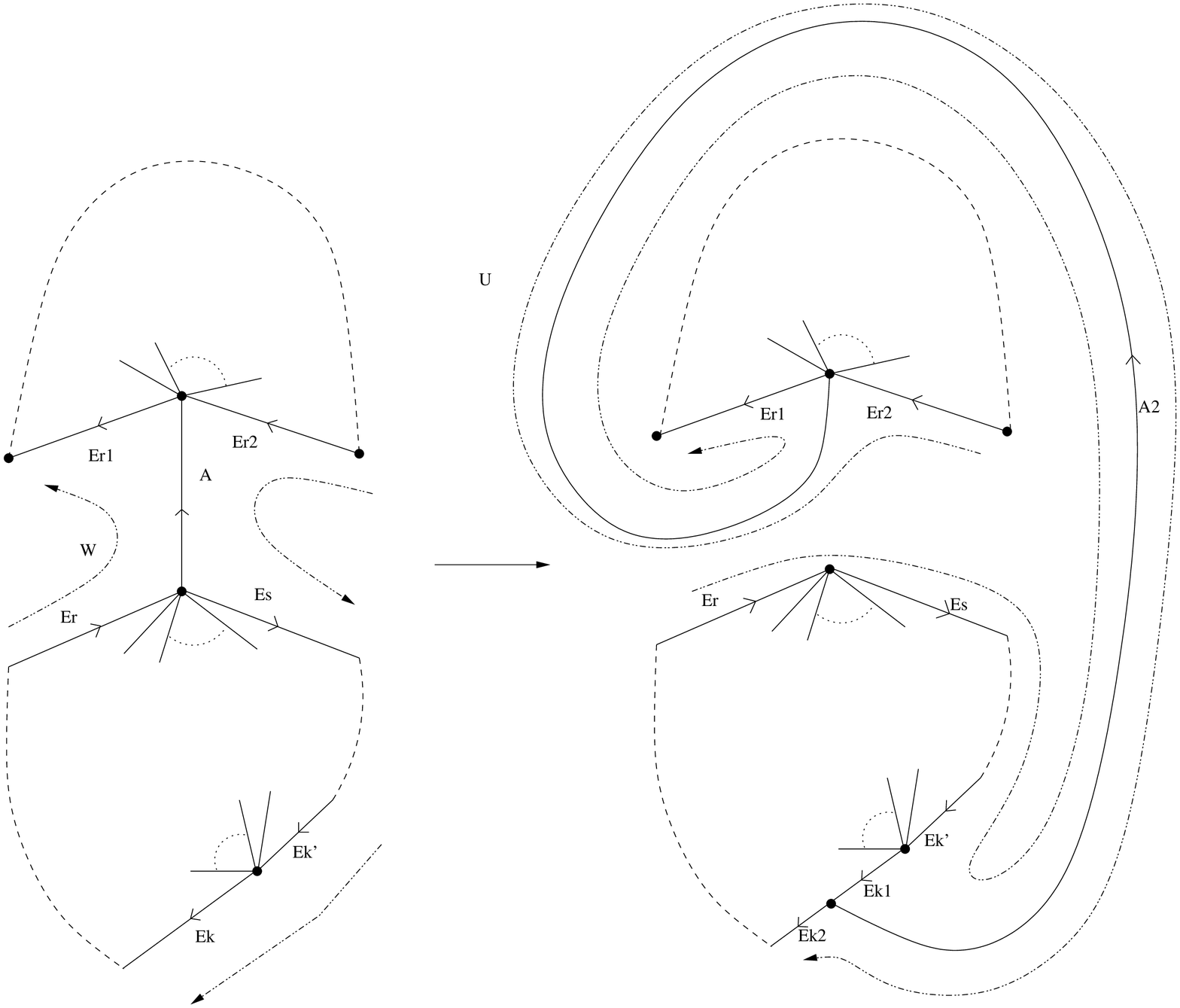}
\\
\vspace{0.38cm}                                                                  
\refstepcounter{figure}\label{newfgra3}                                             
Figure \thefigure                  
\end{center}
\end{figure}
If $\iota(A)$ in $\Gamma _W$ has degree $3$ then we also
remove the edges $E_{r}$ and $E_{s+1}$ and add a new edge $E''\in
\mathcal{A}^{\pm 1}$ from $\iota(E_r)$ to $\tau(E_{s+1})$.
We can see from the new graph that we have a regular Eulerian
circuit $W''$. The cyclic sequence of letters of $W''$ is either
\begin{equation*}
E_1,\ldots ,E_r,E_{s+1},\ldots ,E_{k1},{A''},E_{r+1},\ldots
,E_s,{A''}^{-1},E_{k2},\ldots ,E_t,
\end{equation*}
 if the degree of $\iota(A)$ in $\Gamma _W$ is greater than three, or
\begin{equation*}
E_1,\ldots ,E_{r-1},E'',E_{s+2},\ldots ,,E_{k1},{A''},E_{r+1},\ldots
,E_s,{A''}^{-1},E_{k2},\ldots ,E_t,
\end{equation*}
otherwise. Once again, it is easy to check that the genus of the new
graph is $n$ and it contains no vertices of degree $1$ or $2$. Thus
$W''$ is a genus $n$ Wicks form. We shall define a labelling function
for $W''$. First  consider
the case where $\iota(A)$ has degree greater than $3$. Let
$K=\{E_1,\ldots ,E_{k-1},E_{k1},E_{k2},E_{k+1},\ldots ,E_t,A''\}$. We
define a  homomorphism $\phi:F(K)\to F(X)$
in the following way.
\begin{equation*}
\phi(E) \left
\{\begin{array}{ll}
a_2^{-1}t_2 & \textrm{if $E=A''$}\\
e_{k1} & \textrm{if $E=E_{k1}$}\\
e_{k2} & \textrm{if $E=E_{k2}$}\\
\theta(E) & \textrm{otherwise}
\end{array}
\right. .
\end{equation*} 
Therefore,
\begin{eqnarray*}
\phi(W') & =_H &\phi(E_1\ldots E_rE_{s+1}\ldots
E_{k1}{A''}^{-1}E_{r+1}\ldots E_sA''E_{k2}\ldots E_t\\
         & =_H & \phi(E_1\ldots E_rE_{s+1}\ldots E_{k-1})
         \phi(E_{k1})\phi(A'')^{-1}\phi(E_{r+1} \ldots E_s)
         \\
&&\qquad\qquad\qquad\qquad\qquad\qquad\qquad\qquad\qquad\qquad\phi(A'')\phi(E_{k2})\phi(E_{k+1}\ldots
         E_t)\\
          & =_H & \theta(E_1\ldots E_rE_{s+1}\ldots E_{k-1})
         e_{k1}t_2^{-1}a_2\theta(E_{r+1} \ldots
         E_s)a_2^{-1}t_2e_{k2}\theta(E_{k+1}\ldots
         E_t)\\
         & =_H & \theta(E_1\ldots E_rE_{s+1}\ldots E_{k-1})
         e_{k1}e_{k1}^{-1}\theta(E_{s+1}\ldots
         E_{k-1})^{-1}a_1a_2\theta(E_{r+1}\ldots E_{s})^{-1}a_2^{-1}a_2\\
&& 
         \,\theta(E_{r+1}\ldots E_s)a_2^{-1}a_2\theta(E_{r+1}\ldots E_s)
a_2^{-1}a_1^{-1}\theta(E_{s+1}
\ldots E_{k-1})e_{k1}e_{k2}\theta(E_{k+1}\ldots E_t)\\
        &=_H & \theta(E_1\ldots E_r)a_1a_2\theta(E_{r+1}\ldots
         E_s)a_2^{-1}a_1^{-1}\theta(E_{s+1}\ldots
         E_{k-1})e_{k1}e_{k2}\theta(E_{k+1}\ldots E_t)\\
&=_H & \theta(E_1\ldots E_rAE_{r+1}\ldots E_sA^{-1}E_{s+1}\ldots E_t)\\    
&=_H & \theta(W).
\end{eqnarray*}
This implies that $\phi(W'')$ is conjugate to $h$ in $H$. Thus
$(W'',\phi)$ is an element of $\mathcal{F}$. Now since $|\theta(W)|$ was chosen to
be minimal over all pairs in $\mathcal{F}$, it follows that
\begin{eqnarray}\label{psitheta5}
|\phi(W'')| & \geq & |\theta(W)|.
           \end{eqnarray}
Also, we can see from Figure \ref{newfgra3} that 
\begin{eqnarray}\label{psitheta6}
|\theta(W)| & = & |\phi(W'')|-2|\phi(A'')|-2|\phi (E_{k1})|-2|\phi (E_{k2})| +2|\theta(A)|+2|\theta(E_k)|\nonumber\\
            & = &
            |\phi(W'')|-2|a_2^{-1}t_2|-2|e_{k1}|-2|e_{k2}|+2|a_1a_2|+|e_{k1}e_{k2}|\nonumber \\
            & \geq & |\phi(W'')|-2|t_2|+2|a_1|.
\end{eqnarray}
Therefore, by equations (\ref{psitheta5}) and (\ref{psitheta6}), it is
clear  that
\begin{eqnarray*}
|a_2\theta(E_{r+1}\ldots E_s)a_2^{-1}a_1^{-1}\theta(E_{s+1}\ldots E_{k-1})e_{k_1}|_H & = &|t_2|\\
         & \geq & |a_1|.               
\end{eqnarray*}
as required.
Now Suppose that $\iota(A)$ has degree three in $\Gamma_W$. Let 
\begin{equation*}
K'=\{E_1,\ldots ,E_{r-1},E_{r+1},\ldots ,E_s,E_{s+1},\ldots
,E_{k-1},E_{k1},E_{k2},E_{k+1},\ldots ,E_t,A'',E''\}.
\end{equation*}
We define a homomorphism $\phi':F(K')\to F(X)$ in the following way.
\begin{equation*}
\phi'(E) \left
\{\begin{array}{ll}
a_2^{-1}t_2 & \textrm{if $E=A''$}\\
e_{k1} & \textrm{if $E=E_{k1}$}\\
e_{k2} & \textrm{if $E=E_{k2}$}\\
\theta(E_rE_{s+1}) & \textrm{if $E=E''$}\\ 
\theta(E) & \textrm{otherwise}
\end{array}
\right. .
\end{equation*}
Again, we can use the same argument to show that $(W'',\phi')\in\mathcal{F}$
and again, since $|\theta(W)|$ was chosen to be minimal over all pairs
in $\mathcal{F}$, it follows that
\begin{eqnarray}\label{psitheta7}
|\phi'(W'')| & \geq & |\theta(W)|.
           \end{eqnarray}
Also, we know  that 
\begin{eqnarray}\label{psitheta8}
|\theta(W)| & = &
|\phi'(W'')|-2|\phi'(A'')|-2|\phi'(E'')|-2|\phi'(E_{k1})|-2|E_{k2}|\nonumber\\
&&\qquad\qquad\qquad\qquad+2|\theta(A)|+
2|\theta(E_{r})|+2|\theta(E_{s+1})|+
2|\theta(E_{k})|\nonumber\\
            & = &
            |\phi'(W'')|-2|a_2^{-1}t_2|-2|\theta(E_rE_{s+1})|-2|e_{k1}|-2|e_{k2}|      
\nonumber\\
&&\qquad\qquad\qquad\qquad+2|a_
1a_2|+2|\theta(E_{r})|+2|\theta(E_{s+1})|+2|e_{k1}e_{k2}|
\nonumber\\
            & \geq & |\phi'(W'')|-2|t_2|+2|a_1|.
\end{eqnarray}
Therefore, by equations (\ref{psitheta7}) and (\ref{psitheta8}), it is
clear  that
\begin{eqnarray*}
|a_2\theta(E_{r+1}\ldots E_s)a_2^{-1}a_1^{-1}\theta(E_{s+1}\ldots E_{k-1})e_{k_1}|_H & = &|t_2|\\
         & \geq & |a_1|.               
\end{eqnarray*}
as required. Hence in both cases $(iii)$ holds. 
\end{proof}
Suppose that for each letter $E$ of the Wicks form $W$ we have  $|\theta(E)|\leq
12l+M+4$. By the following lemma the  maximum length of a genus
$n$ Wicks form is  $12n-6$. For a proof of this lemma  see
M.~Culler \cite{cul} Theorem $3.1$.
\begin{lemma}\label{cull}
Let $V$ be a Wicks form in the alphabet $\mathcal{A}^{\pm 1}$ such
that $genus_{F(\mathcal{A})}=m$. Then the length of $V$ is at most $12m-6$.
\end{lemma}
It clearly follows that  we have part
$1$ of the Theorem.  Therefore, we shall assume
that there is at  least one letter of $W$ which is labelled by a word
of length
greater  than $12l+M+4$ in $F(X)$(This
of course implies that there are two since each letter appears twice.) For
convenience in the proof we shall take
$\hat{W}$ to be a cyclic permutation of $W$ such that the last letter
of  $\hat{W}$ is labelled by a word of length  greater than $12l+M+4$ in $F(X)$ but
one  should note that the
proof  does go through using any cyclic permutation.

Consider $\theta(\hat{W})$ as a path in the Cayley graph $\Gamma_X(H)$. Let
$F$ and $R$ be words in $F(X)$ which are minimal  in $H$ such that $F=_H \theta(\hat{W})$ and
$h=_HRFR^{-1}$. See Figure \ref{fandr}.
\begin{figure}[ht]
\begin{center}
\psfrag{F}{{\scriptsize $F$}}
\psfrag{R}{{\scriptsize $R$}}
\psfrag{W}{{\scriptsize $\theta(\hat{W})$}}
\psfrag{h}{{\scriptsize $h$}}
\includegraphics[scale=0.6]{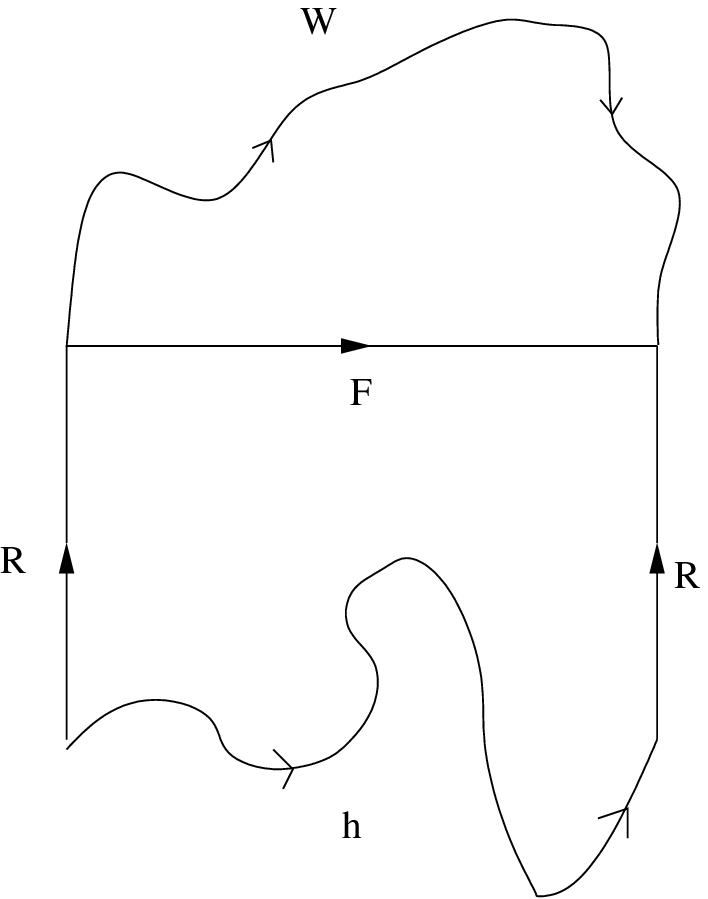}
\\
\vspace{0.38cm}                                                                  
\refstepcounter{figure}\label{fandr}                                             
Figure \thefigure                  
\end{center}
\end{figure} 
Suppose that $\alpha$ is the label on one of the letters of
$\hat{W}$ with $|\alpha|>12l+M+4$. Since the word $\alpha$ is a minimal  in
$H$,  the path in the
Cayley  graph above labelled by $\theta(\hat{W})$  contains a geodesic subpath
labelled by $\alpha$. 
\begin{lemma}\label{lsc}
There exists a vertex $v$ on the geodesic path labelled by $F$ such that
$d(\tau(\alpha),v)\leq 5l+M+3$.
\end{lemma}
\begin{proof}
We shall assume that the letter of $\hat{W}$ labelled by $\alpha$ in
the $\Gamma_X(H)$ appears before be its inverse i.e.~$\hat{W}=\ldots
A \ldots A^{-1}\ldots$ and $\theta(A)=\alpha$. 
It is easy to show that the same proof follows through for the converse.
Let $p_1$ and $q_1$ be vertices on $\alpha$ such that
$d(\iota(\alpha),p_1)=d(q_1,\tau (\alpha))=2l+1$. (Note that we could
use smaller segments of $\alpha$ of length $l+1$ here to prove this lemma
but we require this set up for Lemma \ref{chouv} to hold.) See Figure \ref{l1a}.
 \begin{figure}[ht]
\begin{center}
\psfrag{F}{{\scriptsize $F$}}
\psfrag{A}{{\scriptsize $\alpha$}}
\psfrag{l+1}{{\tiny $2l\hspace{-2pt}+\hspace{-2pt}1$}}
\psfrag{p1}{{\scriptsize $p_1$}}
\psfrag{q1}{{\scriptsize $q_1$}}
\includegraphics[scale=0.6]{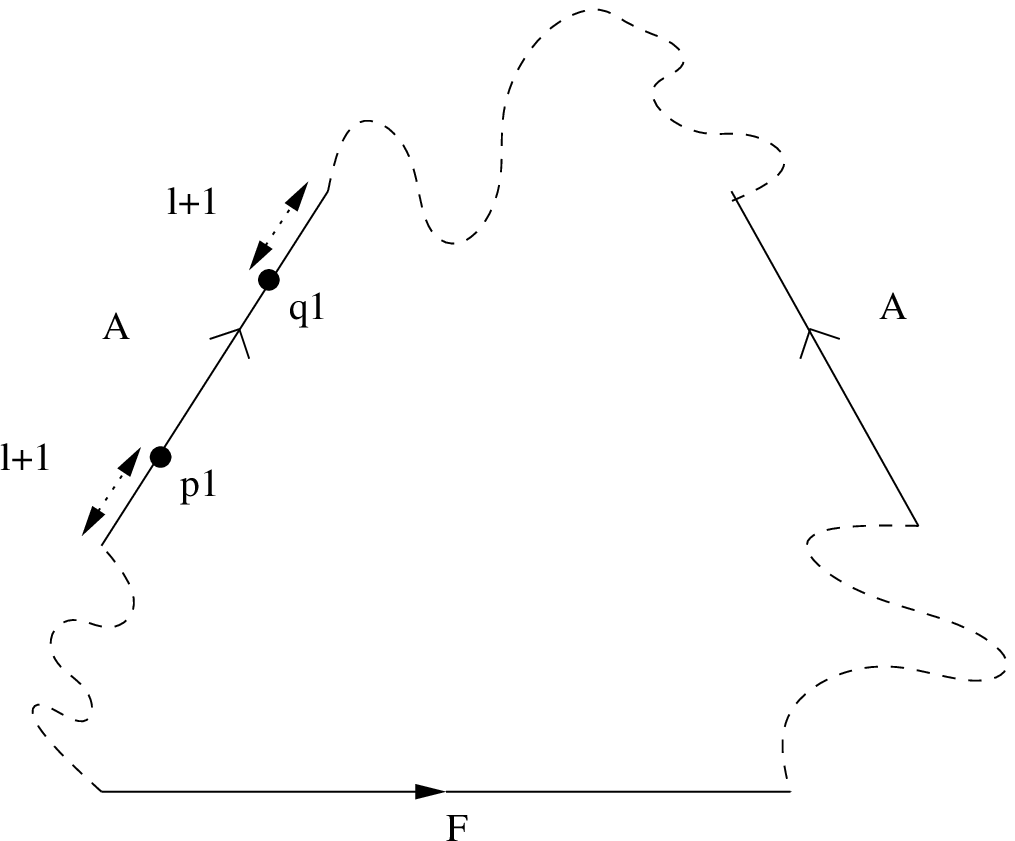}
\\
\vspace{0.38cm}                                                                  
\refstepcounter{figure}\label{l1a}                                             
Figure \thefigure                  
\end{center}
\end{figure}
By Lemma \ref{la} part 1, there are vertices $p_2$ and $q_2$, which lie
either on
the path labelled by $\theta(\hat{W})-\{\alpha\}$ or on the
geodesic path labelled by $F$, such
that 
\begin{eqnarray*}
d(p_1 ,p_2), d(q_1 ,q_2) & \leq &  \delta(\log _2(|\hat{W}|))\\
                         & \leq &  \delta(\log _2(12n-6)+1)\\
                         & \leq & l,
\end{eqnarray*}
Suppose that
$p_2$ lies on some geodesic path $\beta$ which is the label of some
letter $B$ in $\hat{W}$ which is different to $A$. 
Lemma \ref{l1} implies that $p_1$ is within $l$ of 
$\iota(\alpha)\cup\tau(\alpha)$ but we have chosen $p_1$ such that
this is not the case. Therefore $p_2$ and similarly $q_2$  can
only lie on $\alpha ^{-1}\cup F$. Also, if $p_2$ lies on $\alpha ^{-1}$ then by Lemma
\ref{la} part 2, $q_2$ lies on $\alpha^{-1}$. This leaves us with just
three possibilities. See Figure \ref{l1b}.
 \begin{figure}[ht]
\begin{center}
\psfrag{F}{{\scriptsize $F$}}
\psfrag{A}{{\scriptsize $\alpha$}}
\psfrag{l+1}{{\tiny $2l\hspace{-2pt}+\hspace{-2pt}1$}}
\psfrag{l}{{\tiny $\leq l$}}
\psfrag{p1}{{\scriptsize $p_1$}}
\psfrag{q1}{{\scriptsize $q_1$}}
\psfrag{p2}{{\scriptsize $p_2$}}
\psfrag{q2}{{\scriptsize $q_2$}}
\includegraphics[scale=0.5]{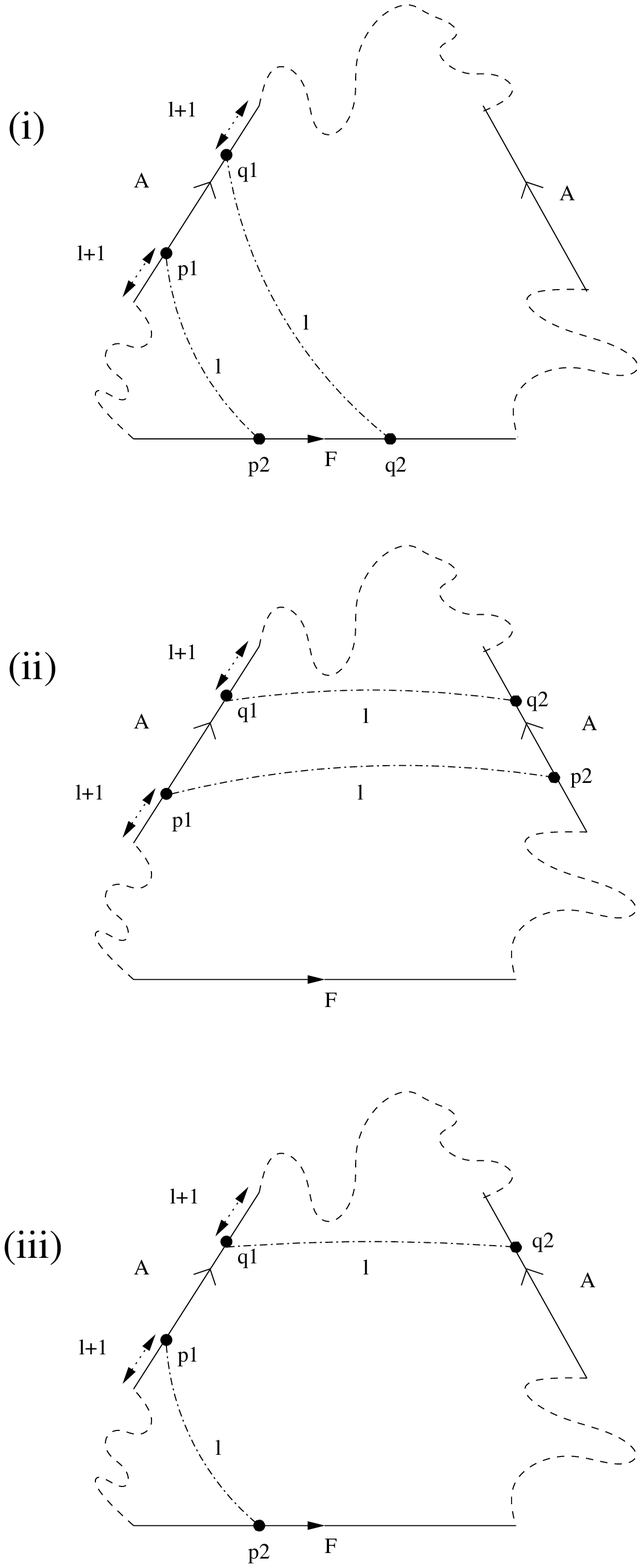}
\\
\vspace{0.38cm}                                                                  
\refstepcounter{figure}\label{l1b}                                             
Figure \thefigure                  
\end{center}
\end{figure}

{\flushleft{(i)}} Suppose that both $p_2$ and $q_2$ lie on $F$. By the triangle
inequality it immediately follows  that
\begin{eqnarray*}
d(\tau(\alpha),q_2) &\leq & d(\tau(\alpha),q_1)+d(q_1,q_2)\\
               &\leq & 3l+1.
\end{eqnarray*}
Therefore in this case the lemma holds.

{\flushleft{(ii)}} Suppose that both $p_2$ and $q_2$ lie on
$\alpha^{-1}$. Let  $q_3$ be
the vertex lying on $\alpha^{-1}$ such that
$d(\iota(\alpha^{-1}),q_3)=2l+1$.
We need the following lemma.
\begin{lemma}\label{x1x2}
Let $x_1$ and $x_2$  be any vertices on $\alpha$ and
$\alpha^{-1}$ respectively such that\\ \mbox{$d(x_1,x_2)\leq k$} for
some constant $k$. If $x_3$ is a vertex on $\alpha^{-1}$ such
that $d(\iota(\alpha^{-1}),x_3)=d(\tau(\alpha ),x_1)$
then $d(x_2,x_3)\leq k$.
\end{lemma}
\begin{proof}
The proof falls into the following two cases:
\begin{enumerate}
\item[(a)] $d(\iota(\alpha^{-1}),x_2)\leq d(\iota(\alpha^{-1}),x_3)=d(\tau(\alpha
  ),x_1)$
\item[(b)]  $d(\iota(\alpha^{-1}),x_2)> d(\iota(\alpha^{-1}),x_3)=d(\tau(\alpha
  ),x_1)$
\end{enumerate}
(a) By Lemma \ref{l1} $d(\tau(\alpha ),x_1)\leq
d(x_1,\iota(\alpha^{-1}))$ and from the hypothesis and the
triangle inequality  it follows that
\begin{equation}\label{alp1}
d(x_1,\iota(\alpha^{-1})) \leq  k+d(\iota(\alpha^{-1}),x_2).
\end{equation} 
Since $\alpha$ is a geodesic path we have  
\begin{equation}\label{alp2}
d(\tau(\alpha),x_1)=d(\iota(\alpha^{-1}),x_3) = d(\iota(
\alpha^{-1}),x_2) + d(x_2,x_3).
\end{equation}
It follows from equations (\ref{alp1}) and (\ref{alp2}) that $d(x_2,x_3)\leq k$.

{\flushleft{ (b)}}  By Lemma \ref{l1} $d(\iota(\alpha^{-1} ),x_2)\leq
d(x_2,\tau(\alpha ))$ and from the hypothesis and the
triangle inequality  it follows that
\begin{equation}\label{alp3}
d(x_2,\tau(\alpha)) \leq  k+d(\tau(\alpha),x_1).
\end{equation} 
Since $\alpha$ is a geodesic path we have  
\begin{equation}\label{alp4}
d(\tau(\alpha),x_1)=d(\iota(\alpha^{-1}),x_3) = d(\iota(
\alpha^{-1}),x_2) - d(x_2,x_3).
\end{equation}
It follows from equations (\ref{alp3}) and (\ref{alp4}) that $d(x_2,x_3)\leq k$.
Hence the lemma holds.
\end{proof}
  
Returning to case (ii), the lemma above implies that $d(q_2,q_3)\leq
l$. Thus $d(q_1,q_3)\leq 2l$. Similarly, if $p_3$ is the vertex on
$\alpha^{-1}$ such that $d(\tau(\alpha ^{-1}),p_3)=2l+1$, we can follow the same
argument to show that $d(p_1,p_3)\leq 2l$. 

Let the segment of $\alpha$ from $p_1$ to $q_1$ be labelled by
$\alpha _1$. Therefore the segment on $\alpha^{-1}$ from $q_3$ to $p_3$ is
labelled $\alpha _1^{-1}$. Let $s$ and $t$ be  geodesic paths from $p_1$ and
$q_1$ to $p_3$ and $q_3$ respectively. We have shown such paths to
have length at most $2l$. See Figure \ref{l1c}.
\begin{figure}[ht]
\begin{center}
\psfrag{A1}{{\scriptsize $\alpha _1$}}
\psfrag{s}{{\scriptsize $s$}}
\psfrag{t}{{\scriptsize $t$}}
\psfrag{F}{{\scriptsize $F$}}
\includegraphics[scale=0.5]{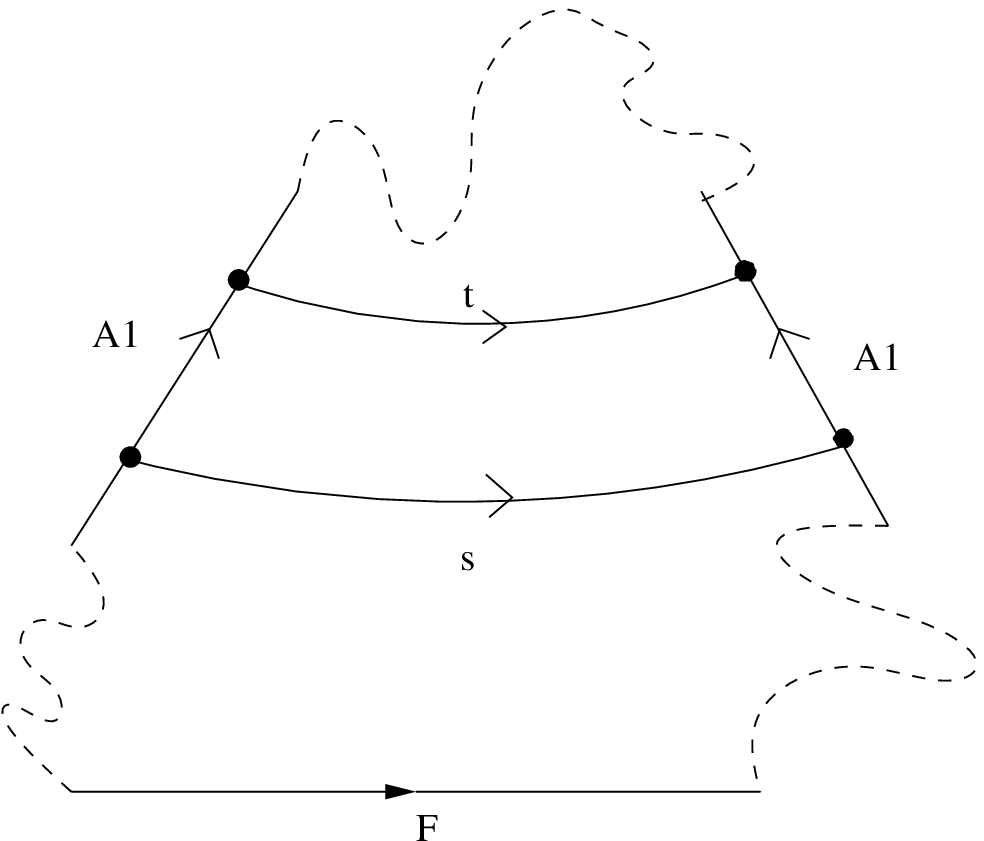}
\\
\vspace{0.38cm}                                                                  
\refstepcounter{figure}\label{l1c}                                             
Figure \thefigure                  
\end{center}
\end{figure} 
Since $\theta(W)$ was chosen to be of shortest length over all pairs
in  $\mathcal{F}$, we can
use Lemma \ref{lb} to show that 
\begin{eqnarray*}
|\alpha_1| &\leq & \frac{1}{2}(|s|+|t|)+M+1\\
      &\leq & \frac{1}{2}(2l+2l)+M+1\\
      &\leq & 2l+M+1.
\end{eqnarray*}
Therefore it follows that
\begin{eqnarray*}
|\alpha| & = & d(\iota(\alpha),p_1)+|\alpha_1|+d(q_1,\tau(\alpha))\\
    &\leq & 2l+1+2l+M+1+2l+1\\
    &\leq & 6l+M+3.
\end{eqnarray*}
We have a contradiction. Hence this case can't occur.

{\flushleft{(iii)}} Suppose that $p_2$ lies on $F$ and $q_2$ lies on
$\alpha^{-1}$. See Figure \ref{l1b}. Again let $q_3$ be the vertex on
$\alpha^{-1}$ such that $d(\iota(\alpha^{-1}),q_3)=2l+1$. As in case (ii) we can
show that $d(q_1,q_3)\leq 2l$. As before, by Lemma \ref{la} and Lemma \ref{l1},
for each vertex $u$ of $\alpha$ lying between $p_1$ and $q_1$, there exists a
vertex $v$ on $F\cup\alpha^{-1}$ such that $d(u,v)\leq l$. Let $u_1$ be the
first vertex along $\alpha$ which is within $l$ of a vertex $v_1$ on
$\alpha^{-1}$. By Lemma \ref{la} part $2$ vertex $u_1$ clearly lies
between $p_1$  and $q_1$ on $\alpha$. Let $v_2$ be the vertex on $\alpha^{-1}$
such that $d(\iota(\alpha^{-1}),v_2)=d(\tau(\alpha),u_1)$. By Lemma
\ref{x1x2} it follows that
$d(v_1,v_2)\leq l$. Thus $d(u_1,v_2)\leq 2l$.

Let the segment of $\alpha$ from $u_1$ to $q_1$ be labelled by
$\alpha_2$. Therefore the segment on $\alpha^{-1}$ from $q_3$ to $v_2$ is
labelled $\alpha_2^{-1}$. Let $s'$ and $t'$ be  geodesic paths from $p_1$ and
$u_1$ to $p_3$ and $v_2$ respectively. We have shown such paths to
have length at most $2l$. Again, since $\theta(W)$ was chosen to be of
shortest  length in $\mathcal{F}$, we can
use Lemma \ref{la} to show that 
\begin{eqnarray*}
|\alpha_2| &\leq & \frac{1}{2}(|s'|+|t'|)+M+1\\
      &\leq & \frac{1}{2}(2l+2l)+M+1\\
      &\leq & 2l+M+1.
\end{eqnarray*}
Let $u_2$ be the vertex on $\alpha$ such that
$d(\iota(\alpha),u_1)=d(\iota(\alpha),u_2)+1$. Since $u_1$ was chosen to be the
first vertex along $\alpha$ which was within $l$ of a vertex on
$\alpha^{-1}$, there exists a vertex $v_3$ on $F$ such that $d(u_2,v_3)\leq
l$. See Figure \ref{l1d}.
\begin{figure}[ht]
\begin{center}
\psfrag{A}{{\scriptsize $\alpha$}}
\psfrag{l}{{\scriptsize $\leq l$}}
\psfrag{x}{{\scriptsize $4l\hspace{-2pt}+\hspace{-2pt}M\hspace{-2pt}+\hspace{-2pt}3\hspace{-2pt}\geq$}}
\psfrag{v3}{{\scriptsize $v_3$}}
\psfrag{u2}{{\scriptsize $u_2$}}
\psfrag{F}{{\scriptsize $F$}}
\includegraphics[scale=0.5]{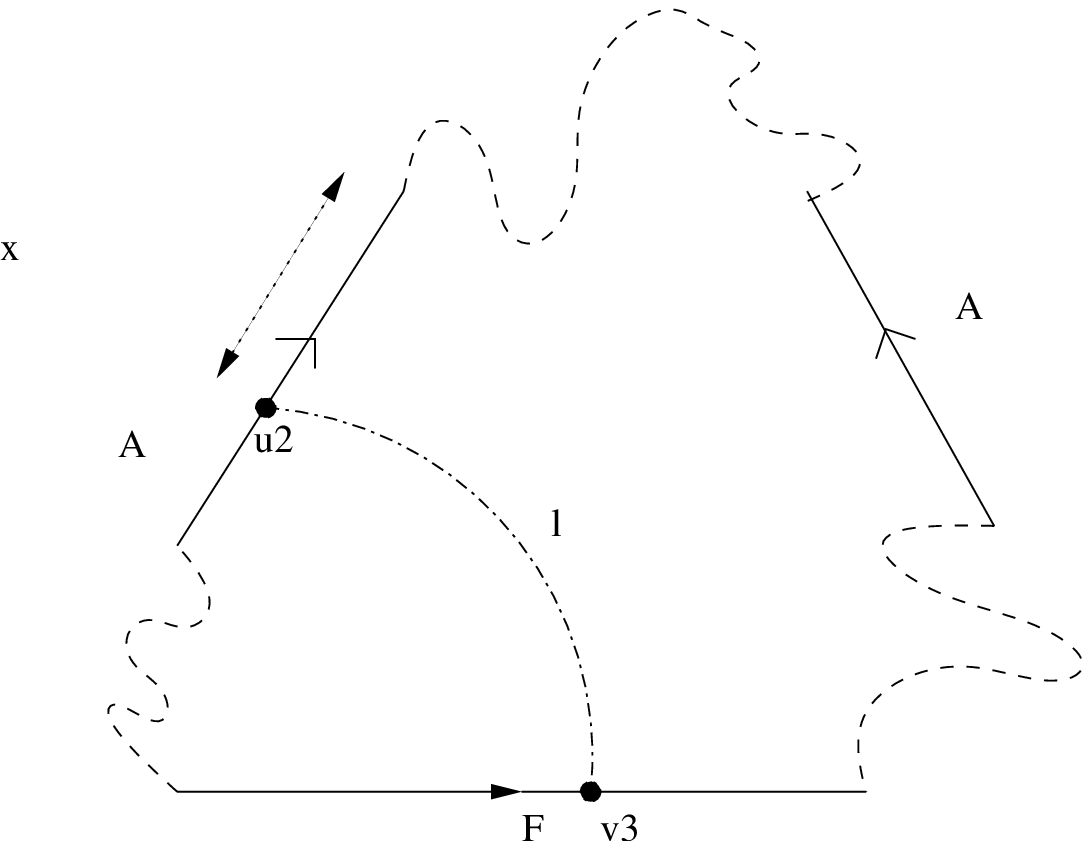}
\\
\vspace{0.38cm}                                                                  
\refstepcounter{figure}\label{l1d}                                             
Figure \thefigure                  
\end{center}
\end{figure} 
Clearly it follows that
\begin{eqnarray*}
d(\tau(\alpha),v_3) &\leq & d(\tau(\alpha),q_1)+|\alpha_2|+1+d(u_2,v_3)\\
               &\leq & 2l+1+2l+M+1+1+l\\
                &\leq & 5l+M+3.
\end{eqnarray*}
Hence the Lemma holds in all cases.
\end{proof}
Consider all letters of $\hat{W}$ which have labels of length greater than
$12l+M+4$ in $F(X)$. We shall call these the {\emph{long edges}} of
$\hat{W}$. All other letters shall be called {\emph{short edges}}. The
terminal vertex of each long edge in
$\Gamma_X(H)$ has each been shown, in the previous lemma,
 to be within $5l+M+3$ of some vertex on $F$. 
Let $B$ be a long edge of $\hat{W}$ which is not the first long edge
in the sequence of letters. Since
$\hat{W}$ is quadratic,  $B$
appears twice, once with exponent $1$ and once with exponent
$-1$. First we shall consider the appearance of $B$ with exponent $1$.
In the sequence of letters of $\hat{W}$, let $A^{\pm 1}$ be the long edge
before $B$ in the sequence such that no long edge appears between
$A^{\pm 1}$ and $B$ (note that $A^{\pm 1}$ could be $B^{-1}$). By Lemma
\ref{lsc} there exist vertices $u$ and $v$ on
$F$ such that $d(\tau(\theta(A^{\pm 1})),u),d(\tau(\theta(B)),v)\leq 5l+M+3$.
See  Figure \ref{apmb}.
\begin{figure}[ht]
\begin{center}
\psfrag{A}{{\scriptsize $\theta(A^{\pm 1})$}}
\psfrag{l}{{\scriptsize $\leq
    5l\hspace{-1pt}+\hspace{-1pt}M\hspace{-1pt}+\hspace{-1pt} 3$}}
\psfrag{B}{{\scriptsize $\theta(B)$}}
\psfrag{u}{{\scriptsize $u$}}
\psfrag{F}{{\scriptsize $F$}}
\psfrag{v}{{\scriptsize $v$}}
\psfrag{all short edges}{all short edges}
\includegraphics[scale=0.4]{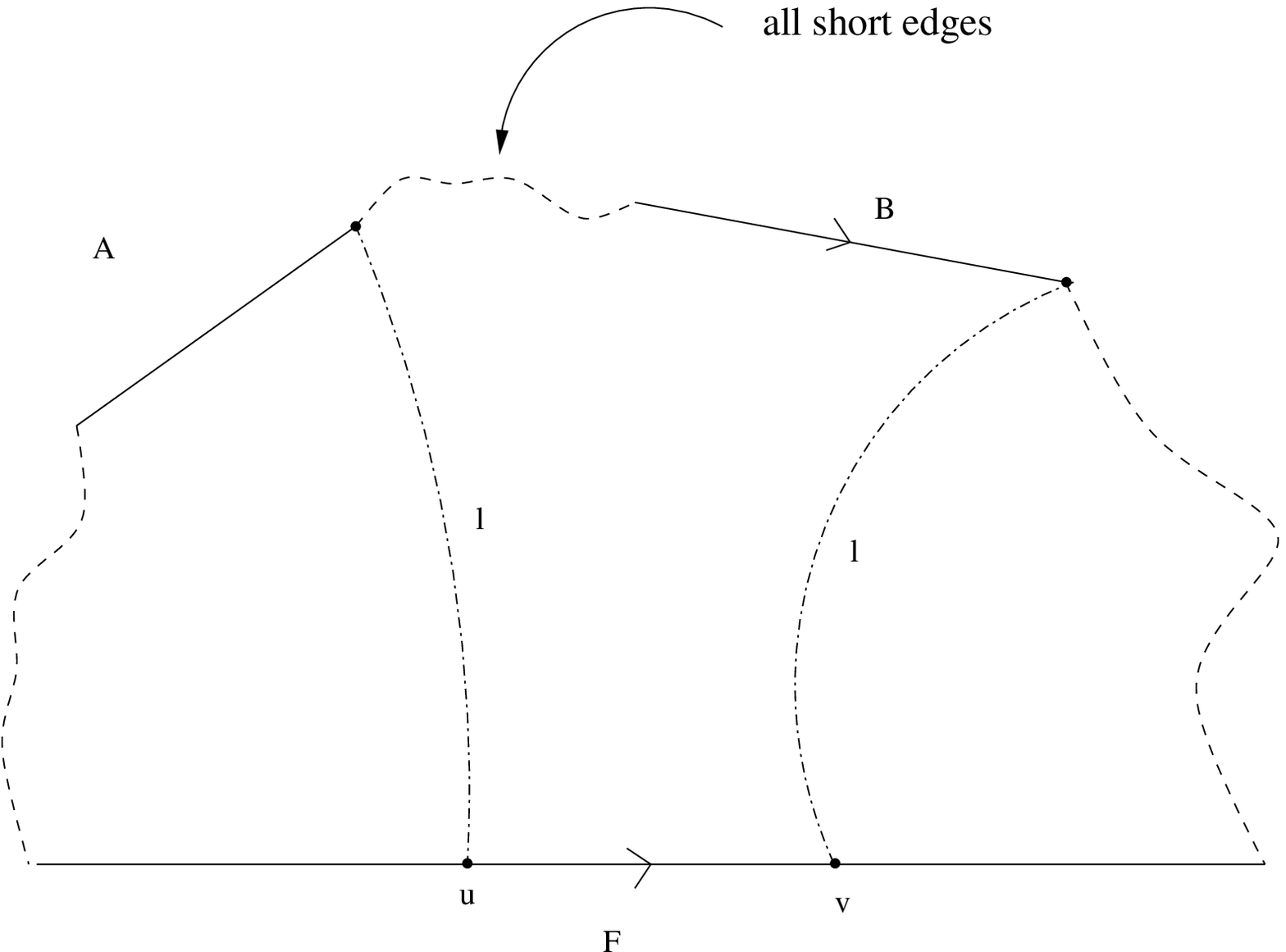}
\\
\vspace{0.38cm}                                                                  
\refstepcounter{figure}\label{apmb}                                             
Figure \thefigure                  
\end{center}
\end{figure} 
\begin{lemma}\label{chouv}
We can choose $u$ and $v$ such that $d(\iota(F),u)<d(\iota(F),v)$.
\end{lemma}  
\begin{proof}
Suppose that  $d(\iota(F),u)\geq d(\iota(F),v)$.
From the proof of Lemma \ref{lsc}, there exist vertices $u'$ and $v'$
lying on $\theta(A^{\pm 1})$ and $\theta(B)$ respectively such that we
have the following inequalities.
\begin{enumerate}\label{state12}
\item $d(u',u),d(v',v)\leq l$;
\item $2l+1\leq d(u',\tau(\theta(A^{\pm 1}))), d(v',\tau(\theta(B)))\leq 4l+M+3$.
\end{enumerate}
In this proof $u$ and $v$ were chosen in $\Gamma _X(H)$ from $u'$ and
$v'$ using part $1.$ of Lemma \ref{la}. Let $q$ be the closed path
labelled by the cyclic word $\theta(\hat{W})F^{-1}$ in $F(X)$ starting
at $\iota(\theta(A^{\pm 1}))$. Now we can
use  Lemma \ref{lac} to show
that there exists another vertex $v''$ on $q$ such that $d(v',v'')\leq 2l$ and
\begin{equation*}
d_q(\iota(\theta(A^{\pm 1})),u)>d_q(\iota(\theta(A^{\pm
  1})),v'')>d_q(\iota(\theta(A^{\pm 1})),u').
\end{equation*}
Since $d(v',\tau(B))\geq 2l+1$, $v''$ must lie on
$F\cup\theta(B^{-1})$ or Lemma
\ref{l1} would be violated. We need to consider two cases.

{\flushleft{{\large{{\bf{Case 1}} ($v''$ lies on $F$)}}}}\\
See Figure \ref{apmb2}.
\begin{figure}[ht]
\begin{center}
\psfrag{A}{{\scriptsize $\theta(A^{\pm 1})$}}
\psfrag{l}{{\scriptsize $\leq l$}}
\psfrag{2l}{{\scriptsize $\leq 2l$}}
\psfrag{B}{{\scriptsize $\theta(B)$}}
\psfrag{u}{{\scriptsize $u$}}
\psfrag{F}{{\scriptsize $F$}}
\psfrag{v}{{\scriptsize $v$}}
\psfrag{v2}{{\scriptsize $v''$}}
\psfrag{v1}{{\scriptsize $v'$}}
\psfrag{u1}{{\scriptsize $u'$}}
\includegraphics[scale=0.4]{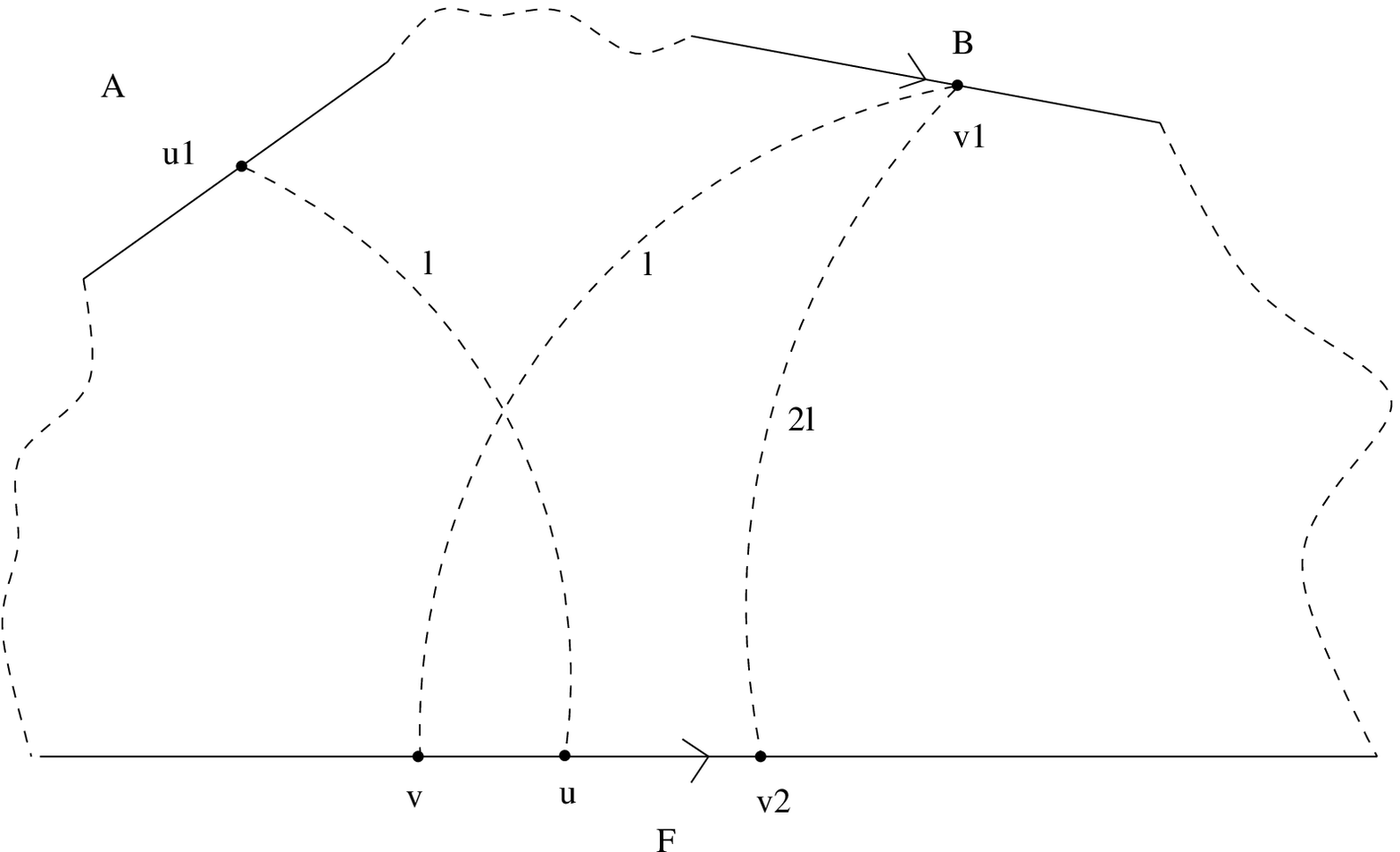}
\\
\vspace{0.38cm}                                                                  
\refstepcounter{figure}\label{apmb2}                                             
Figure \thefigure                  
\end{center}
\end{figure} 
Using the triangle inequality we have the following inequalities.
\begin{eqnarray}
d(v,u)+d(u,v'')&\leq  &3l;\\
d(u',v') &\leq & 2l+d(v,u);\\
d(u',v') &\leq & 3l+d(u,v'').
\end{eqnarray}
We can combine these to give $2d(u',v')\leq 5l+d(v,u)+d(u,v'')\leq 8l$ which
implies that $d(u',v')\leq 4l$.

We need to consider two possibilities. First
suppose that $A^{\pm 1}\neq B^{-1}$. Then by Lemma \ref{l1} it follows
that  $d(\iota(\theta(B)),v')\leq d(u',v')\leq
4l$ and from 2 above 
\begin{equation*}
|\theta(B)|=d(\iota(\theta(B)),v')+d(v',\tau(\theta(B)))\leq 4l+4l+M+3\leq 8l+M+3.
\end{equation*}  
Since $B$ is a long edge this is a contradiction. 

Now suppose that $A^{\pm 1}=B^{-1}$. We can now strengthen statement 2 above to
\begin{eqnarray}
2l+1 \leq d(u',\tau(\theta(B^{-1})))&\leq &4l+M+3\label{aeqb1}\\
\textrm{and}\qquad d(v',\tau(\theta(B)))& = &2l+1\label{aeqb2}.
\end{eqnarray}
We can choose $v'$ to be this vertex since this vertex is within $l$
of a vertex on $F\cup \theta(B^{-1})$(again see Lemmas \ref{la} and
\ref{l1}) and if it lies on $\theta(B^{-1})$  
we would have case (ii) of the proof of Lemma \ref{lsc},
which we have already shown cannot occur. 

Now consider the vertex $u''$ on $\theta(B^{-1})$ such that
$d(\iota(\theta(B^{-1})),u'')= d(\iota(\theta(B^{-1})),u')+1$. If $u$
was chosen by case (iii) of Lemma \ref{lsc} then $u''$ is within $l$ of
a vertex $v'''$ on $\theta(B)$, otherwise $u$ was chosen using case
(i) of Lemma \ref{lsc} and
$d(u'',\tau(\theta(B^{-1})))=2l$. First consider the latter case. Let
$E$ be the first letter after $B^{-1}$ in $\hat{W}$. By
Lemma \ref{l1}, a geodesic path from  $\tau(\theta(B))$ to
$\iota(\theta(E))=\tau(\theta(B^{-1}))$ has length greater than
$|\theta(B)|$. But by the triangle inequality
\begin{eqnarray*}
d(\tau(\theta(B)),\tau(\theta(B^{-1}))) &\leq &
d(\tau(\theta(B)),v')+d(v',u')+d(u',u'')+d(u'',\tau(\theta(B^{-1})))\\
       &\leq & 2l+1+4l+1+2l\\
       &\leq & 8l+2.
\end{eqnarray*}
This implies that $|\theta(B)|\leq 8l+2$, but $B$ is a long edge so
this is a contradiction. Therefore
assume that $u''$ is within $l$ of a vertex $v'''$ on
$\theta(B)$.  Let $p$ be the vertex on $\theta(B)$ such that
$d(\iota(\theta(B),p)=d(\tau(\theta(B^{-1}),u'')$. By Lemma
\ref{x1x2} $d(p,v''')\leq l$. Therefore, $d(p,u'')\leq 2l$.   Thus we
have all  the bounded distances shown in 
Figure \ref{apmb3}. 
\begin{figure}[ht]
\begin{center}
\psfrag{l}{{\scriptsize $\leq l$}}
\psfrag{2l}{{\scriptsize $\leq 2l$}}
\psfrag{4l}{{\scriptsize $\leq 4l$}}
\psfrag{B}{{\scriptsize $\theta(B)$}}
\psfrag{u2}{{\scriptsize $u''$}}
\psfrag{u}{{\scriptsize $u$}}
\psfrag{F}{{\scriptsize $F$}}
\psfrag{v}{{\scriptsize $v$}}
\psfrag{v3}{{\scriptsize $p$}}
\psfrag{v2}{{\scriptsize $v''$}}
\psfrag{v1}{{\scriptsize $v'$}}
\psfrag{u1}{{\scriptsize $u'$}}
\includegraphics[scale=0.5]{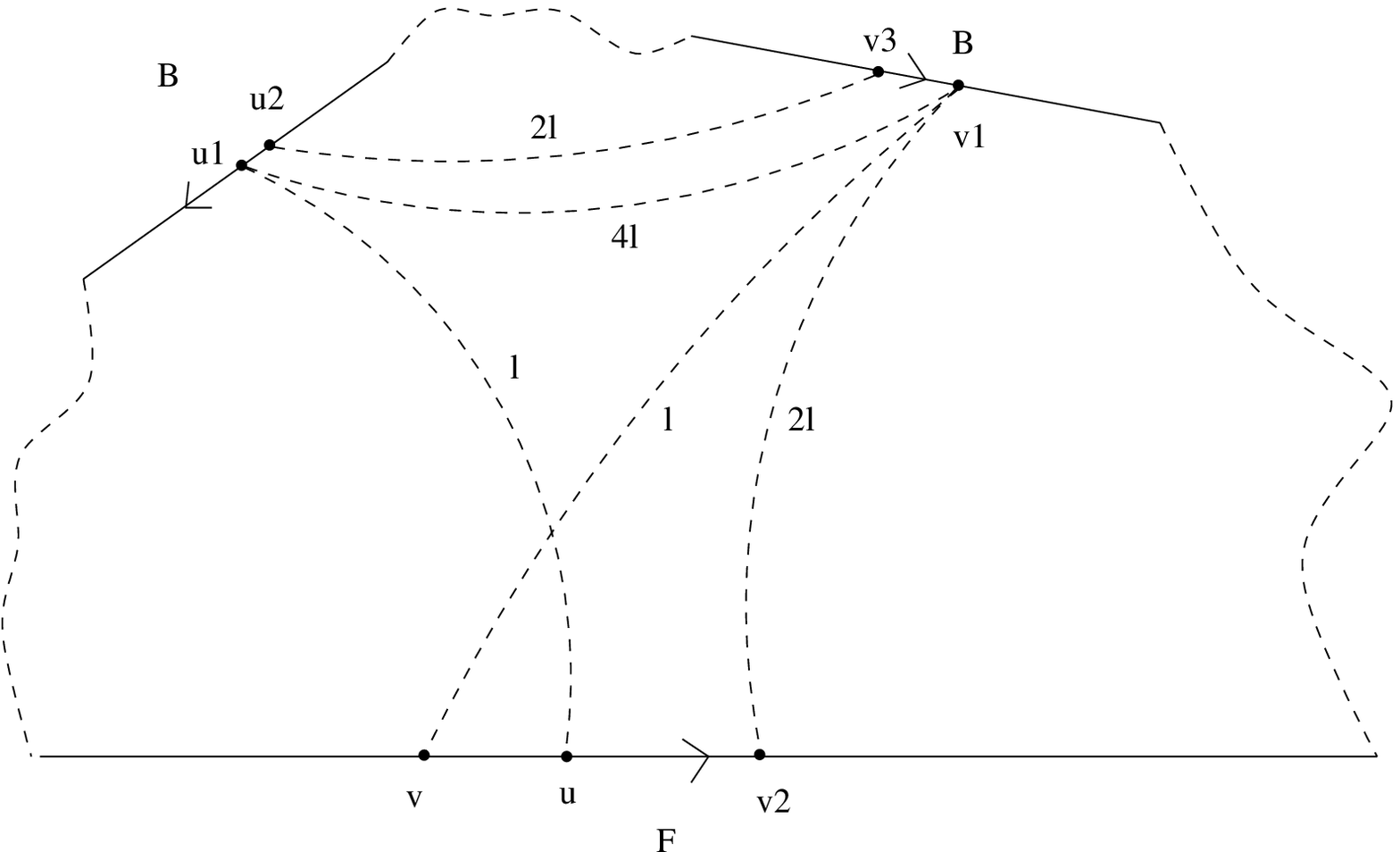}
\\
\vspace{0.38cm}                                                                  
\refstepcounter{figure}\label{apmb3}                                             
Figure \thefigure                  
\end{center}
\end{figure} 
From this it is clear that $d(v',p)\leq 4l+1+2l=6l+1$. Since
$\theta(B)$ is a geodesic path, it follows that
\begin{eqnarray*}
|\theta(B)| & = &d(\iota(\theta(B)),p)+d(p,v')+d(v',\tau(\theta(B)))\\
            & \leq & d(\tau(\theta(B^{-1})),u'')+6l+1+2l+1\\
            & \leq & 4l+M+2+6l+1+2l+1\\
            & \leq & 12l+M+4.
\end{eqnarray*}
But $B$ is a long edge so this cannot occur.
{\flushleft{{\large{{\bf{Case 2}} ($v''$ lies on $\theta(B^{-1})$)}}}}\\
We can see from Lemma \ref{lac} that this case falls into two subcases.
\begin{enumerate}
\item[(i)] $B^{-1}= A^{\pm 1}$ and
  $d(\iota(\theta(B^{-1})),v'')>d(\iota(\theta(B^{-1}),u')$ or
\item[(ii)] $B^{-1}$ lies after $B$ in the sequence of letters of
  $\hat{W}$.
\end{enumerate}
We first consider subcase (i). See Figure \ref{apmb5}.
\begin{figure}[ht]
\begin{center}
\psfrag{A}{{\scriptsize $\theta(A^{\pm 1})$}}
\psfrag{l}{{\scriptsize $\leq l$}}
\psfrag{2l}{{\scriptsize $\leq 2l$}}
\psfrag{B}{{\scriptsize $\theta(B)$}}
\psfrag{u}{{\scriptsize $u$}}
\psfrag{F}{{\scriptsize $F$}}
\psfrag{v}{{\scriptsize $v$}}
\psfrag{v2}{{\scriptsize $v''$}}
\psfrag{v1}{{\scriptsize $v'$}}
\psfrag{u1}{{\scriptsize $u'$}}
\includegraphics[scale=0.4]{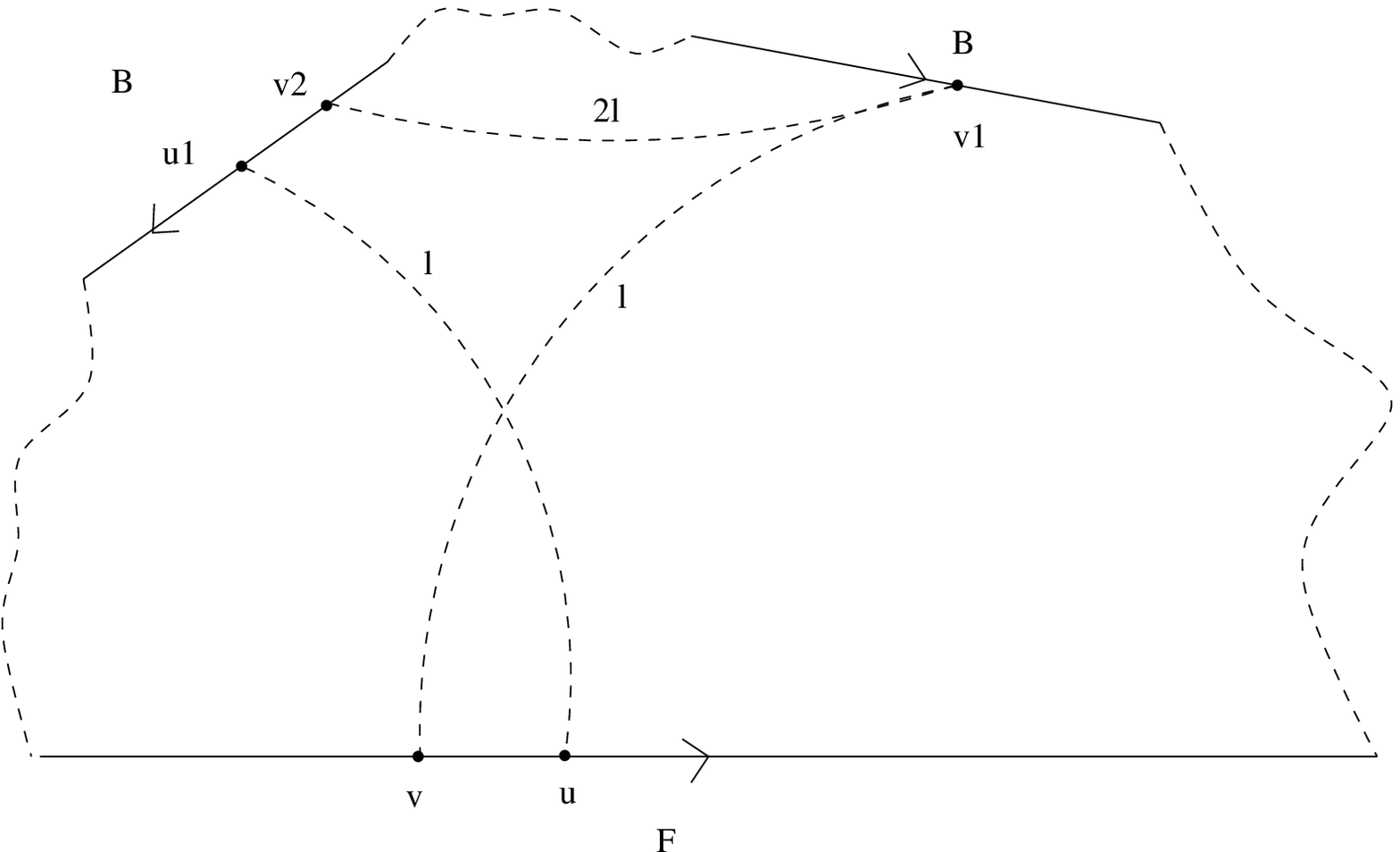}
\\
\vspace{0.38cm}                                                                  
\refstepcounter{figure}\label{apmb5}                                             
Figure \thefigure                  
\end{center}
\end{figure}
As in the previous case with $A^{\pm 1}=B^{-1}$ (see equations
(\ref{aeqb1}) and (\ref{aeqb2})) we can strengthen statement 2 on page
\pageref{state12} to
\begin{eqnarray*}
2l+1 \leq d(u',\tau(\theta(B^{-1})))&\leq &4l+M+3\\
\textrm{and}\qquad d(v',\tau(\theta(B)))& = &2l+1.
\end{eqnarray*} 
It follows that $d(v'',\tau(\theta(B^{-1}))\leq 4l+M+3$. Consider a
geodesic path $w$ from $v'$ to $\tau(\theta(B^{-1}))$. Lemma \ref{l1}
implies that $|w|\geq d(\iota(\theta(B)),v')$ and by the triangle
inequality we have that
\begin{eqnarray*}
|w| &\leq & d(v',v'')+d(v'',\tau(\theta(B^{-1})))\\
    & \leq & 2l+4l+M+3\\
    & = & 6l+M+3.
\end{eqnarray*}
It follows that
\begin{eqnarray*}
|\theta(B)| & = & d(\iota(\theta(B)),v')+d(v',\tau(\theta(B)))\\
            & \leq & 6l+M+3+2l+1\\
            & \leq & 8l+M+4.
\end{eqnarray*}
But $B$ is a long letter so we have a contradiction. Therefore this
subcase can't occur.

Now consider (ii). See Figure  \ref{apmb6}.
\begin{figure}[ht]
\begin{center}
\psfrag{A}{{\scriptsize $\theta(A^{\pm 1})$}}
\psfrag{l}{{\scriptsize $\leq l$}}
\psfrag{2l}{{\scriptsize $\leq 2l$}}
\psfrag{B}{{\scriptsize $\theta(B)$}}
\psfrag{u}{{\scriptsize $u$}}
\psfrag{F}{{\scriptsize $F$}}
\psfrag{v}{{\scriptsize $v$}}
\psfrag{v2}{{\scriptsize $v''$}}
\psfrag{v1}{{\scriptsize $v'$}}
\psfrag{u1}{{\scriptsize $u'$}}
\includegraphics[scale=0.4]{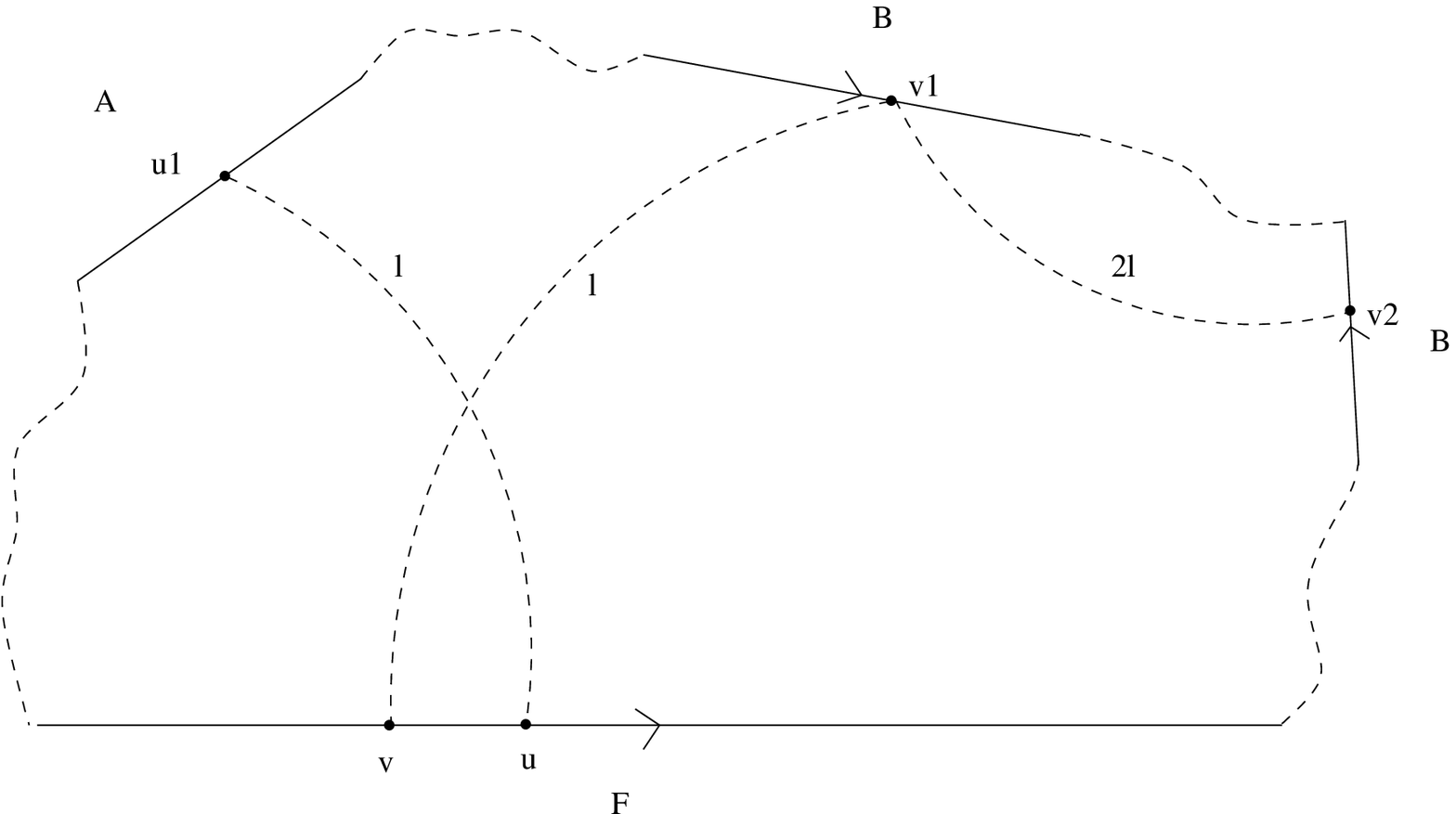}
\\
\vspace{0.38cm}                                                                  
\refstepcounter{figure}\label{apmb6}                                             
Figure \thefigure                  
\end{center}
\end{figure}
Consider the vertex $p$ which lies on $\theta(B)$ such that
$d(\iota(\theta(B)),p)=2l+1$. Let $q'$ be the closed path labelled by
the cyclic word $\theta(\hat{W})^{-1}F$ in $F(X)$ starting at
$\tau(\theta(B))$. Lemma \ref{la} part 1 and part 2 imply that there
is a vertex $p'$ on the closed path $q'$ such that
\begin{eqnarray}
d(p,p') & \leq & l \textrm{ and}\nonumber \\
d_{q'}(\tau(\theta(B)),v) & > &
d_{q'}(\tau(\theta(B)),p').\label{case2eq}
\end{eqnarray}  
It follows from Lemma \ref{l1} that $p'$ must lie on $F$ (It can't lie
on $\theta(B^{-1})$ or equation (\ref{case2eq}) would not hold). This
also means that we have part 3 of Lemma \ref{la}, i.e.
\begin{equation}
d(p',v)=d(p,v').\label{case2eq2}
\end{equation}
See Figure \ref{ampb7}.
\begin{figure}[ht]
\begin{center}
\psfrag{A}{{\scriptsize $\theta(A^{\pm 1})$}}
\psfrag{l}{{\scriptsize $\leq l$}}
\psfrag{2l}{{\scriptsize $\leq 2l$}}
\psfrag{B}{{\scriptsize $\theta(B)$}}
\psfrag{u}{{\scriptsize $u$}}
\psfrag{F}{{\scriptsize $F$}}
\psfrag{v}{{\scriptsize $v$}}
\psfrag{v2}{{\scriptsize $v''$}}
\psfrag{v1}{{\scriptsize $v'$}}
\psfrag{u1}{{\scriptsize $u'$}}
\psfrag{p}{{\scriptsize $p$}}
\psfrag{p1}{{\scriptsize $p'$}}
\includegraphics[scale=0.4]{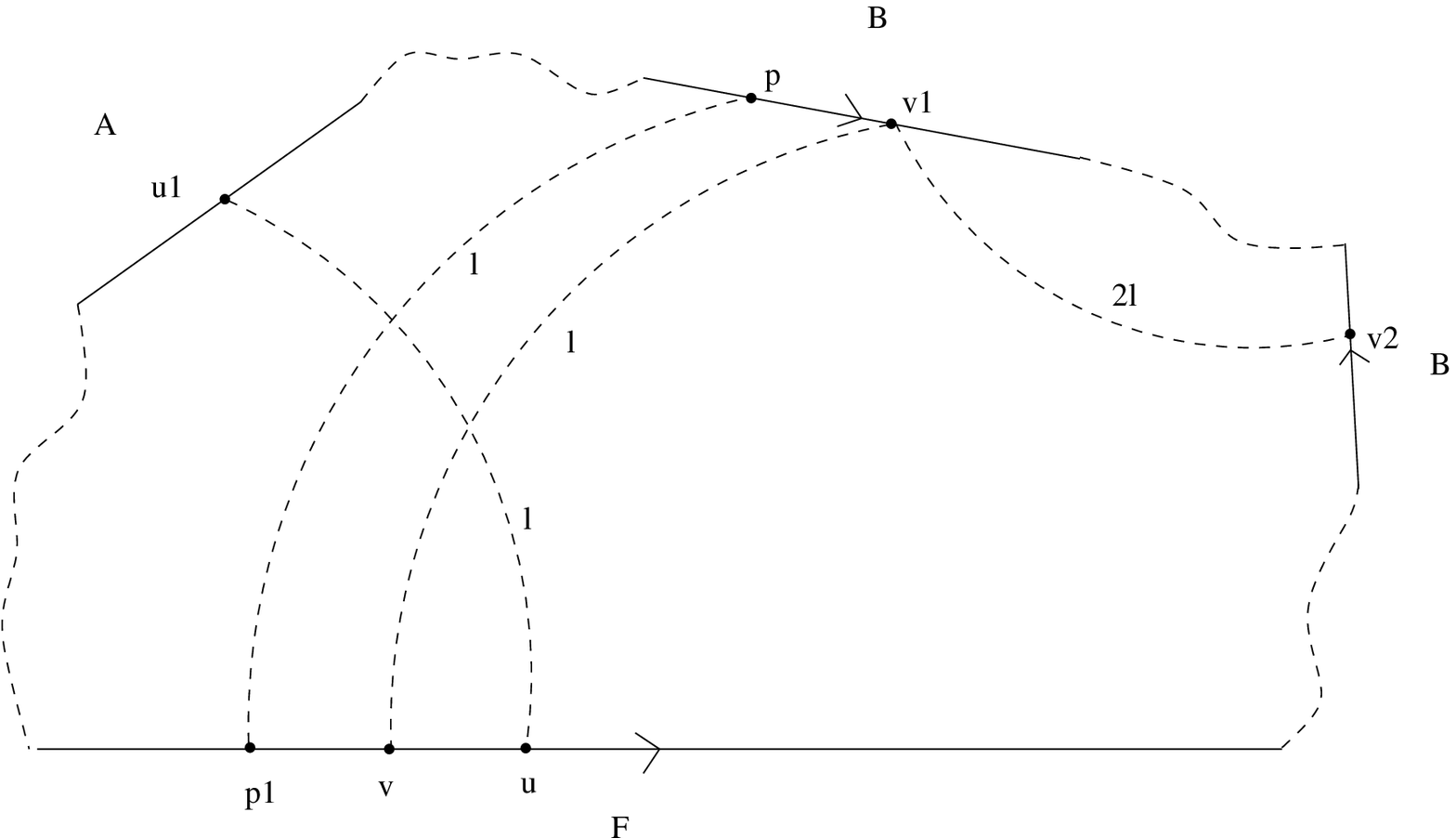}
\\
\vspace{0.38cm}                                                                  
\refstepcounter{figure}\label{ampb7}                                             
Figure \thefigure                  
\end{center}
\end{figure}
Now consider the closed path $q$ again. Lemma \ref{lac} implies that
there exists a vertex $p''$ on $q$ such that $d(p,p'')\leq 2l$ and
\begin{equation*}
d_q(\iota(\theta(A^{pm 1})),u)>d_q(\iota(\theta(A^{pm 1})),p'')>  
d_q(\iota(\theta(A^{pm 1})),u').
\end{equation*}
Since $d(p,\iota(\theta(B)))=2l+1$ and $d(p,\tau(\theta(B)))>2l+1$, it
follows from Lemma \ref{l1} that $p''$ lies on either $F$ or
$\theta(B^{-1})$. First consider the case where $p''$ lies on $F$. We
have the bounded distances as shown in Figure \ref{ampb8}.
\begin{figure}[ht]
\begin{center}
\psfrag{A}{{\scriptsize $\theta(A^{\pm 1})$}}
\psfrag{l}{{\scriptsize $\leq l$}}
\psfrag{2l}{{\scriptsize $\leq 2l$}}
\psfrag{B}{{\scriptsize $\theta(B)$}}
\psfrag{u}{{\scriptsize $u$}}
\psfrag{F}{{\scriptsize $F$}}
\psfrag{v}{{\scriptsize $v$}}
\psfrag{v2}{{\scriptsize $v''$}}
\psfrag{v1}{{\scriptsize $v'$}}
\psfrag{u1}{{\scriptsize $u'$}}
\psfrag{p}{{\scriptsize $p$}}
\psfrag{p1}{{\scriptsize $p'$}}
\psfrag{p2}{{\scriptsize $p''$}}
\includegraphics[scale=0.4]{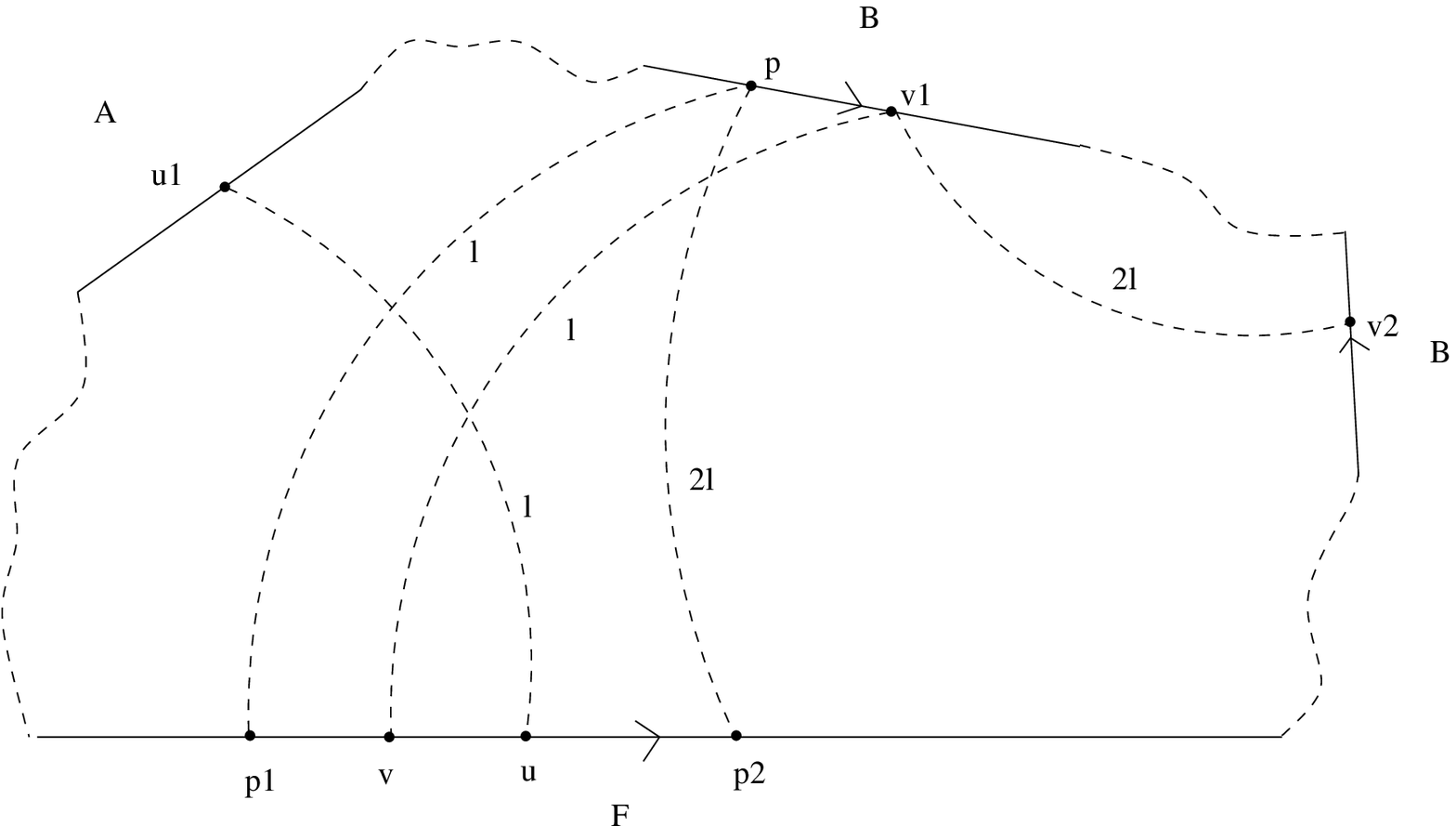}
\\
\vspace{0.38cm}                                                                  
\refstepcounter{figure}\label{ampb8}                                             
Figure \thefigure                  
\end{center}
\end{figure}
Using the triangle inequality we have the following equations.
\begin{eqnarray*}
d(p',v)+d(v,u)+d(u,p'') & \leq & 3l,\\
d(u',v') & \leq & d(v,u)+2l,\\
d(u',v') &\leq & d(u,p'')+3l+d(p,v').
\end{eqnarray*}
It follows from these equations and equation (\ref{case2eq2}) that
\begin{eqnarray*}
2d(u',v') & \leq & d(u,p'')+3l+d(p,v')+d(v,u)+2l\\
          & = & d(u,p'')+3l+d(p',v)+d(v,u)+2l\\
          & \leq & 8l\\
\implies\qquad d(u',v') &\leq & 4l.
\end{eqnarray*}
Lemma \ref{l1} implies that $d(v',\iota(\theta(B)))\leq d(v',u')\leq
4l$. Thus 
\begin{equation*}
|\theta(B)|=d(\iota(\theta(B)),v')+d(v',\tau(\theta(B)))\leq 4l +4l+M+3.
\end{equation*}
But $B$ is a long letter so we have a contradiction.

Finally we need to consider the case where $p''$ lies on
$\theta(B^{-1})$. We have the bounded distances as shown in Figure \ref{ampb9}.
\begin{figure}[ht]
\begin{center}
\psfrag{A}{{\scriptsize $\theta(A^{\pm 1})$}}
\psfrag{l}{{\scriptsize $\leq l$}}
\psfrag{2l}{{\scriptsize $\leq 2l$}}
\psfrag{B}{{\scriptsize $\theta(B)$}}
\psfrag{u}{{\scriptsize $u$}}
\psfrag{F}{{\scriptsize $F$}}
\psfrag{v}{{\scriptsize $v$}}
\psfrag{v2}{{\scriptsize $v''$}}
\psfrag{v1}{{\scriptsize $v'$}}
\psfrag{u1}{{\scriptsize $u'$}}
\psfrag{p}{{\scriptsize $p$}}
\psfrag{p1}{{\scriptsize $p'$}}
\psfrag{p2}{{\scriptsize $p''$}}
\includegraphics[scale=0.4]{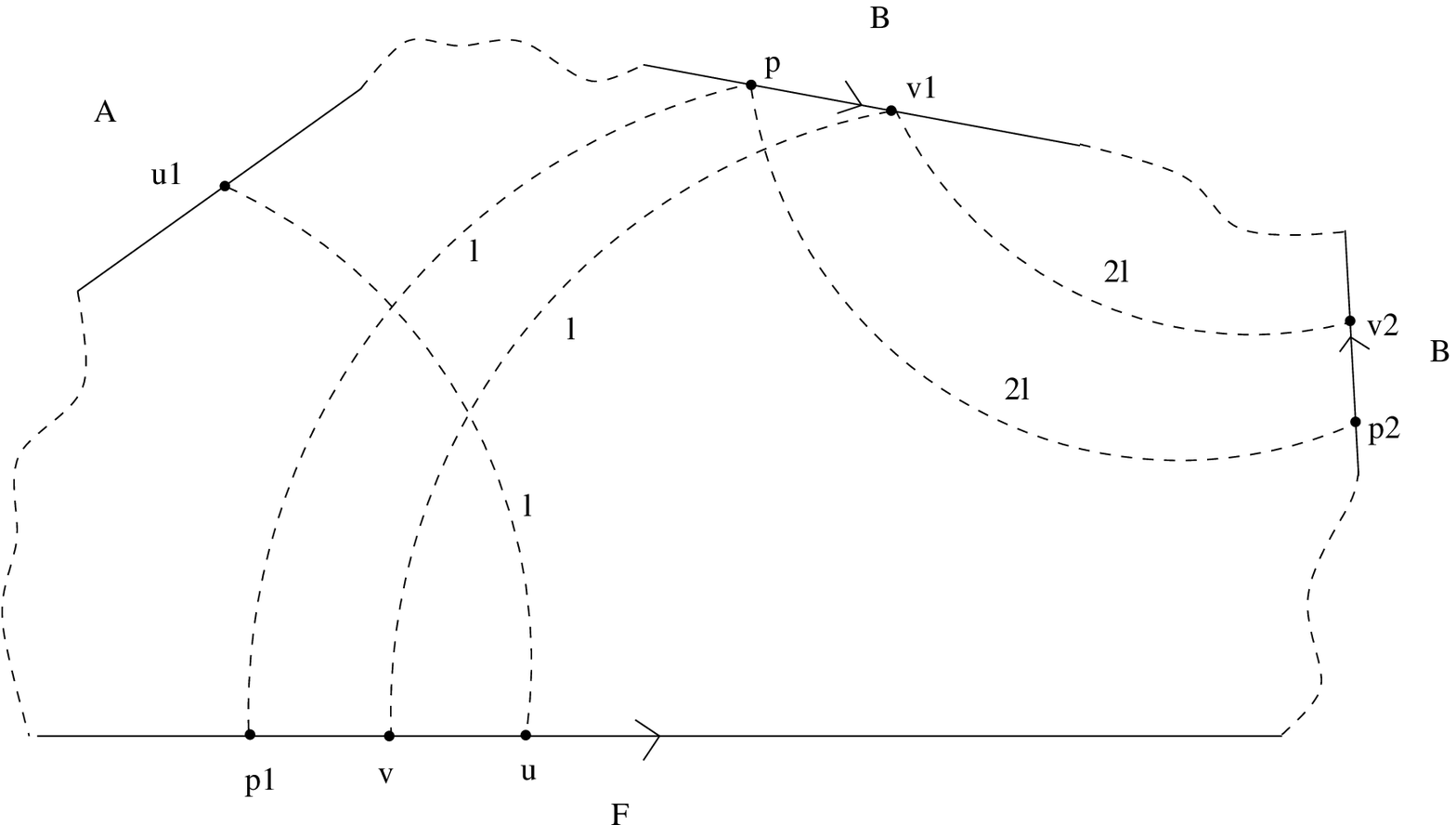}
\\
\vspace{0.38cm}                                                                  
\refstepcounter{figure}\label{ampb9}                                             
Figure \thefigure                  
\end{center}
\end{figure}
We know by the choice of vertex $p$ that
$d(\iota(\theta(B)),p)=2l+1$. The vertex  $v'$ was chosen from either
case (i) or case (iii) of Lemma \ref{lsc}. If it was chosen by case (i)
then $d(v',\tau(\theta(B)))=2l+1$. If it was chosen by case (iii) then
there exists a vertex $v'''$ on $\theta(B)$ such that
$d(v''',\tau(\theta(B)))=2l+1$ and $v'''$ is within a distance $l$ of
a vertex on $\theta(B^{-1})$. Either way there is vertex $x$ which is on the
$\theta(B)$ at a distance $2l+1$ from $\tau(\theta(B))$ which is within $2l$
of a vertex $x'$ on $\theta(B^{-1})$. See Figure \ref{xyx1y1}. 
\begin{figure}[ht]
\begin{center}
\psfrag{2l+1}{{\scriptsize $2l+1$}}
\psfrag{2l}{{\scriptsize $\leq 2l$}}
\psfrag{B}{{\scriptsize $\theta(B)$}}
\psfrag{x}{{\scriptsize $x$}}
\psfrag{x1}{{\scriptsize $x'$}}
\psfrag{y}{{\scriptsize $p$}}
\psfrag{y1}{{\scriptsize $p''$}}
\includegraphics[scale=0.4]{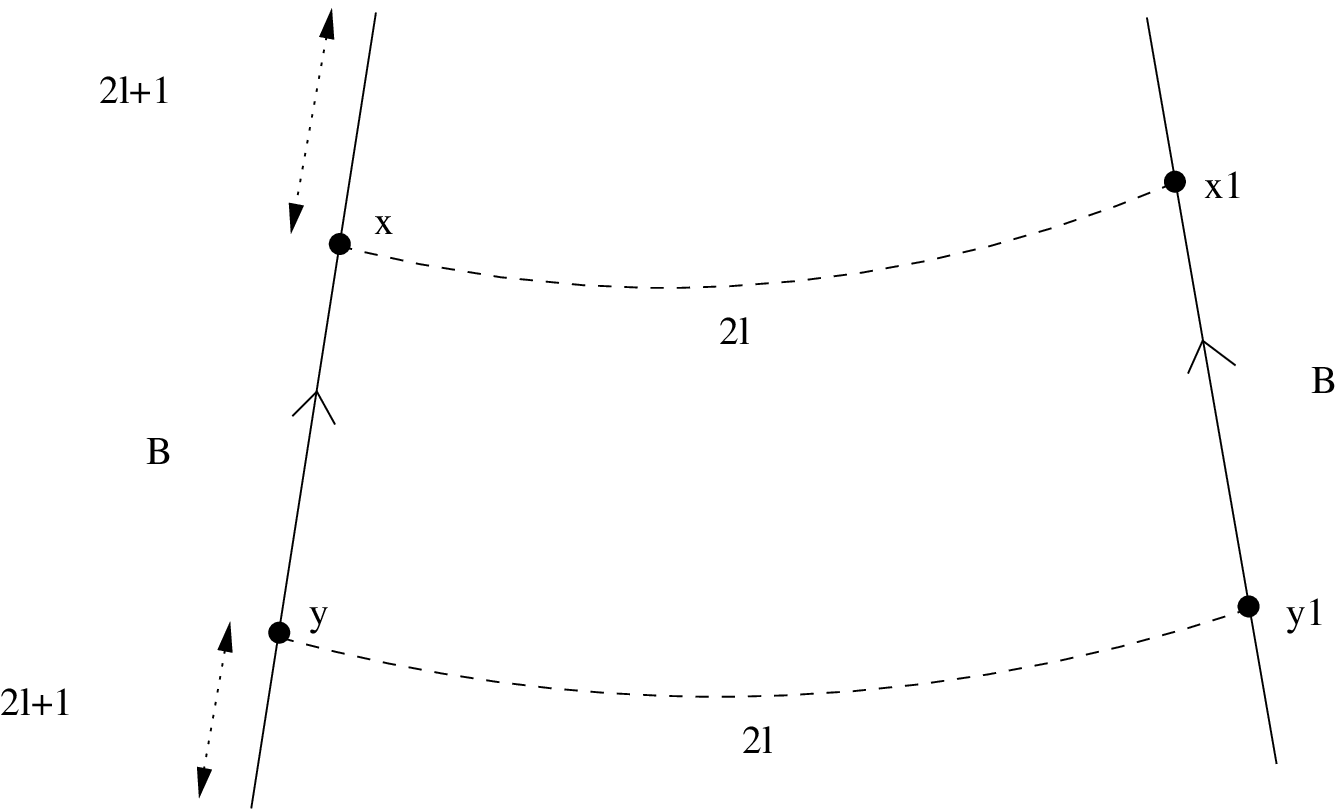}
\\
\vspace{0.38cm}                                                                  
\refstepcounter{figure}\label{xyx1y1}                                             
Figure \thefigure                  
\end{center}
\end{figure}
It follows from Lemma \ref{x1x2} that there exist vertices $x''$ and
$p'''$ lying on $\theta(B^{-1})$ such that we have the following
equations.
\begin{eqnarray}
d(x,\tau(\theta(B))) & = &d(x'',\iota(\theta(B^{-1}))),\nonumber\\
d(p,\tau(\theta(B))) & =& d(p''',\iota(\theta(B^{-1}))),\nonumber\\
d(x',x'') &\leq & 2l, \label{x'x''}\\
d(p'',p''') &\leq & 2l.\label{p''p'''}
\end{eqnarray} 
Equations (\ref{x'x''}) and (\ref{p''p'''}) imply that $d(x,x'')\leq
4l$ and $d(p,p''')\leq 4l$. See Figure \ref{xyx2y2}. 
\begin{figure}[ht]
\begin{center}
\psfrag{2l+1}{{\scriptsize $2l+1$}}
\psfrag{2l}{{\scriptsize $\leq 4l$}}
\psfrag{B}{{\scriptsize $\theta(B)$}}
\psfrag{x}{{\scriptsize $x$}}
\psfrag{x1}{{\scriptsize $x'$}}
\psfrag{y}{{\scriptsize $p$}}
\psfrag{y1}{{\scriptsize $p'''$}}
\includegraphics[scale=0.4]{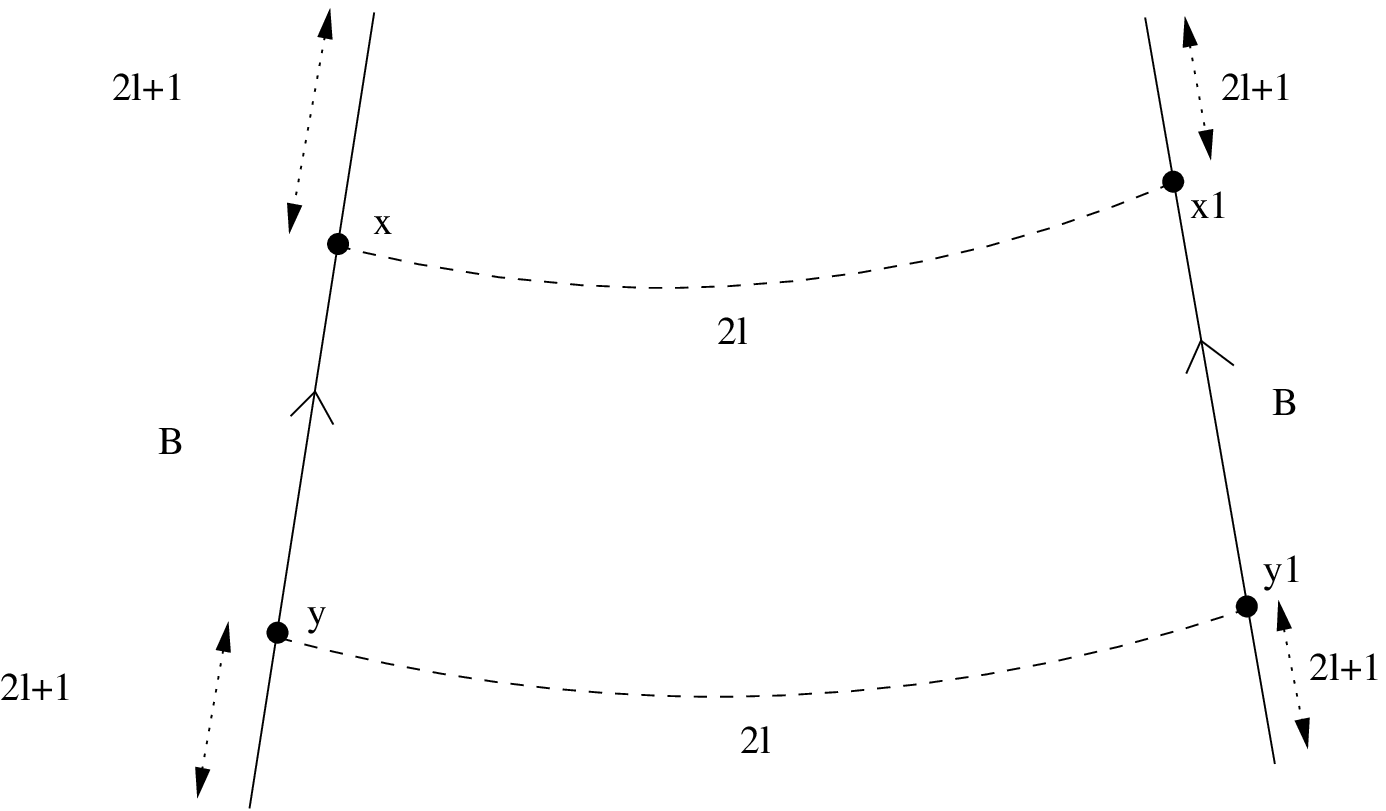}
\\
\vspace{0.38cm}                                                                  
\refstepcounter{figure}\label{xyx2y2}                                             
Figure \thefigure                  
\end{center}
\end{figure}
By Lemma \ref{lb} we have that
\begin{eqnarray*}
d(x,p)&\leq& \frac{1}{2}(4l+4l)+M+1\\
      &=& 4l+M+1.
\end{eqnarray*}
It follows that $|\theta(B)|\leq 2l+1+2l+1+4l+M+1\leq 8l+M+3$. But $B$
is a long letter so we have a contradiction. Therefore this case can't
occur.

Hence we have the required result.
\end{proof}
We shall label the segment of $F$ between $u$ and $v$ by $B_1$ and a
geodesic path from $\tau(\theta(B))$ to $v$ by $b_1$. If we
were considering $B^{-1}$ we would label the appropriate paths $B_2$
and $b_2$ respectively. Remember
the last letter of $\hat{W}$ was chosen to be long, thus if $B$ is
the last letter then $v=\tau(\theta(B))$ and $|b_1|=0$. We shall do
this for every long edge in $\hat{W}$.
\begin{example}
Suppose that $\hat{W}=ABCA^{-1}B^{-1}C^{-1}$ where $A$ and $C$ are
long edges and $B$ is a short edge. Then we label paths in
$\Gamma_X(H)$ as shown in Figure \ref{exABC}.
\begin{figure}[ht]
\begin{center}
\psfrag{A}{{\scriptsize $\theta(A)$}}
\psfrag{C}{{\scriptsize $\theta(C)$}}
\psfrag{B}{{\scriptsize $\theta(B)$}}
\psfrag{a1}{{\scriptsize $a_1$}}
\psfrag{c1}{{\scriptsize $c_1$}}
\psfrag{a2}{{\scriptsize $a_2$}}
\psfrag{A1}{{\scriptsize $A_1$}}
\psfrag{C1}{{\scriptsize $C_1$}}
\psfrag{A2}{{\scriptsize $A_2$}}
\psfrag{C2}{{\scriptsize $C_2$}}
\includegraphics[scale=0.45]{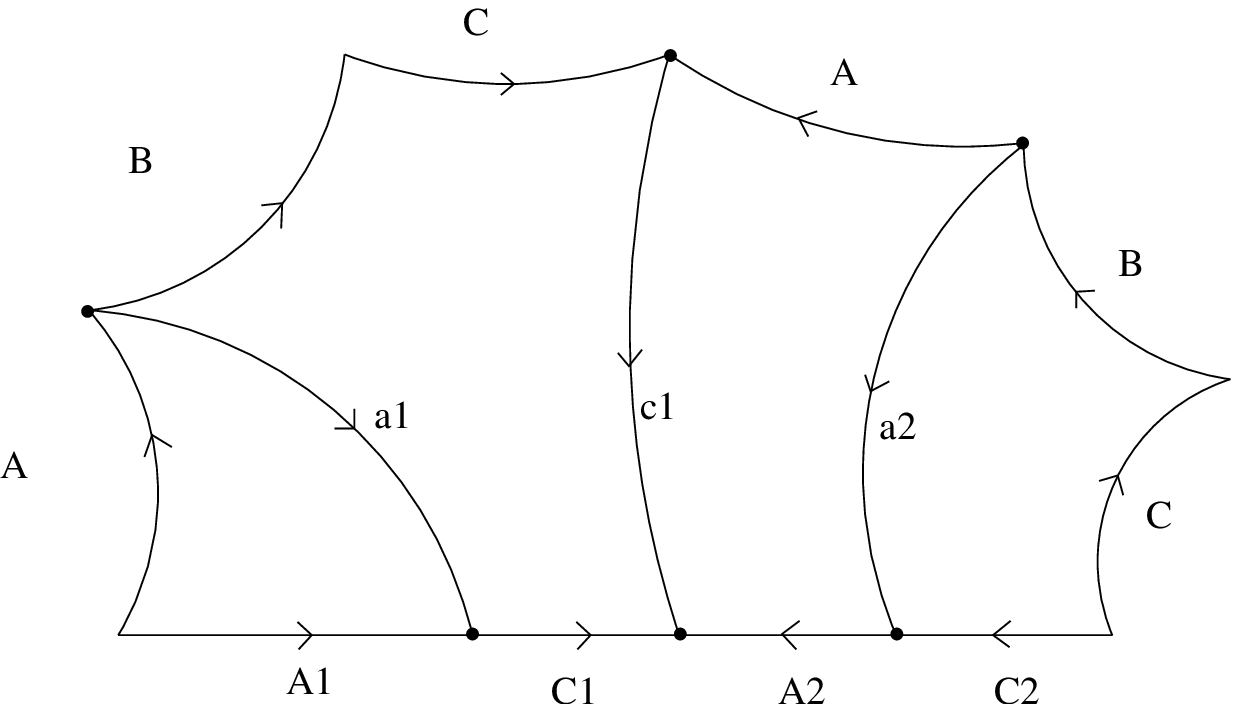}
\\
\vspace{0.38cm}                                                                  
\refstepcounter{figure}\label{exABC}                                             
Figure \thefigure                  
\end{center}
\end{figure}  
Where $|a_1|$, $|a_2|$ and $|c_1|$ are at most $5l+M+3$ and $|c_2|=0$.
\end{example} 
Let us write $\hat{W}$ around the boundary of a disc (i.e.~divide the
boundary up into $|\hat{W}|$ segments, assigning a letter to each) and identify the
long edges, respecting orientation. We obtain a surface $S$ of genus
$k\leq n$ with $Q$ holes. The boundary of the disc becomes a graph on
this surface, we shall denote this graph $\Gamma_S$. This graph
consists of  short edges all of which are written
around the boundary components and long edges all of which are properly embedded on
the surface. 

Separate the cyclic words in $\mathcal{A}^{\pm 1}$ written
around the $Q$ boundary components into $p$ disjoint sets,
$B_1,\ldots, B_p$, with $p$ as large as possible as follows. Each set $B_i$
contains some $t_i$-tuple of cyclic words $(W_1^i,\ldots ,W_{t_i}^i)$, where
$W_j^i$ is the path written around a boundary component of $S$,  such that if a
small edge
 $A$ of $\hat{W}$ appears in $W_j^i$, for some $j=1,\ldots ,t_i$
then  $A^{-1}$ appears in $W_k^i$ for some $k=1,\ldots ,t_i$. We shall
define the
genus of $B_i$ to be equal to  $g_i$, for $i=1,\ldots ,p$ , where
$g_i-t_i+1=genus_H(\theta(W_1^i),\theta(W_2^i),\ldots ,\theta(W_{t_i}^i))$.
\begin{lemma}\label{gi} 
Let $g=\sum _{i=1}^{p} g_i$.
For each $B_i$ containing a $t_i$-tuple $(W_1^i,\ldots
,W_{t_i}^i)$ we have
\begin{equation*}
genus_H(\theta(W_1^i),\theta(W_2^i),\ldots ,\theta(W_{t_i}^i))=genus_{F(\mathcal{A})}(W_1^i,\ldots
  ,W_{t_i}^i),\quad \textrm{for}\,\,i=1,\ldots ,p.
\end{equation*}
It follows that $g+k=n$.
\end{lemma}  
\begin{proof}
Suppose that this is not the case. Clearly the genus of $(W_1^i,\ldots
,W_{t_i}^i)$ in
$F(\mathcal{A})$, which we shall denote $h_i-t_i+1$, is at least $g_i-t_i+1$ for
all $i=1,\ldots ,p$.  
Therefore assume that $h_j>g_j$ for some
$1\leq j\leq p$. Now, since $W$ has
genus  $n$ in
$F(\mathcal{A})$, if we identify all
the  short edges on the genus $k$ surface $S$, respecting orientation,
we obtain a closed compact surface of genus $n$.  Therefore,
\begin{eqnarray*}
n &=&k+\sum_{i=1}^{p}h_i\\
& > &k+\sum_{i=1}^{p}g_i.
\end{eqnarray*}
But this implies that $genus_H(h)=genus_H(\theta(W))\leq
k+\sum_{i=1}^{p}g_i<n$, a contradiction. Thus $g_i=h_i$ for all
$i=1,\ldots ,p$  and
$n=k+\sum g_i =k+g$.  Hence the lemma holds.   
\end{proof}
Consider the surface $S$ with $Q$ holes with the embedded graph
$\Gamma_S$ consisting of  long and
short edges. If we paste a disc onto each of the $Q$
boundary components and contract the cyclic word of short edges to a
point, we obtain a graph of genus $k$ consisting  of long edges only
on a closed  compact surface
of genus $k$. The orientable word $U$ associated with this graph is
obviously the word obtained by setting all the short edges of $W$ to
$1$. As usual we denote this graph $\Gamma_U$. We should note that, by
the
way the
geodesic path labelled by $F$ in $\Gamma_X(H)$ has been cut up into
segments,  the cyclic sequence of letters of $U$ gives  the
cyclic sequence of segments of $F$ by replacing a letter $E$ of
$U$ with $E_1$ for its occurrence with exponent $1$ and $E_2^{-1}$ for
its occurrence with exponent $-1$. 
We now show that an extension over $H$  may be carried out on the
orientable word $U$, so that part $2$ of the
theorem holds.\label{edw} 
\subsection{Extension of $U$}\label{stext}
First we shall construct $p+1$ sets, $C_0,C_1,\ldots ,C_p$,  of cyclic
words in $F(X)$  such that each vertex of
$\Gamma_U$ will be extended by a unique element from one of these sets.
Consider a cyclic word $W_j^i\in B_i$, $1\leq j\leq t_i$ and $1\leq
i\leq p$, in the graph $\Gamma_S$. Thus, if we are thinking of
$\Gamma_S$ as being embedded in $S$, $W_j^i$ is a word in the short
edges written around a boundary component of the surface. Suppose that there
are $d$ end points of long edges lying on this boundary component. Let $E^1,\ldots
,E^d$ be the long edges. See Figure \ref{eandws}.
 \begin{figure}[ht]
\begin{center}
\psfrag{E1}{{\scriptsize $E^1$}}
\psfrag{E2}{{\scriptsize $E^2$}}
\psfrag{E3}{{\scriptsize $E^3$}}
\psfrag{Ee}{{\scriptsize $E^{d-1}$}}
\psfrag{Ed}{{\scriptsize $E^d$}}
\psfrag{w1}{{\scriptsize $W_{j_1}^i$}}
\psfrag{w2}{{\scriptsize $W_{j_2}^i$}}
\psfrag{we}{{\scriptsize $W_{j_{d-1}}^i$}}
\psfrag{wd}{{\scriptsize $W_{j_{d}}^i$}}
\includegraphics[scale=0.5]{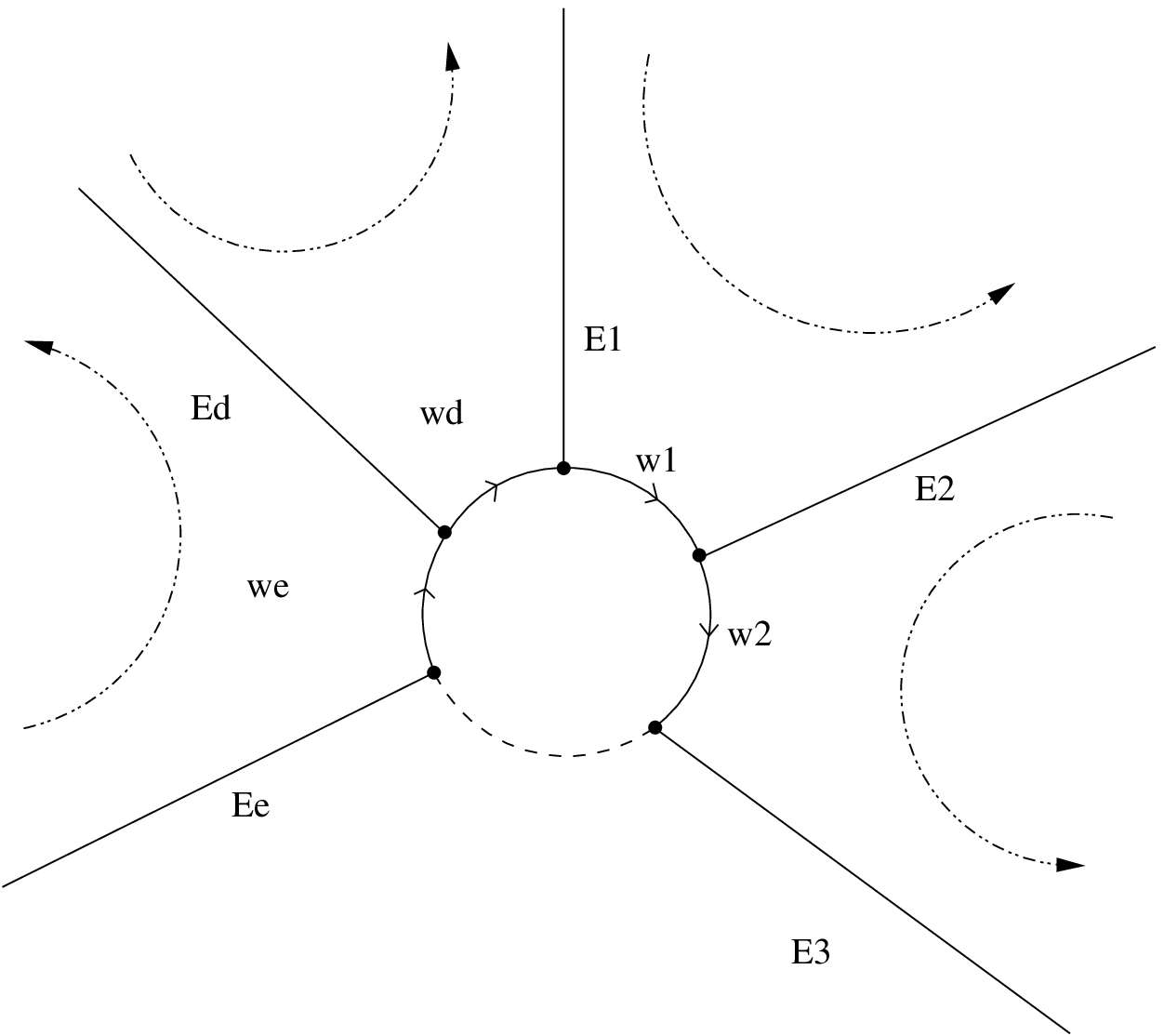}
\\
\vspace{0.38cm}                                                                  
\refstepcounter{figure}\label{eandws}                                             
Figure \thefigure                  
\end{center}
\end{figure}   
Here  $W_j^i=W_{j_1}^i\ldots W_{j_d}^i$, where each $W_{j_k}$
is a subword which may or may not have length zero. Also, note that the $E's$
are not necessarily distinct (both end points may lie on the same boundary component).

Since $U$ is an orientable word, each vertex of $\Gamma_U$ is
regular. Therefore, since $\Gamma_U$ is obtained from $\Gamma_S$ by
pasting each boundary component of the surface $S$ with a disc and
contracting the edges around the boundary components to a point, we may
think of the boundary components also as being `regular'. That is, 
we may renumber the long edges, as shown
in the diagram,  such that
the following cyclic subwords appear in $W$.
\begin{equation*}
(E^1)^{\varepsilon _1}W_{j_1}^i(E^2)^{-\varepsilon _2},
(E^2)^{\varepsilon _2}W_{j_2}^i(E^3)^{-\varepsilon _3}, \ldots
,(E^{d-1})^{\varepsilon _{d-1}}W_{j_{d-1}}^i(E^d)^{-\varepsilon _d}, 
(E^{d})^{\varepsilon _{d}}W_{j_{d}}^i(E^1)^{-\varepsilon _1},
\end{equation*} 
where $\varepsilon _k=\pm 1$, for all $k=1,\ldots ,d$. Now we have
already shown that, for each long edge $E$ in $\hat{W}$, there exists
 words, $e_1$ and $e_2$, in $F(X)$ of length at most $5l+M+3$ such
 that, in the Cayley graph $\Gamma_X(H)$, $e_1$ and $e_2$ label  geodesic paths
 from $\tau(\theta(E))$  and $\tau(\theta(E^{-1}))$  respectively to
 vertices  on the  geodesic path labelled by $F$.

Let $z_j^i$ be a cyclic word in $F(X)$ such that 
\begin{equation*}
z_j^i=z_{j_1}^i\ldots z_{j_d}^i,
\end{equation*}
where 
\begin{eqnarray*}
z_{j_{k+1}}^i & =  & (e_{\mu _k}^k)^{-1}\theta(W_{j_k}^i)e_{\mu
  _{k+1}}^{k+1}\qquad\textrm{for}\, k=1,\ldots,d-1,\\
\textrm{and}\qquad 
z_{j_1}^i & =  & (e_{\mu _d}^d)^{-1}\theta(W_{j_d}^i)e_{\mu
  _{1}}^1,
\end{eqnarray*}
with 
\begin{equation*}
\mu _k=\left
\{\begin{array}{ll}
1 & \textrm{if $\varepsilon _k=1$}\\
2 & \textrm{if $\varepsilon _k=-1$}\\
\end{array}
\right. .
\end{equation*}
From the definition of $z_j^i$, in the group $H$ we can see that
\begin{eqnarray*}
z_j^i & =_H & z_{j_1}^i\ldots z_{j_d}^i\\
      & =_H & (e_{\mu _1}^1)^{-1}\theta(W_{j_1}^i)e_{\mu _{2}}^2 (e_{\mu
        _2}^2)^{-1}\theta(W_{j_2}^i)e_{\mu _{3}}^3\ldots (e_{\mu
        _d}^k)^{-1}\theta(W_{j_d}^i)e_{\mu _{1}}^1\\
      &=_H & (e_{\mu _1}^1)^{-1}\theta(W_{j_1}^i)\theta(W_{j_2}^i)\ldots
      \theta(W_{j_d}^i) e_{\mu _{1}}^1\\
      & =_H &  (e_{\mu _1}^1)^{-1}\theta(W_j^i)e_{\mu _{1}}^1\\
      &\sim _H & \theta(W_j^i).
\end{eqnarray*}
We construct $z_j^i$  for all $j=1,\ldots ,t_i$ and $i=1,\ldots
,p$. We shall denote the set  $\{z_j^i, j=1,\ldots ,t_i\}$ by
$C_i$. Now, since  $z_j^i\sim _H  \theta(W_j^i)$ for all $j=1.\ldots
,t_i$, it follows that 
\begin{equation*}
genus_H(z_1^i,\ldots ,z_{t_i}^i)=genus_H(\theta(W_1^i),\ldots ,\theta(W_{t_i}^i))=g_i-t_i+1.
\end{equation*}
We shall say that the genus of $C_i$ is $g_i$. 
\begin{lemma}\label{C_i}
With $g_i$ and $d$ defined as above, if $g_i=0$ and hence $t_i=1$ then $d\geq 3$. 
\end{lemma}
\begin{proof}
Suppose that $d=1$. Then $W_1^i=W_{11}^i$. We know that
$genus_H(z_1^i)=genus_H(\theta(W_1^i))=0$ and Lemma \ref{gi} implies
that $genus_{F(\mathcal{A})}(W_1^i)=0$. That is
$W_1^i=W_{11}^i=1$. But $W_{11}^i$ is a cyclic subword of the Wicks
form $W$ and by definition $W$ is cyclically reduced. Therefore this
can't occur. 

Now suppose that $d=2$. Then $W_1^i=W_{11}^iW_{12}^i$. With a similar
argument it follows that $genus_{F(\mathcal{A})}(W_1^i)=0$ and thus
$W_{11}^i=(W_{12}^i)^{-1}$. But the cyclic subwords
$(E^1)^{\varepsilon _1}W_{11}^i(E^2)^{-\varepsilon _2}$ and,
$(E^2)^{\varepsilon _2}W_{12}^i(E^1)^{-\varepsilon _1}$ appear in the
Wicks form $W$
which implies we that we have redundancy in $W$. Therefore this also can't
occur. 
\end{proof}

Finally we need to construct a set of cyclic words in $F(X)$ which we
shall denote $C_0$. We do this in the following way. Let $v_1,\ldots
,v_t$ be the set of
vertices of $\Gamma_S$ which do not lie on a boundary component of
$S$ i.e.~they lie on the endpoints of only long edges. Note that
each of these vertices has degree at least $3$\label{least3} since
$\Gamma_S$ is the graph obtained by identifying long letters of $W$
and $W$ is a Wicks form (recall that $\Gamma_W$ has no vertices of degree $1$ or
$2$). Consider $v_j$, for some $1\leq j\leq t$. Let the degree of
this vertex be $r$.  Thus there are long
edges $F^1,\ldots ,F^r$ which have an end point which is $v_j$. Note
that these $F's$ are not necessarily distinct, that is we could have loops. 
See Figure \ref{vertvandf}.
 \begin{figure}[ht]
\begin{center}
\psfrag{F1}{{\scriptsize $F^1$}}
\psfrag{F2}{{\scriptsize $F^2$}}
\psfrag{F3}{{\scriptsize $F^3$}}
\psfrag{Fe}{{\scriptsize $F^{r-1}$}}
\psfrag{Fr}{{\scriptsize $F^r$}}
\includegraphics[scale=0.5]{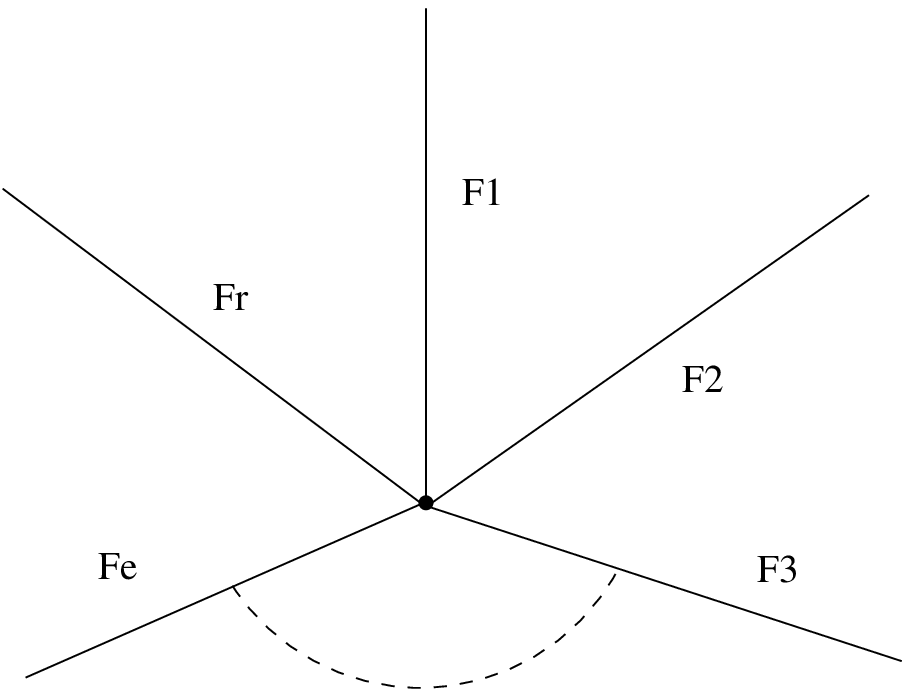}
\\
\vspace{0.38cm}                                                                  
\refstepcounter{figure}\label{vertvandf}                                             
Figure \thefigure                  
\end{center}
\end{figure}     

Thus, since $v_j$ is regular, we can renumber the long edges which have
an end point $v_j$ such that the following cyclic subwords appear in
$W$.
\begin{equation*}
(F^1)^{\varepsilon _1}(F^2)^{-\varepsilon _2},
(F^2)^{\varepsilon _2}(F^3)^{-\varepsilon _3}, \ldots
,(F^{r-1})^{\varepsilon _{d-1}}(F^r)^{-\varepsilon _r}, 
(F^{r})^{\varepsilon _{r}}(F^1)^{-\varepsilon _1},
\end{equation*}
where $\varepsilon _k=\pm 1$, for all $k=1,\ldots ,d$. As in the
construction of $C_i$ we again use the fact that
for each long edge $e$ in $\hat{W}$ there exists
 words $e_1$ and $e_2$, in $F(X)$, of length at most $5l+M+3$, such
 that, in the Cayley graph $\Gamma_X(H)$, $e_1$ labels a geodesic path
 from $\tau(\theta(E))$ to a vertex on $F$ and $e_2$ labels a geodesic
 path from $\tau(\theta(E^{-1}))$ to a vertex on $F$. 

Let $z_j^0$ be a cyclic word in $F(X)$ such that 
\begin{equation*}
z_j^0=z_{j_1}^0\ldots z_{j_d}^0,
\end{equation*}
where 
\begin{eqnarray*}
z_{j_{k+1}}^0 & =  & (f_{\mu _k}^k)^{-1}f_{\mu
  _{k+1}}^{k+1}\qquad\textrm{for}\, k=1,\ldots,r-1,\\
\textrm{and}\qquad 
z_{j_1}^0 & =  & (f_{\mu _r}^r)^{-1}f_{\mu
  _{1}}^1,
\end{eqnarray*}
again with 
\begin{equation*}
\mu _k=\left
\{\begin{array}{ll}
1 & \textrm{if $\varepsilon _k=1$}\\
2 & \textrm{if $\varepsilon _k=-1$}\\
\end{array}
\right. .
\end{equation*}
Clearly $z_j^0$ is equal to $1$ in $H$ so has genus $0$ in $H$. 
Let $C_0=\{z_j^0: 1\leq j\leq t\}$. 

We have constructed the required sets of cyclic words in $F(X)$ which
shall be used in the extension of $U$ needed to obtain part $2$ of
the Theorem.

Write $W$ around the outer boundary component of an annulus and label
the inner boundary component 
with $F$ from some fixed base point. See Figure \ref{wandf}.
\begin{figure}[ht]
\begin{center}
\psfrag{W}{{\scriptsize $\hat{W}$}}
\psfrag{F}{{\scriptsize $F$}}
\includegraphics[scale=0.4]{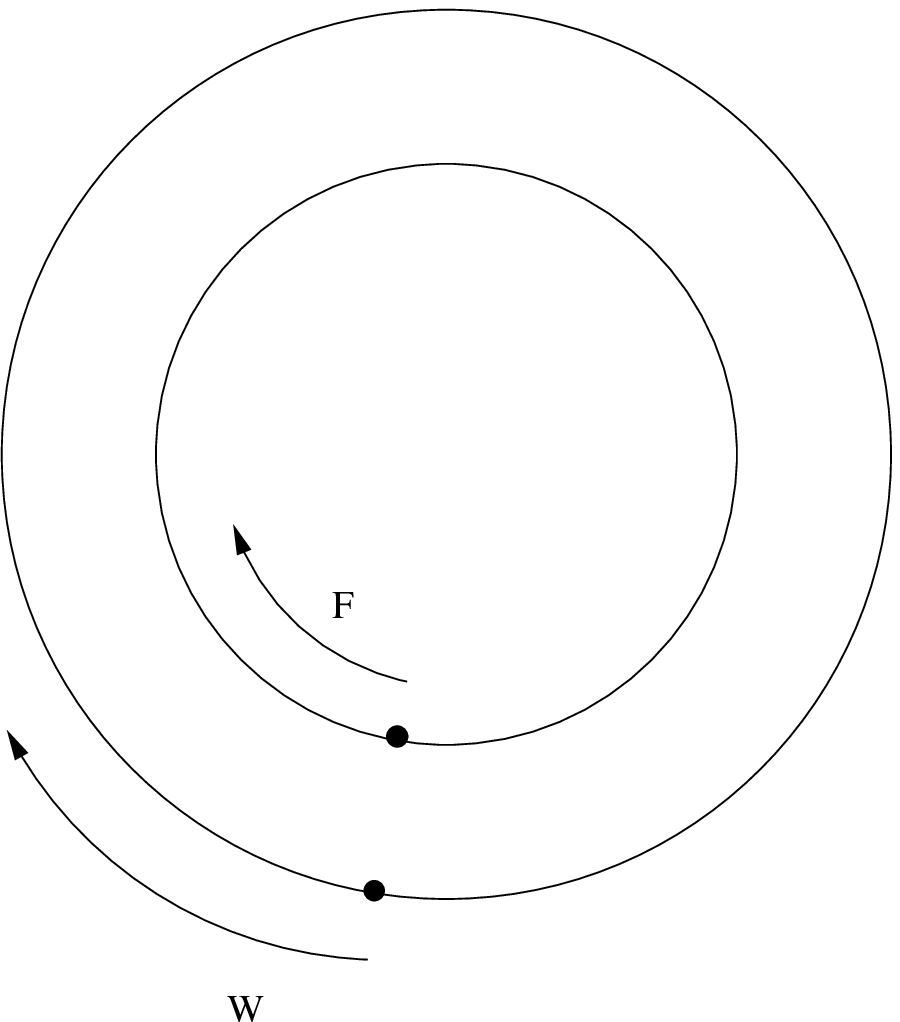}
\\
\vspace{0.38cm}                                                                  
\refstepcounter{figure}\label{wandf}                                             
Figure \thefigure                  
\end{center}
\end{figure}   
Let $A$ be a long edge of $W$. For the occurrence of $A$ with
exponent $1$ we add a properly embedded path, labelled by $a_1\in
F(X)$ (remember this has length at most $5l+M+3$ in $H$, see Lemma
\ref{lsc}), from $\tau(A)$  to $\tau(A_1)$ on $F$
 and similarly,
for the occurrence of $A$ with exponent $-1$, we add a properly embedded
path, labelled by  $a_2\in F(X)$, from
$\tau(A^{-1})$ to $\tau(A_2^{-1})$ on $F$. See Figure \ref{wanadf2}.
\begin{figure}[ht]
\begin{center}
\psfrag{W}{{\scriptsize $\hat{W}$}}
\psfrag{F}{{\scriptsize $F$}}
\psfrag{A}{{\scriptsize $A$}}
\psfrag{A1}{{\scriptsize $A_1$}}
\psfrag{A2}{{\scriptsize $A_2$}}
\psfrag{a1}{{\scriptsize $a_1$}}
\psfrag{a2}{{\scriptsize $a_2$}}
\includegraphics[scale=0.4]{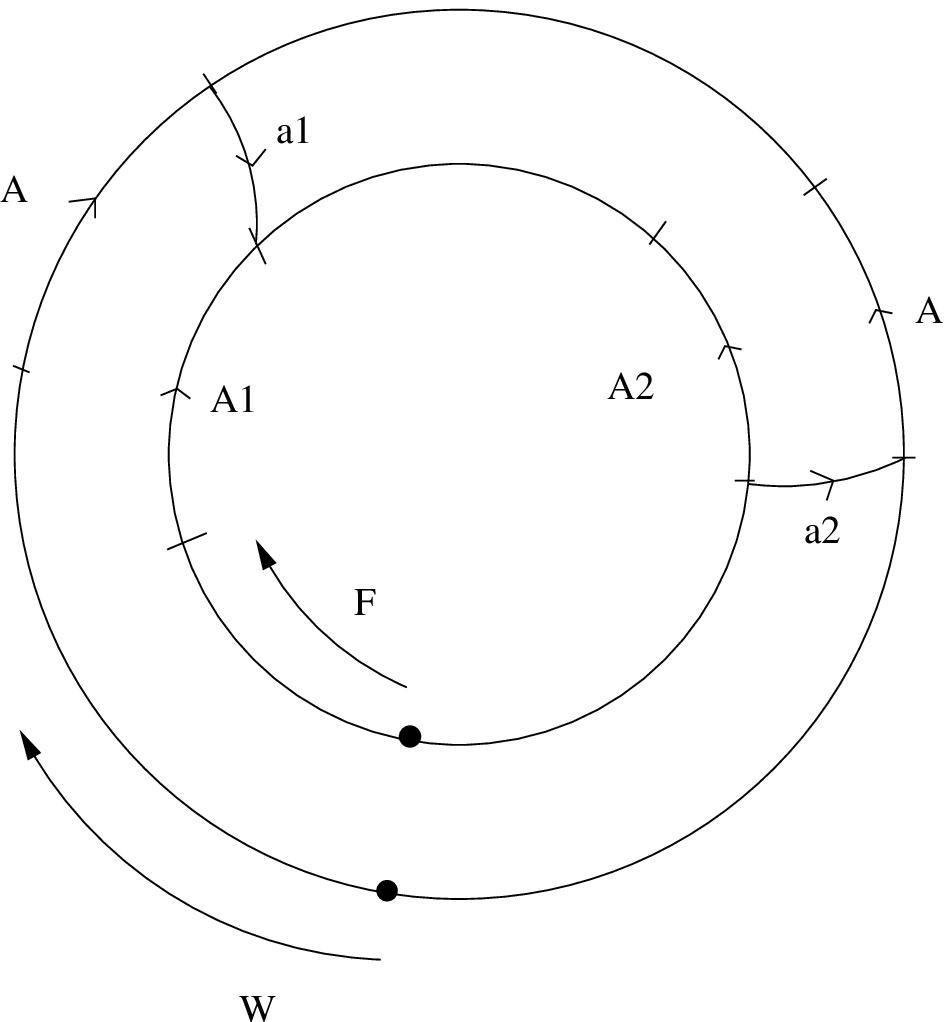}
\\
\vspace{0.38cm}                                                                  
\refstepcounter{figure}\label{wanadf2}                                             
Figure \thefigure                  
\end{center}
\end{figure}   
 
We do this for each long edge of $W$.  Identify the long edges of $W$,
respecting orientation to obtain a surface $S'$. 
The new surface $S'$ is just the surface $S$ with a disc removed. The
boundary of the disc removed from $S$ is labelled by
$F$. Now the graph $\Gamma _S$ is embedded in $S'$ and is a subgraph
of a larger graph also embedded in $S'$, which consists of the
boundary of the annulus and the properly embedded paths on the annulus
after identification.
We denote this  graph $\Gamma_{S'}$. 
See Figure
\ref{gammasA}.
\begin{figure}[ht]
\begin{center}
\psfrag{S}{{\scriptsize $S'$}}
\includegraphics[scale=0.4]{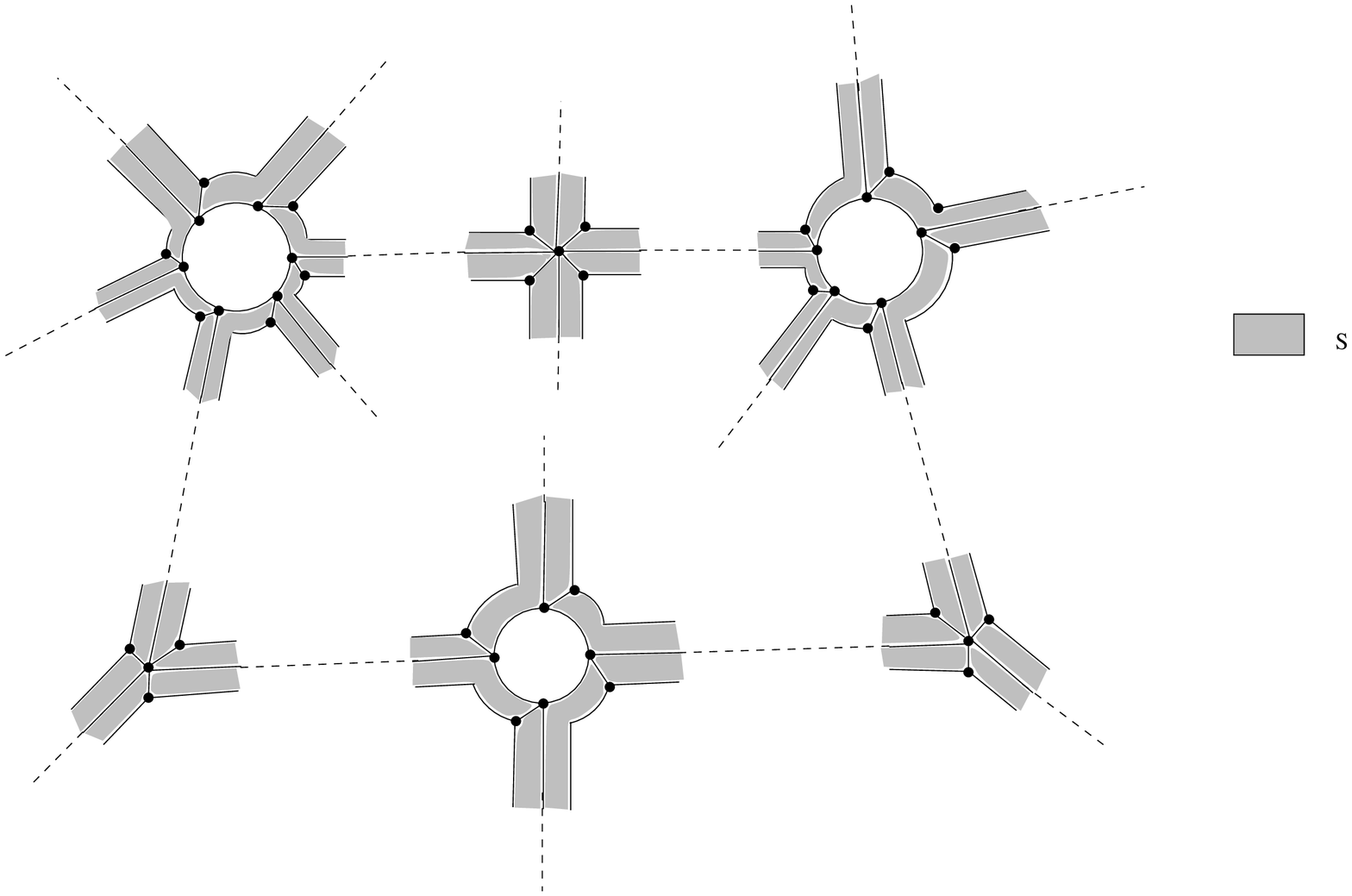}
\caption{The graph $\Gamma_{S'}$ embedded in $S'$}\label{gammasA}
\end{center}
\end{figure}   

We again consider sections of the graph $\Gamma_S$ where long edges
either meet a cyclic word $W_j^i\in B_i$, $1\leq j\leq t_i$,
$1\leq i\leq p$ or one of the vertices $v_{j'}$, $1\leq j'\leq t$,
which does not lie on a boundary component of $S$. Let $E^1,\ldots
,E^d$  be the long edges which  have an end point on
$W_j^i$(\emph{resp.} which have an end point which is $v_{j'}$) . Let
$W_j^i=W_{j_1}^i\ldots W_{j_d}^i$. Again these can be renumbered such
that $W$ contains the cyclic subwords
\begin{equation*}
(E^1)^{\varepsilon _1}W_{j_1}^i(E^2)^{-\varepsilon _2},
(E^2)^{\varepsilon _2}W_{j_2}^i(E^3)^{-\varepsilon _3}, \ldots
,(E^{d-1})^{\varepsilon _{d-1}}W_{j_{d-1}}^i(E^d)^{-\varepsilon _d}, 
(E^{d})^{\varepsilon _{d}}W_{j_{d}}^i(E^1)^{-\varepsilon _1},
\end{equation*} 
where $\varepsilon _k=\pm 1$, for all $k=1,\ldots ,d$.
 See Figure \ref{eandws}. The word
$W_j^i$ obviously has length zero if we are considering $v_{j'}$. In
our new  extended graph $\Gamma_{S'}$ this section of the graph now
takes the form as shown in Figure \ref{eandws2},
 \begin{figure}[ht]
\begin{center}
\psfrag{E1}{{\scriptsize $(E^1)^{\varepsilon _1}$}}
\psfrag{E2}{{\scriptsize $(E^2)^{\varepsilon _2}$}}
\psfrag{E3}{{\scriptsize $(E^3)^{\varepsilon _3}$}}
\psfrag{Ee}{{\scriptsize $(E^{d-1})^{\varepsilon _{d-1}}$}}
\psfrag{Ed}{{\scriptsize $(E^d)^{\varepsilon _d}$}}
\psfrag{w1}{{\tiny $W_{j_1}^i$}}
\psfrag{w2}{{\tiny $W_{j_2}^i$}}
\psfrag{we}{{\tiny $W_{j_{d-1}}^i$}}
\psfrag{wd}{{\tiny $W_{j_{d}}^i$}}
\psfrag{E11}{{\scriptsize $(E_{\mu _1}^1)^{\varepsilon _1}$}}
\psfrag{E21}{{\scriptsize $(E_{\mu _2}^2)^{\varepsilon _2}$}}
\psfrag{E31}{{\scriptsize $(E_{\mu _3}^3)^{\varepsilon _3}$}}
\psfrag{Ee1}{{\scriptsize $(E_{\mu _{d-1}}^{d-1})^{\varepsilon _{d-1}}$}}
\psfrag{Ed1}{{\scriptsize $(E_{\mu _d}^d)^{\varepsilon _d}$}}
\psfrag{E12}{{\scriptsize $(E_{\nu _1}^1)^{\varepsilon _1}$}}
\psfrag{E22}{{\scriptsize $(E_{\nu _2}^2)^{\varepsilon _2}$}}
\psfrag{E32}{{\scriptsize $(E_{\nu _3}^3)^{\varepsilon _3}$}}
\psfrag{Ee2}{{\scriptsize $(E_{\nu _{d-1}}^{d-1})^{\varepsilon _{d-1}}$}}
\psfrag{Ed2}{{\scriptsize $(E_{\nu _d}^d)^{\varepsilon _d}$}}
\psfrag{e1}{{\tiny $e_{\mu _1}^1$}}
\psfrag{e2}{{\tiny $e_{\mu _2}^2$}}
\psfrag{e3}{{\tiny $e_{\mu _3}^3$}}
\psfrag{ee}{{\tiny $e_{\mu _{d-1}}^{d-1}$}}
\psfrag{ed}{{\tiny $e_{\mu _d}^d$}}
\includegraphics[scale=0.9]{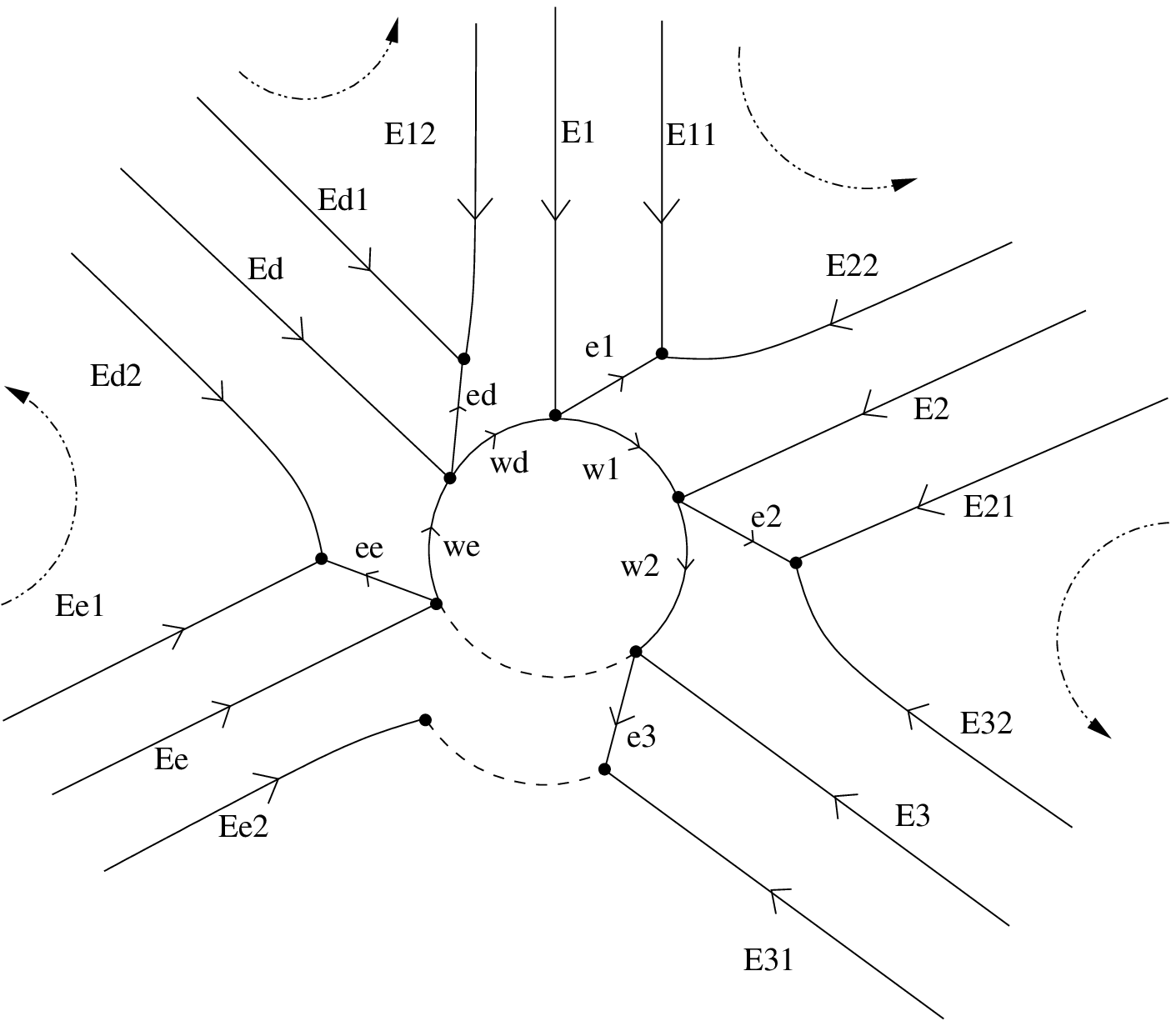}
\\
\vspace{0.38cm}                                                                  
\refstepcounter{figure}\label{eandws2}                                             
Figure \thefigure                  
\end{center}
\end{figure}     
where 
\begin{equation*}
\mu _k=\left
\{\begin{array}{ll}
1 & \textrm{if $\varepsilon _k=1$}\\
2 & \textrm{if $\varepsilon _k=-1$}\\
\end{array}
\right.
\end{equation*}
and
\begin{equation*}
\nu _k=\left
\{\begin{array}{ll}
1 & \textrm{if $\mu _k=2$}\\
2 & \textrm{if $\mu _k=1$}\\
\end{array}
\right. .
\end{equation*}
Remove the edges $E^1,\ldots , E^{d}$ from $\Gamma_{S'}$. Now,  for
each $k=1,\ldots ,d-1$,  there is an  edge path from
$\tau(E_{\mu _k}^k)$ to $\tau(E_{\mu _{k+1}}^{k+1})$ labelled by
$(e_{\mu _k}^k)^{-1}W_{j_k}^ie_{\mu _{k+1}}^{k+1}$, where  $e_{\mu
  _k}^k$, $e_{\mu _{k+1}}^{k+1}\in F(X)$ and  $W_{j_k}^i$ is sequence
of short edges in $W$. Note that if $k=d$ then the edge path from
$\tau(E_{\mu _d}^d)$  to $\tau(E_{\mu _{1}}^{1})$ is labelled by
 $(e_{\mu _d}^d)^{-1}W_{j_k}^ie_{\mu _{1}}^{1}$.  We have already
 defined  $z_{j_k}^i=(e_{\mu
  _k}^k)^{-1}\theta(W_{j_k}^i)e_{\mu _{k+1}}^{k+1} $, for 
$k=1,\ldots ,d-1$ and  $(e_{\mu _d}^d)^{-1}\theta(W_{j_k}^i)e_{\mu
  _{1}}^{1}$ for $k=d$. For all $k=1,\ldots ,d$, we shall
also remove the  edges
labelled by $W_{j_k}^i$ and $e_{j_k}^k$, add  new edges from
$\tau(E_{\mu _k}^k)$ to $\tau(E_{\mu _{k+1}}^{k+1})$, labelled by
$z_{j_k}^i$, for $k=1,\ldots ,d-1$, and add a new edge from  $\tau(E_{\mu
  _d}^d)$  to $\tau(E_{\mu _{1}}^{1})$, labelled by $z_{j_d}^i$, for
$k=d$. See Figure \ref{eandws3}.
 \begin{figure}[ht]
\begin{center}
\psfrag{E11}{{\tiny $(E_{\mu _1}^1)^{\varepsilon _1}$}}
\psfrag{E21}{{\tiny $(E_{\mu _2}^2)^{\varepsilon _2}$}}
\psfrag{E31}{{\tiny $(E_{\mu _3}^3)^{\varepsilon _3}$}}
\psfrag{Ed1}{{\tiny $(E_{\mu _d}^d)^{\varepsilon _d}$}}
\psfrag{E12}{{\tiny $(E_{\nu _1}^1)^{\varepsilon _1}$}}
\psfrag{E22}{{\tiny $(E_{\nu _2}^2)^{\varepsilon _2}$}}
\psfrag{E32}{{\tiny $(E_{\nu _3}^3)^{\varepsilon _3}$}}
\psfrag{Ed2}{{\tiny $(E_{\nu _d}^d)^{\varepsilon _d}$}}
\psfrag{z1}{{\tiny $z_{j_1}^i$}}
\psfrag{z2}{{\tiny $z_{j_2}^i$}}
\psfrag{z3}{{\tiny $z_{j_3}^i$}}
\psfrag{zd}{{\tiny $z_{j_d}^i$}}
\includegraphics[scale=0.6]{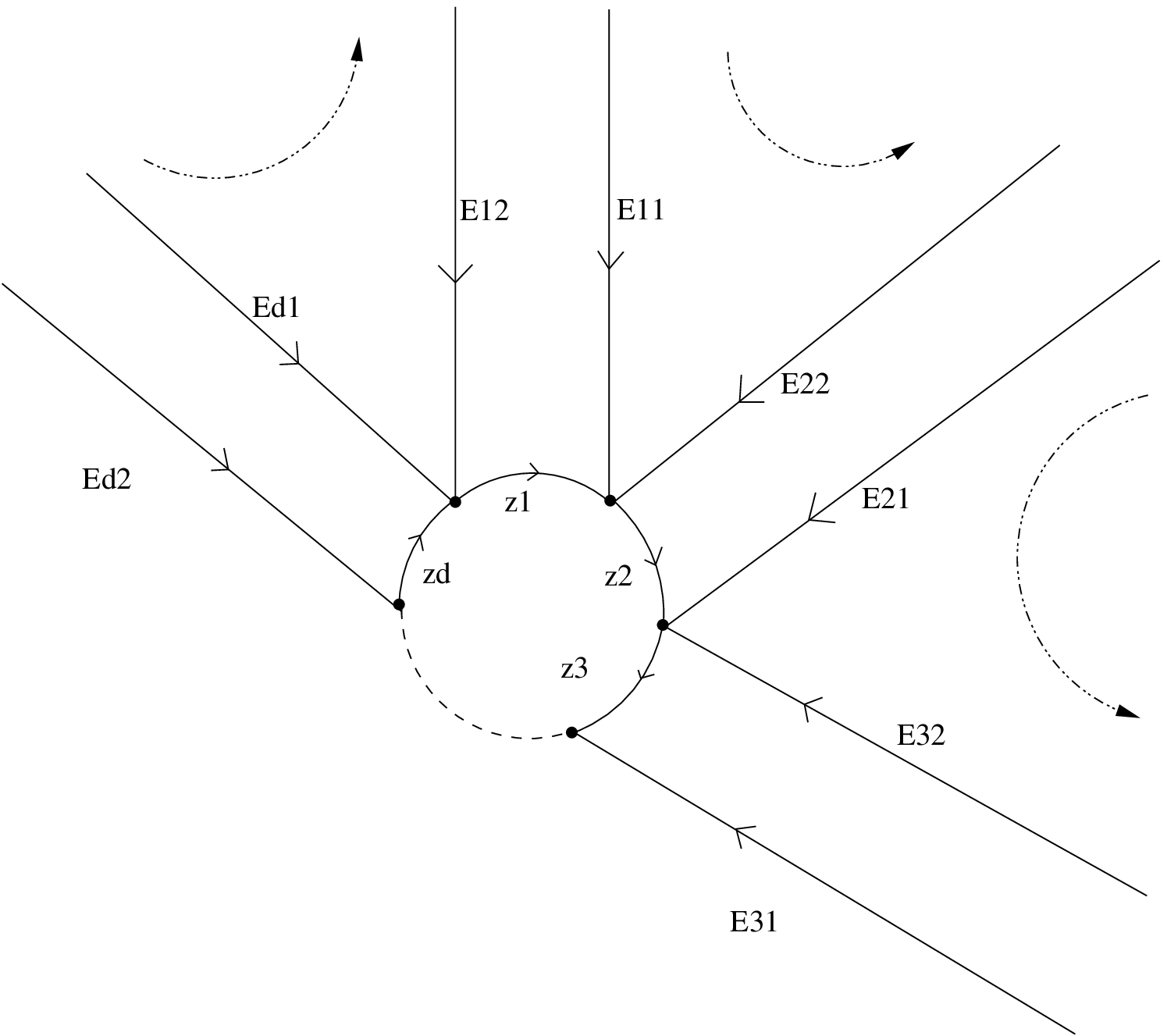}
\\
\vspace{0.38cm}                                                                  
\refstepcounter{figure}\label{eandws3}                                             
Figure \thefigure                  
\end{center}
\end{figure}        

We now do this for all $W_j^i\in B_i$, $j=1,\ldots, t_i$, $i=1,\ldots ,p$,
and $v_{j'}$, $j'=1,\ldots ,t$. We shall call this new graph
$\Gamma_{S''}$.  It is clear, by the way we numbered $E^1,\ldots
,E^d$, that the following are cyclic subwords of the  Hamiltonian
cycle in $\Gamma_{S''}$ labelled by $F$.  
\begin{equation*}
(E_{\mu 1}^1)^{\varepsilon _1}(E_{\nu 2}^2)^{-\varepsilon _2},
(E_{\mu 2}^2)^{\varepsilon _2}(E_{\nu 3}^3)^{-\varepsilon _3}, \ldots
,(E_{\mu _{d-1}}^{d-1})^{\varepsilon _{d-1}}(E_{\nu d}^d)^{-\varepsilon _d}, 
(E_{\mu _d}^{d})^{\varepsilon _{d}}(E_{\nu 1}^1)^{-\varepsilon _1},
\end{equation*} 
where $\varepsilon _k=\pm 1$, for all $k=1,\ldots ,d$.  
The cyclic sequence of letters of $U$ gives  the
cyclic sequence of segments of $F$ by replacing a letter $E$ of
$U$ with $E_1$ for its occurrence with exponent $1$ and $E_2^{-1}$ for
its occurrence with exponent $-1$. Thus, to show that $\Gamma_{S''}$ is
an extension of $U$ over $H$ by the cyclic words in the sets
$C_0,\ldots ,C_p$, it is clear that we must show that  step $3$, of
the construction of an extension over $H$, holds. Let $A$ be a long
edge of $\hat{W}$ and without loss of generality assume that $A$ appears
before $A^{-1}$.  Consider the two edges
$A_1$ and $A_2$ in $\Gamma_{S''}$. There exist an $z_{j_r}^i$ and
$z_{j'_s}^{i'}$ which we shall denote $x$ and $y$ respectively such
that the edge path  $A_1, x^{-1}, A_2^{-1}, y^{-1}$ is a cycle in
$\Gamma_{S''}$. See Figure  \ref{extstep3b}.
 \begin{figure}[ht]
\begin{center}
\psfrag{x}{{\scriptsize $x$}}
\psfrag{y}{{\scriptsize $y$}}
\psfrag{e1}{{\scriptsize $A_1$}}
\psfrag{e2}{{\scriptsize $A_2$}}
\includegraphics[scale=0.5]{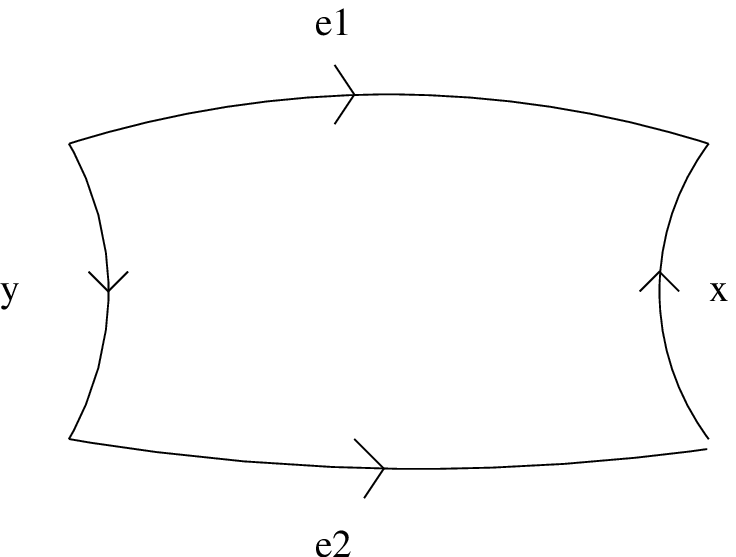}
\\
\vspace{0.38cm}                                                                  
\refstepcounter{figure}\label{extstep3b}                                             
Figure \thefigure                  
\end{center}
\end{figure}     
By the construction (see definitions of the $z_j^i$'s
and Figure \ref{eandws2}) $x$
and $y$ can be split up into subwords, $x=x_1x_2$ and $y=y_1y_2$
respectively, such that in $\Gamma_X(H)$ we have closed paths as shown
in Figure \ref{apmb4}. Note that $x_2$ and $y_2$ are $a_1$ and $a_2$ but $x_1$
and $y_1$ may include part of $\theta(W)$ as indicated in the diagram.
\begin{figure}[ht]
\begin{center}
\psfrag{A}{{\scriptsize $\theta(A)$}}
\psfrag{A1}{{\scriptsize $A_1$}}
\psfrag{A2}{{\scriptsize $A_2$}}
\psfrag{x1}{{\scriptsize $x_1$}}
\psfrag{F}{{\scriptsize $F$}}
\psfrag{x2}{{\scriptsize $x_2$}}
\psfrag{y2}{{\scriptsize $y_2$}}
\psfrag{y1}{{\scriptsize $y_1$}}
\includegraphics[scale=0.4]{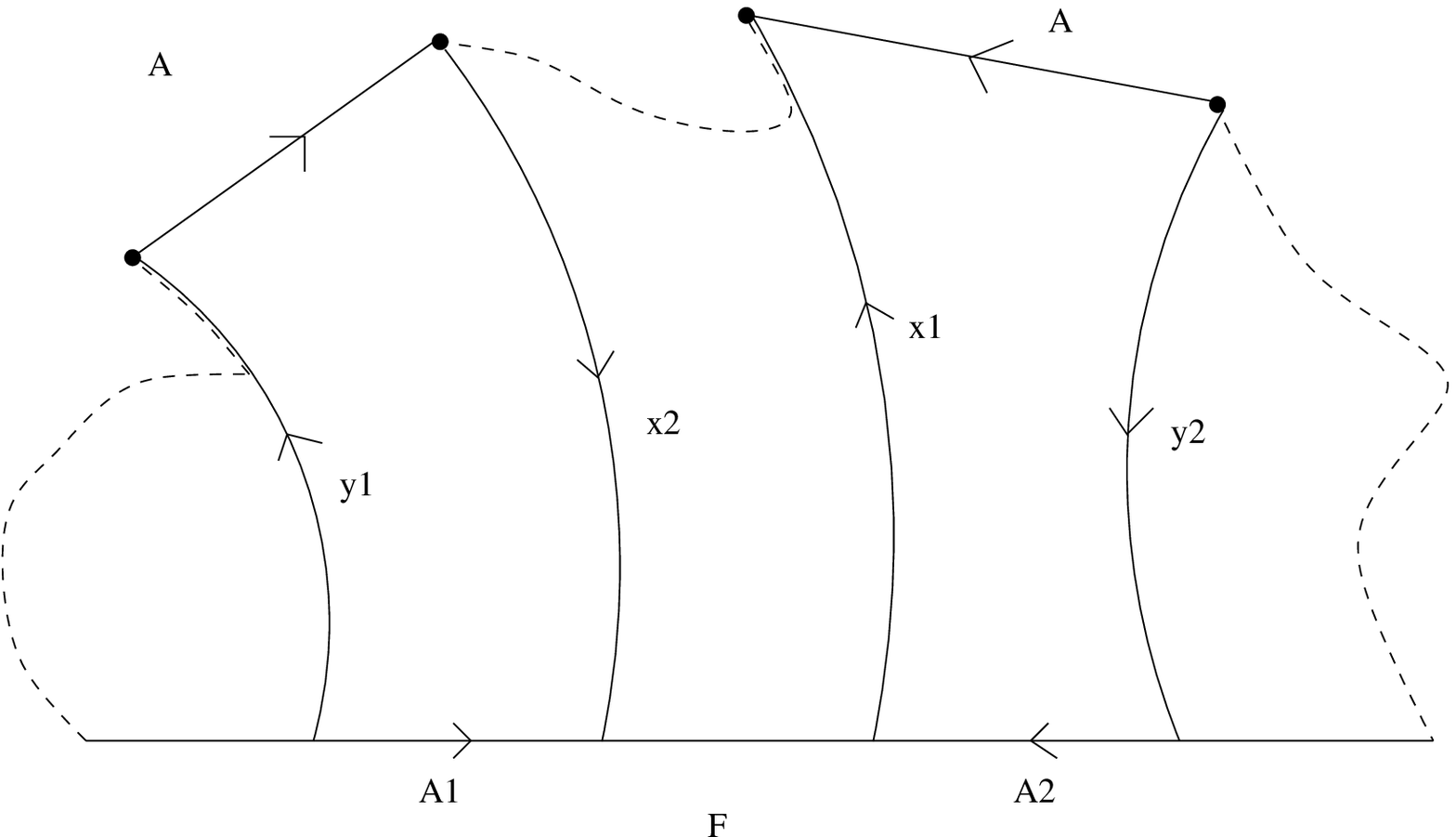}
\\
\vspace{0.38cm}                                                                  
\refstepcounter{figure}\label{apmb4}                                             
Figure \thefigure                  
\end{center}
\end{figure}
From $\Gamma_X(H)$ we can see that  
\begin{eqnarray*}
\theta(A) & =_H & y_1^{-1} A_1 x_2^{-1}\\
          & =_H & y_2 A_2 x_1.    
\end{eqnarray*}
It follows that $A_1x^{-1}A_2^{-1}y^{-1}=_H 1$. Hence step $3$
holds and we have an extension of $U$ over $H$. 
\subsection{Length and Genus of the Extension on $U$}
We have already shown that the genus of $U$ is $k$. 
Remember that the genus of the extension is given by the sum of the genus of the
cyclic words used. But these are just the elements of the sets
$C_0,C_1,\ldots ,C_p$. Now the genus of every element of $C_0$ is
zero (Note also that the vertices extended by elements of $C_0$ have
degree of at least $3$, see page \pageref{least3}) and the genus of
the elements in each set $C_i$, for $i=1,\ldots ,p$ is
equal to $g_i-t_i+1$ (again note that if the genus of $C_i=0$ then the
vertex extended by the only element of $C_i$ is of degree at least
$3$, see Lemma \ref{C_i}). Therefore, it follows from the definition that
the genus of the extension is equal to the sum of the genus of each
set $C_i$, for $i=1,\ldots ,p$.  We know that 
\begin{equation*}
g=\sum_{i=1}^{p}g_i,
\end{equation*}
and $n=k+g$, see Lemma \ref{gi}. Thus the extension is of the required genus for part $2$ of the
Theorem to hold. 

The length of the extension is given by the sum of the length of the
words in the sets $C_0,C_1,\ldots ,C_p$. Each element of a set $C_i$,
for $i=0,\ldots ,p$, takes the form 
\begin{equation*}
z_j^i=z_{j_1}^i\ldots z_{j_d}^i,
\end{equation*}
where 
\begin{eqnarray*}
z_{j_{k+1}}^i & = & (e_{\mu _k}^k)^{-1}\theta(W_{j_k}^i)e_{\mu
  _{k+1}}^{k+1},\qquad\textrm{for}\, k=1,\ldots,d-1,\\
\textrm{and}\qquad 
z_{j_1}^i & = & (e_{\mu _d}^d)^{-1}\theta(W_{j_d}^i)e_{\mu
  _{1}}^1,
\end{eqnarray*}
with all the $\theta(W_{j_k}^i)$'s  having length zero if $i=0$.
Each short edge of $W$ appears in a unique $W_{j_k}^i$ and these have
labels in $F(X)$ which have
length  at most  $12l+M+4$.  Each
$e_{\mu _k}^k$ is a word of length at most $5l+M+3$ arising from  a
long edge of $W$, see Lemma \ref{lsc} and these each appear
twice in some $z_j^i$. By Lemma \ref{cull},  it follows that the maximum
number of letters in $W$ is $12n-6$. Thus the sum of number of short
letters and long letters is at most $12n-6$. Therefore, if
$\mathcal{S}$ and $\mathcal{L}$ are the numbers of short edges and long
edges respectively, the length of the extension 
is given by
\begin{eqnarray*}
(12l+M+4)\mathcal{S}+2(5l+M+3)\mathcal{L} &\leq & 2(12n-6)(12l+M+4).
\end{eqnarray*}   
Hence we have the required extension and the theorem holds.\label{endthm}
\end{proof}
\section{Forms for Commutators in Hyperbolic Groups}\label{commutators}
We shall now use Theorem \ref{Thexp} to obtain a full list of all
possible forms for commutators in $H=\langle X|R\rangle$. Now since in this
case $n=1$, it follows that  $l=\delta(\log _2(6)+1)$.
\begin{prop}\label{propos}
If $h$ is a commutator in $H$ then there are  words
$R$ and $F$ in $F(X)$ which are minimal in  $H$ such that $h=_H
RFR^{-1}$,  where $|R|\leq 
59l+8M+28+2\delta +\frac{|h|_H}{2}$ and $F$ takes one of the following forms.
\begin{enumerate}
\item $|F|\leq 6(12l+M+4)$ with $F=_H XYZX^{-1}Y^{-1}Z^{-1}$ and
$|X|$, $|Y|$, $|Z|\leq 12l+M+4$.
\item $F=A_1A_2^{-1}$ with $A_1=_H \xi _1^{-1}A_2 \xi _2$. Where $|\xi
_1|+|\xi _2|\leq 12(12l+M+4)$ and $\xi _1$ is conjugate to $\xi _2$ in
$H$.  
\item $F=A_1B_1A_2 ^{-1}B_2 ^{-1}$ with $A_1=_H \xi _1A_2\xi _3$, $B_1=_H
\xi _4B_2\xi _2$. Where $|\xi _1|+|\xi _2|+|\xi _3|+|\xi _4|\leq
12(12l+M+4)$   and
$\xi _1\xi _2\xi _3\xi _4=_H 1$.
\item $F=A_1B_1C_1A_2^{-1}B_2^{-1}C_2^{-1}$ with $A_1=_H \xi _1A_2\rho
  _1$,  $B_1=_H
\rho _2B_2\xi _2$ and $C_1=_H \xi _3C_2\rho _3$. Where $|\xi _1|+|\xi
_2|+|\xi _3|+|\rho _1|+|\rho _2|+|\rho _3|\leq 12(12l+M+4)$ and   $\xi _1\xi
_2\xi _3 =_H\rho _1\rho _2\rho _3=_H 1$. 
\end{enumerate}
\end{prop}
\begin{proof}
By Theorem \ref{Thexp} and our knowledge of genus $1$ Wicks forms (see
beginning of the paper), $h$
is conjugate to a  minimal word $F$ which
either has form $1$
above or is obtained by a genus $g$ extension of length at most $12(12l+M+4)$
 on some orientable word of genus $k$ such that $g+k=1$. This implies
 that  there are only
three possible orientable words which can have a suitable
extension(this is easy to check). 
\begin{enumerate}
\item[(i)] An orientable word $U=AA^{-1}$ of genus $0$ with a joint genus
  $1$ extension constructed on the two vertices of $\Gamma_U$.
\item[(ii)] An orientable word $V=ABA^{-1}B^{-1}$ of genus $1$ with a genus
  $0$ extension constructed on the only vertex of $\Gamma_V$.
\item[(iii)] An orientable word $W=ABCA^{-1}B^{-1}C^{-1}$ of genus $1$ with a
  genus $0$ extension constructed on both of the vertices of $\Gamma_W$.
\end{enumerate}
{\flushleft{(i)}} We extend the graph $\Gamma_U$ as shown in Figure \ref{form2pic}. 
\begin{figure}[ht]
\begin{center}
\psfrag{A}{{\scriptsize $A$}}
\psfrag{A1}{{\scriptsize $A_1$}}
\psfrag{A2}{{\scriptsize $A_2$}}
\psfrag{w1}{{\scriptsize $\xi _1$}}
\psfrag{w2}{{\scriptsize $\xi _2$}}
\includegraphics[scale=0.4]{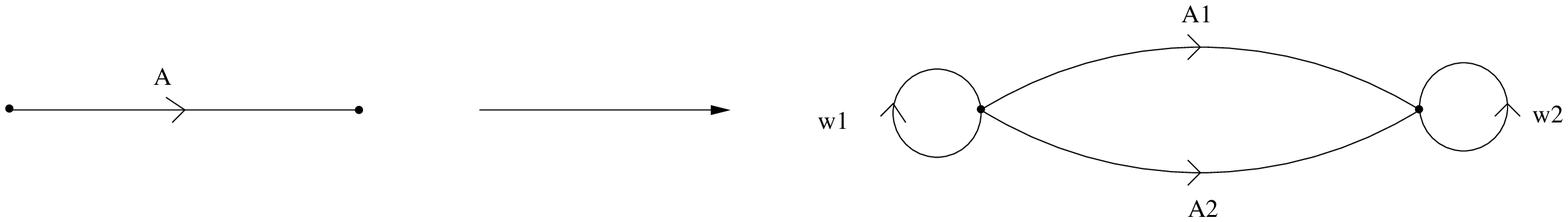}
\\
\vspace{0.38cm}                                                                  
\refstepcounter{figure}\label{form2pic}                                             
Figure \thefigure                  
\end{center}
\end{figure} 
By Theorem \ref{Thexp}, $F$ takes the form of the Hamiltonian cycle
$A_1A_2^{-1}$ in the extended graph and from the nature of the
extension constructed on $U$, it is clear that we have form $2$. 

{\flushleft{(ii)}} We extend the graph $\Gamma_V$ as shown in Figure
\ref{form3pic}.
\begin{figure}[ht]
\begin{center}
\psfrag{A}{{\scriptsize $A$}}
\psfrag{A1}{{\scriptsize $A_1$}}
\psfrag{A2}{{\scriptsize $A_2$}}
\psfrag{B}{{\scriptsize $B$}}
\psfrag{B1}{{\scriptsize $B_1$}}
\psfrag{B2}{{\scriptsize $B_2$}}
\psfrag{w1}{{\tiny $\xi _1$}}
\psfrag{w2}{{\tiny $\xi _2$}}
\psfrag{w3}{{\tiny $\xi _3$}}
\psfrag{w4}{{\tiny $\xi _4$}}
\includegraphics[scale=0.6]{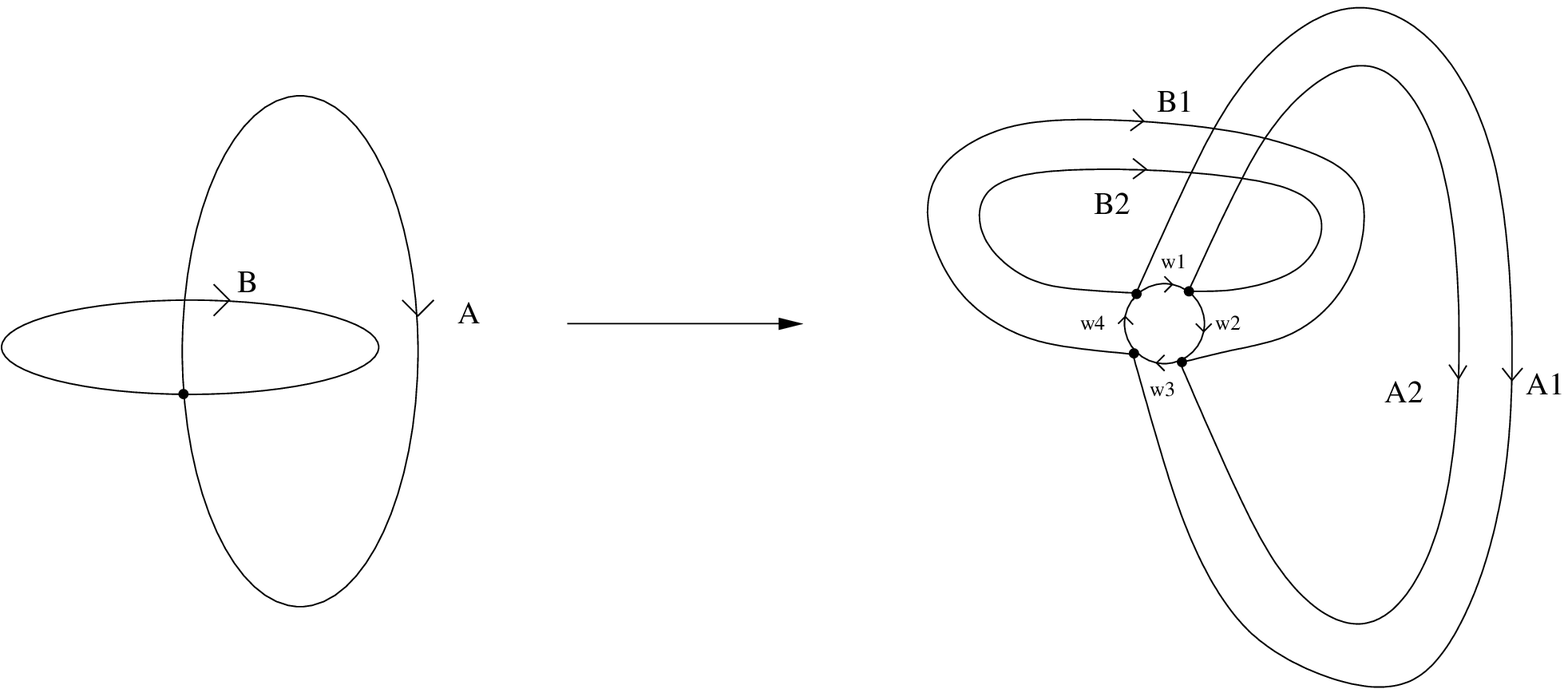}
\\
\vspace{0.38cm}                                                                  
\refstepcounter{figure}\label{form3pic}                                             
Figure \thefigure                  
\end{center}
\end{figure} 
By Theorem \ref{Thexp}, $F$ takes the form of the Hamiltonian cycle
$A_1B_1A_2^{-1}B_2^{-1}$ in the extended graph and from the nature of the
extension constructed on $V$, it is clear that we have form $3$. 

{\flushleft{(iii)}}  Finally, we extend the graph $\Gamma_W$ as shown in Figure
\ref{form4pic}.
\begin{figure}[ht]
\begin{center}
\psfrag{A}{{\scriptsize $A$}}
\psfrag{A1}{{\scriptsize $A_1$}}
\psfrag{A2}{{\scriptsize $A_2$}}
\psfrag{B}{{\scriptsize $B$}}
\psfrag{B1}{{\scriptsize $B_1$}}
\psfrag{B2}{{\scriptsize $B_2$}}
\psfrag{C}{{\scriptsize $C$}}
\psfrag{C1}{{\scriptsize $C_1$}}
\psfrag{C2}{{\scriptsize $C_2$}}
\psfrag{z1}{{\tiny $\rho _1$}}
\psfrag{z2}{{\tiny $\rho _2$}}
\psfrag{z3}{{\tiny $\rho _3$}}
\psfrag{w1}{{\tiny $\xi _1$}}
\psfrag{w2}{{\tiny $\xi _2$}}
\psfrag{w3}{{\tiny $\xi _3$}}
\includegraphics[scale=0.35]{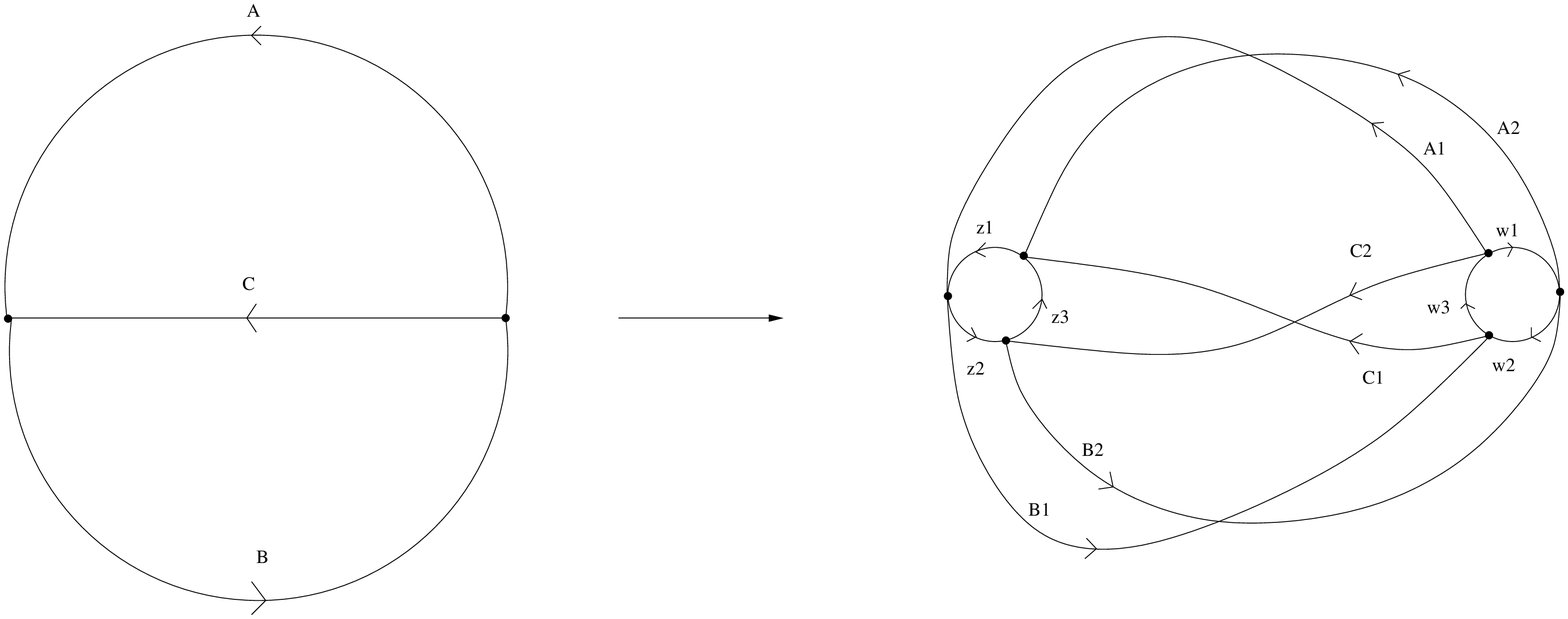}
\\
\vspace{0.38cm}                                                                  
\refstepcounter{figure}\label{form4pic}                                             
Figure \thefigure                  
\end{center}
\end{figure} 
Again, by Theorem \ref{Thexp}, $F$ takes the form of the Hamiltonian cycle
$A_1B_1C_1A_2^{-1}B_2^{-1}C_2^{-1}$ in the extended graph and from the
nature  of the
extension constructed on $W$, it is clear that we have form $4$. 
Hence $h$ is conjugate to some $F$ which takes one of the required forms. 

Let $R$ be the shortest word in $H$ such that $h=_HRFR^{-1}$. 
If $F$ takes form $1$ then, by Lemma  \ref{lb}, it follows that
\begin{eqnarray*}
R &\leq& \frac{1}{2}(|F|+|h|_H)+M+1\\
   &\leq & \frac{1}{2}(6(12l+M+4)+|h|_H)+M+1\\
   &\leq & 36l+4M+13+\frac{|h|_H}{2}.
\end{eqnarray*}
Thus the proposition holds. Therefore, suppose that $F$ is obtained by
an  extension of some orientable word. In the proof of Theorem
\ref{Thexp}, $F$ was constructed from some Wicks form $W$ and a
labelling function $\theta $ such that $\theta(W)$ was minimal over
the set of pairs in $\mathcal{F}$. In
the proof we chose a cyclic permutation $\hat{W}$ of $W$ such that the last
letter was a long edge. Now in the genus $1$ case all Wicks forms take
the form $XYZX^{-1}Y^{-1}Z^{-1}$ where at most one of these letters is
set to $1$,
see \cite{wicks}, and obviously all cyclic permutations are of this
form too. Thus, in this case, we shall 
choose the cyclic permutation $\hat{W}$ which ends in a long letter such that $R$ is the
shortest over all cyclic permutations which end in a long letter. Let
$\hat{W}=ABCA^{-1}B^{-1}C^{-1}$.  It follows from the proof of Theorem
\ref{Thexp} that $F=_H \theta(\hat{W})$. See Figure \ref{caywhat}. 
\begin{figure}[ht]
\begin{center}
\psfrag{A}{{\scriptsize $\theta(A)$}}
\psfrag{B}{{\scriptsize $\theta(B)$}}
\psfrag{C}{{\scriptsize $\theta(C)$}}
\psfrag{F}{{\scriptsize $F$}}
\psfrag{R}{{\scriptsize $R$}}
\psfrag{h}{{\scriptsize $h$}}
\includegraphics[scale=0.4]{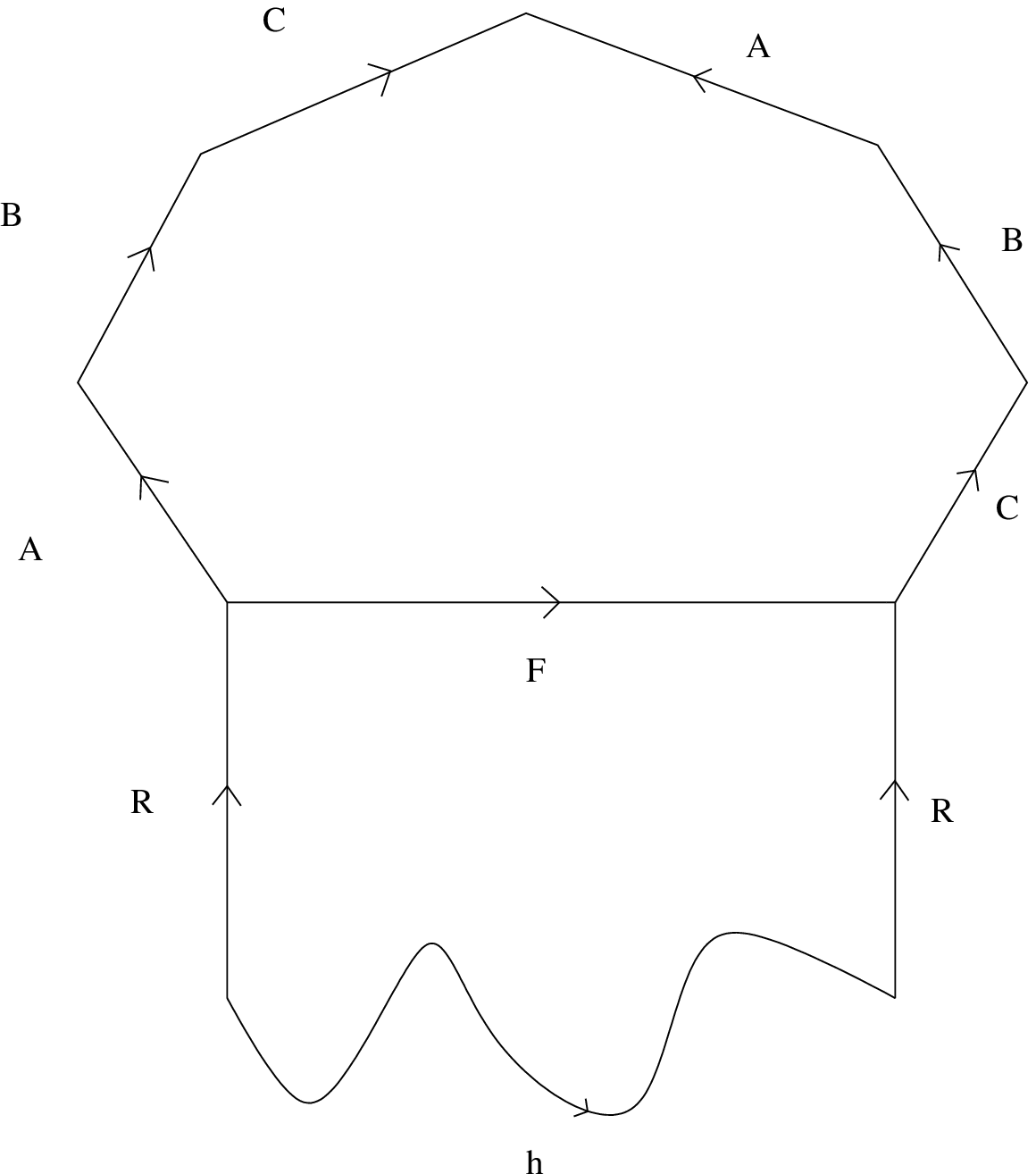}
\\
\vspace{0.38cm}                                                                  
\refstepcounter{figure}\label{caywhat}                                             
Figure \thefigure                  
\end{center}
\end{figure}

We need the following lemma.
\begin{lemma}\label{conj}
Let $u_1$ be the terminal vertex of the label of a long edge of $\hat{W}$ in
$\Gamma_X(H)$.  Suppose that
there exist   a vertex $u_2$ lying on
$F$  such that $d(u_1,u_2)\leq K$, for some
constant $K$. If $d(u_2,\tau (F))-L\leq d(\iota (F),u_2)\leq
d(u_2,\tau (F))+L$ for some constant $L$ then 
\begin{equation*}
|R|\leq \frac{|h|_H}{2}+K+\frac{L}{2}+M+2\delta +1.
\end{equation*}
\end{lemma}
\begin{proof}
Without loss of generality let $u_1=\tau (\theta(A))$. Also we shall let $h'$
be a minimal word such that $h'=_H h$. 
Consider the vertex $u_2$. By Lemma
\ref{la} part $1.$, there exists a vertex $u_3$ on $h'\cup R\cup R^{-1}$
such that $d(u_2,u_3)\leq 2\delta$. First suppose that $u_3$ lies on
either $R$ or $R^{-1}$. Without loss of generality we shall assume
this to be $R$. Now by Lemma \ref{la} part 3 we can see that
$d(\iota (F),u_2)=d(\iota (F),u_3)$. By the
triangle inequality $d(u_1,u_3)\leq K+2\delta$ so there exists a path $s$
in Cayley graph from $u_1$ to $u_3$ of length at most $K+2\delta
$. Let $R=R_1R_2$ such that $u_3=\iota (R_1)=\tau (R_2)$. Cut
and paste along $s$ as shown in Figure \ref{cuts}.
\begin{figure}[ht]
\begin{center}
\psfrag{x}{{\scriptsize $\theta(A)$}}
\psfrag{y}{{\scriptsize $\theta(B)$}}
\psfrag{z}{{\scriptsize $\theta(C)$}}
\psfrag{r1}{{\scriptsize $F_1$}}
\psfrag{r2}{{\scriptsize $F_2$}}
\psfrag{f1}{{\scriptsize $R_1$}}
\psfrag{f2}{{\scriptsize $R_2$}}
\psfrag{s}{{\scriptsize $s$}}
\psfrag{v}{{\scriptsize $h'$}}
\psfrag{u1}{{\tiny $u_1$}}
\psfrag{u2}{{\tiny $u_2$}}
\psfrag{u3}{{\tiny $u_3$}}
\includegraphics[scale=0.4]{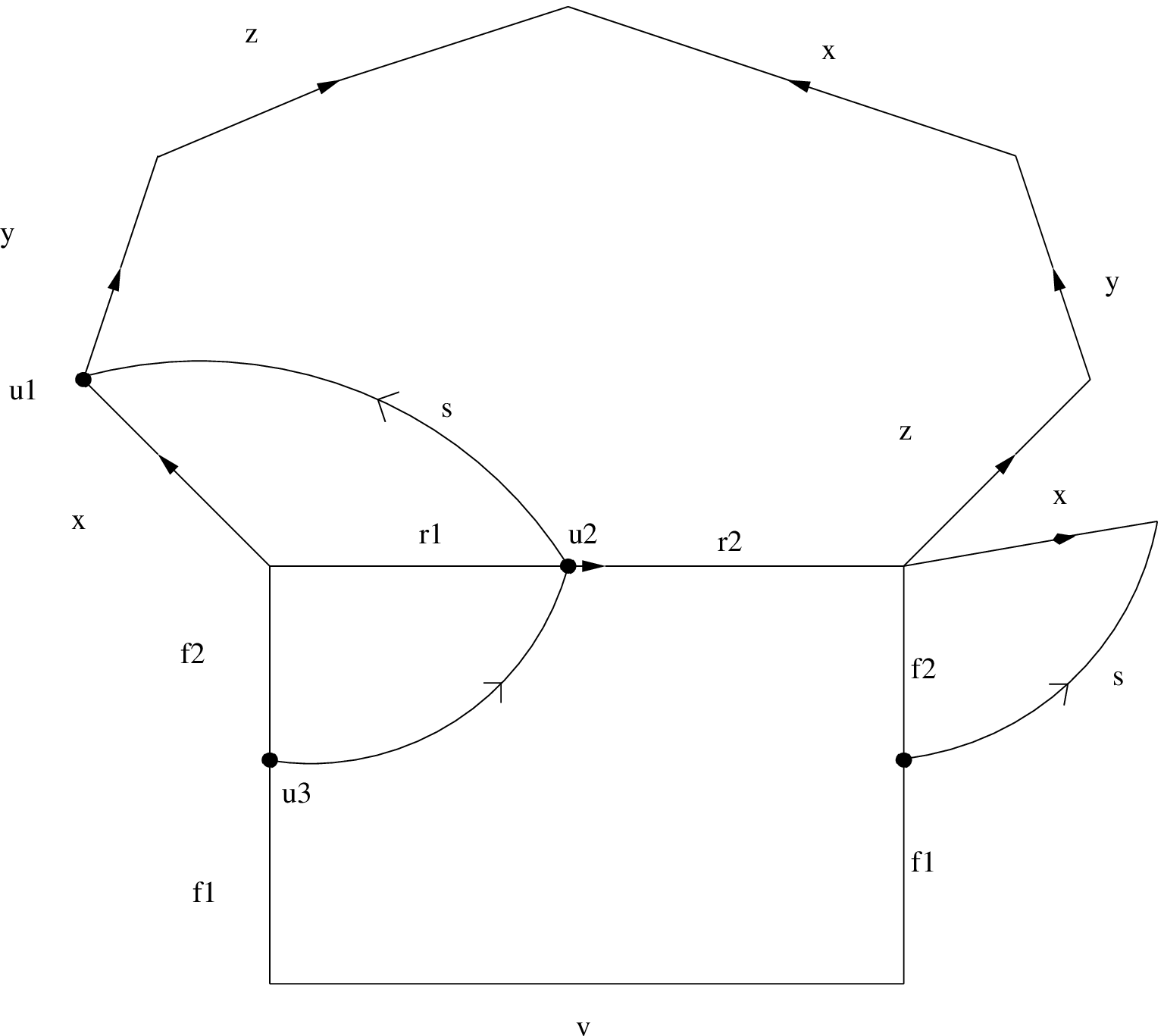}
\caption{Cut and paste along $s$}\label{cuts}
\end{center}
\end{figure}
It is easy to see from the Cayley graph that 
\begin{equation*}
h =_H h'= _H
\underbrace{R_1s}_{R'}\theta(B)\theta(C)\theta(A^{-1})\theta(B^{-1})\theta(C^{-1})
\theta(A)\underbrace{s^{-1}R_1^{-1}}_{(R')^{-1}}.
\end{equation*}
Thus we have a new cyclic permutation of $W$, with the last letter $A$ being a long
edge.
Now, since  $\hat{W}$ was chosen to be the cyclic permutation which
end in a long letter such that $R$ is minimal over all cyclic permutations
which ends in a long letter, it follows that
 $|R|=|R_1|+|R_2|\leq |s|+|R_1|$. This
implies that $d(\iota (F),u_2)=|R_2|\leq |s|\leq K+\delta $. 
 Therefore, by the hypothesis
\begin{eqnarray*}
|F|& = & d(\iota (F),u_2)+d(u_2,\tau (F))\\
   &\leq & 2d(\iota (F),u_2)+L\\
   &\leq & 2K+4\delta +L. 
\end{eqnarray*}  
Thus by Lemma \ref{lb}
\begin{eqnarray*}
|R| &\leq & \frac{1}{2}(|F|+|h'|)+M+1\\
    &\leq & \frac{1}{2}(2K+4\delta +L+|h|_H)+M+1\\
    &\leq & \frac{|h|_H}{2}+K+\frac{L}{2}+M+2\delta +1.
\end{eqnarray*}  
As required.

Now suppose that $u_3$ lies on $h'$. It is easy to see that either
$d(u_3,\iota (h))\leq \frac{|h|_H}{2}$ or $d(u_3,\tau (h))\leq
\frac{|h|_H}{2}$. Without loss of generality we shall assume the
former. Using the triangle in equality it follows  that 
\begin{eqnarray*}
d(u_1,\iota (h)) &\leq & d(u_1,u_2)+d(u_2,u_3)+\frac{|h|_H}{2}\\
                       &\leq & K+2\delta +\frac{|h|_H}{2}.
\end{eqnarray*}
Therefore, there exists a path $t$ of length at most $K+2\delta
+\frac{|h|_H}{2}$ from $\iota
(h)$ to $u_1$. Since $t=_H R\theta(A)$, it is easy to see that
\begin{eqnarray*}
h=_H h' &=_H &
R\theta(A)\theta(B)\theta(C)\theta(A)^{-1}\theta(B)^{-1}\theta(C)^{-1}R^{-1}\\
& =_H &
\underbrace{t}_{R''}\theta(B)\theta(C)\theta(A)^{-1}\theta(B)^{-1}\theta(C)^{-1}
\theta(A)\underbrace{t^{-1}}_{(R'')^{-1}}.
\end{eqnarray*}
Again, from our choice of cyclic permutation of $W$,  we know that
$|R|\leq |t|\leq K+2\delta +\frac{|h|_H}{2}$. Hence the lemma holds.
\end{proof}
Suppose that $F$ has form $2$, that is $F=A_1A_2^{-1}$.  By the
triangle inequality it follows that
\begin{eqnarray*}
|A_2|-|\xi _1|-|\xi _2| & \leq |A_1| \leq &|A_2|+|\xi _1|+|\xi _2|\\
\implies\qquad |A_2|-12(12l+M+4) &\leq  |A_1| \leq &|A_2|+12(12l+M+4).
\end{eqnarray*}
Also, by Lemma \ref{lsc}, $\tau(A_1)$ on $F$ in $\Gamma_X(H)$ is within $5l+M+3$ of
a terminal vertex of some long edge of $\theta(\hat{W})$. Therefore, 
Lemma \ref{conj} implies that 
\begin{eqnarray*}
|R|    &\leq & \frac{|h|_H}{2}+5l+M+3+\frac{12(12l+M+4)}{2}+M+2\delta +1\\
  &\leq &  \frac{|h|_H}{2}+8M+59l+2\delta+28.\\
\end{eqnarray*}
Thus in this case the Proposition  holds. Now if $F$ takes form $3$
or $4$ we follow the same procedure.
For form $3$, that is $F=A_1B_1A_2^{-1}B_2^{-1}$, we have 
\begin{eqnarray*}
\hspace{-2pt}|A_2|\hspace{-1.pt}+\hspace{-1.pt}|B_2|-|\xi _1|-|\xi
_2|-|\xi _3|-|\xi _4| &  \leq |A_1|\hspace{-1pt}+\hspace{-1pt}|B_1|
\leq & \hspace{-6pt}|A_2|\hspace{-1pt}+\hspace{-1pt}|B_2|+ |\xi _1|+|\xi _2|+|\xi _3|+|\xi _4|\\
\implies\, |A_2|\hspace{-1pt}+\hspace{-1pt}|B_2|-12(12l+M+4) &\leq  |A_1|\hspace{-1pt}+\hspace{-1pt}|B_2| \leq
&\hspace{-6pt}|A_2|\hspace{-1pt}+\hspace{-1pt}|B_2|+12(12l+M+4)
\end{eqnarray*}
and we know that  $\tau(B_1)$ on $F$ in
$\Gamma_X(H)$ is within  $5l+M+3$ of
a terminal vertex of some long edge of $\theta(\hat{W})$.
For form $4$, that is $F=A_1B_1C_1A_2^{-1}B_2^{-1}C_2^{-1}$, we have
\begin{align*}
&|A_2|+|B_2|+|C_2|-|\xi _1|-|\xi _2|-|\xi _3|-|\rho _1|-|\rho _2|-|\rho _3| \leq |A_1|+|B_1|+|C_1|\\
\textrm{and}\qquad& |A_1|+|B_1|+|C_1| \leq |A_2|+|B_2|+|C_2| +|\xi
_1|+|\xi _2|+|\xi _3|+|\rho _1+|\rho _2|+|\rho _3|\\
\implies\qquad& |A_2|+|B_2|+|C_2|-12(12l+M+4)\leq |A_1|+|B_1|+|C_1|\\
\textrm{and}\qquad& |A_1|+|B_1|+|C_1| \leq |A_2|+|B_2|+|C_2|+12(12l+M+4)
\end{align*}
and we know that $\tau(C_1)$ on $F$ in
$\Gamma_X(H)$ is within  $5l+M+3$ of
a terminal vertex of some long edge of $\theta(\hat{W})$.
 Thus we may use
Lemma \ref{conj} in both cases to get the required bound for $|R|$. 
Hence the proposition holds.
\end{proof}
Similar lists of forms can of course be constructed for elements of
higher genus. Although, the number of possible extension increases
dramatically with the increase of genus. A.~Vdovina lists the number of
maximal orientable  Wicks forms 
up to genus $15$, see \cite{vdo3}. This  gives an idea of the number
of extensions one  would need to do.


\begin{thebibliography}{99}
\bibitem{alon} J.M.~Alonso, T.~Brady, D.~Cooper, V.~Ferlini,
  M.~Lustig, M.~Mihalik, M.~Shapiro and H.~Short, \emph{Notes on Word
    Hyperbolic Groups} from Group Theory from a Geometric Viewpoint
  edited by E.~Ghys, A.~Haefliger and A.~Verjovsky, World Scientific,
  1990, pp 3-63.
\bibitem{vdo3} R.~Bacher and A.~Vdovina, \emph{Counting 1-vertex Triangulations of
  Oriented Surfaces}, Discrete Math., {\bf{246}}, 2002, pp. 13-27.
\bibitem{com1} J.A.~ Comerford, L.P.~ Comerford, Jr., and C.C.~Edmunds,
  \emph{Powers as Products of Commutators}, Comm. Algebra, {\bf{19}}, no. 2,
  1991, pp. 675-684.
\bibitem{com2} L.P.~Comerford and C.C.~Edmunds, \emph{Products of
    Commutators and Products of Squares in a Free Group}, Inter. Jour.
  of Alg. and Computation, vol. 4, no. 3, 1994, pp. 469-480.
\bibitem{com3} L.P.~Comerford, C.C.~Edmunds and G.~ Rosenberger,
  \emph{Commutators as Powers in Free Products}, Proc. of
  Amer. Math. Soc., vol. 122, no. 1, 1994, pp. 47-52. 
\bibitem{cul} M.~Culler, \emph{Using Surfaces to Solve Equations in Free
groups}, Topology, vol. 20, 1981, pp. 133-145.
\bibitem{friel} K.J.~Friel, \emph{Decision Problems in Hyperbolic
    Groups}, Phd Thesis, Newcastle Upon Tyne, 2000.
\bibitem{grig} R.I.~Grigorchuk and I.G~Lysionok, \emph{A Description
    of Solutions of Quadratic Equations in Hyperbolic Groups},
  Inter. Jour. of Alg. and Computation, vol. 2, no. 3, 1992, pp. 237-274. 
\bibitem{lys} I.G~Lysionok, \emph{On some Algorithmic Properties of
    Hyperbolic Groups}, Izv. Akad. Nauk. SSSR Ser. Math, vol. 53(4),
  1989, English transl. in Math. USSR Izv., vol. 35, pp. 145-163, 1990.
\bibitem{mas} W.S.~Massey, \emph{Algebraic Topology: An introduction},
  Harbrace College Mathematics Series, 1967.
\bibitem{vdo1} A.~Vdovina, \emph{Products of Commuators in Free
    Products}, Inter. Jour. of Alg. and Computation, vol. 7, no. 4,
  1997, pp. 471-485.
\bibitem{vdo2} A.~Vdovina, \emph{On the Number of Nonorientable Wicks
  forms in a free group}, Proc. Royal soc. of Edin. {\bf{126A}}, 1996, pp. 113-116.
\bibitem{wicks} M.J.~Wicks, \emph{Commutators in Free Products},
  J. London Math. Soc. {\bf{37}}, 1962, pp. 433-444.
\end{thebibliography}
\end{document}